\documentclass[10pt,reqno]{amsart}

\usepackage[cp1252]{inputenc}
\usepackage[T1]{fontenc}
\usepackage{amsmath,amsfonts,amssymb,amsthm,cases}
\usepackage{array,dcolumn,booktabs}
\usepackage[colorlinks=true,linkcolor=NavyBlue,citecolor=BrickRed]{hyperref} 

\usepackage[left=2cm,right=2cm,top=2cm,bottom=2cm]{geometry}
\usepackage{enumitem}
\usepackage[dvipsnames,svgnames]{xcolor} 
\usepackage{graphicx} 
\usepackage{stmaryrd} 
\usepackage{dsfont} 
\usepackage{multicol} 
\usepackage{textcomp} 
\usepackage{cite} 
\usepackage{tikz} 
\usepackage{pgfplots} 
\pgfplotsset{compat=1.10}
\usepackage[scr=boondoxo,scrscaled=1.1]{mathalfa} 
\usepackage{float}

\usepackage{etoolbox}
\makeatletter
\patchcmd{\@maketitle}
  {\ifx\@empty\@dedicatory}
  {\ifx\@empty\@date \else {\vskip3ex \centering\footnotesize\@date\par\vskip1ex}\fi
   \ifx\@empty\@dedicatory}
  {}{}
\patchcmd{\@adminfootnotes}
  {\ifx\@empty\@date\else \@footnotetext{\@setdate}\fi}
  {}{}{}
\makeatother

 \author{Olivier Pierre}
\address{Laboratoire de Math\'ematiques Jean Leray, Universit\'e de Nantes, UMR CNRS 6629, \newline
2, rue de la Houssini\`ere, 44322 Nantes Cedex 03, France}
\email{olivier.pierre@univ-nantes.fr}
\date{\today}

\title[Analytic current-vortex sheets in incompressible MHD]
{Analytic current-vortex sheets in incompressible magnetohydrodynamics}

\newtheorem{thm}{Theorem}[section]
\newtheorem{definition}[thm]{Definition}
\newtheorem{prop}[thm]{Proposition}
\newtheorem{lemme}[thm]{Lemma}
\newtheorem{coro}[thm]{Corollary}
\theoremstyle{definition}
\newtheorem{rmq}[thm]{Remark}
\renewenvironment{proof}{\paragraph{\textit{Proof.}}}{\begin{flushright}$\square$\end{flushright}}

\numberwithin{equation}{section}

\newcommand{\bE}{\mathbb{E}}
\newcommand{\bF}{\mathbb{F}}
\newcommand{\bG}{\mathbb{G}}

\newcommand{\bN}{\mathbb{N}}

\newcommand{\bR}{\mathbb{R}}

\newcommand{\bT}{\mathbb{T}}

\newcommand{\bZ}{\mathbb{Z}}

\newcommand{\cC}{\mathcal{C}}

\newcommand{\cF}{\mathcal{F}}
\newcommand{\cG}{\mathcal{G}}

\newcommand{\cT}{\mathcal{T}}

\newcommand{\sT}{\mathscr{T}}

\newcommand{\tildv}{\widetilde{v}}
\newcommand{\tildB}{\widetilde{B}}

\newcommand{\tildA}{\widetilde{A}}

\newcommand{\tildrho}{\widetilde{\rho}}
\newcommand{\tildcF}{\widetilde{\mathcal{F}}}
\newcommand{\tildcG}{\widetilde{\mathcal{G}}}

\newcommand{\uu}{\underline{u}}

\newcommand{\ub}{\underline{b}}
\newcommand{\uf}{\underline{f}}
\newcommand{\uQ}{\underline{Q}}

\newcommand{\uA}{\underline{A}}
\newcommand{\upsi}{\underline{\psi}}

\newcommand{\ucF}{\underline{\mathcal{F}}}
\newcommand{\ucG}{\underline{\mathcal{G}}}
\newcommand{\uJ}{\underline{J}}
\newcommand{\uC}{\underline{C}}

\newcommand{\ds}{\displaystyle}
\newcommand{\sg}{\cdot\nabla}
\newcommand{\dv}{\nabla\cdot}

\newcommand{\p}{\partial}
\newcommand{\bp}{\bar{\p}}

\makeatletter
\newsavebox\myboxA
\newsavebox\myboxB
\newlength\mylenA

\newcommand*\xoverline[2][0.75]{%
    \sbox{\myboxA}{$\m@th#2$}%
    \setbox\myboxB\null
    \ht\myboxB=\ht\myboxA%
    \dp\myboxB=\dp\myboxA%
    \wd\myboxB=#1\wd\myboxA
    \sbox\myboxB{$\m@th\overline{\copy\myboxB}$}
    \setlength\mylenA{\the\wd\myboxA}
    \addtolength\mylenA{-\the\wd\myboxB}%
    \ifdim\wd\myboxB<\wd\myboxA%
       \rlap{\hskip 0.5\mylenA\usebox\myboxB}{\usebox\myboxA}%
    \else
        \hskip -0.5\mylenA\rlap{\usebox\myboxA}{\hskip 0.5\mylenA\usebox\myboxB}%
    \fi}

\newcommand{\nt}[1]{{\big\vert\kern-0.25ex\big\vert\kern-0.25ex\big\vert #1 \big\vert\kern-0.25ex\big\vert\kern-0.25ex\big\vert}}

\newcommand{\sbt}{\ \begin{picture}(-1,1)(-1,-3)\circle*{2}\end{picture}\ \, }

\makeatother

\begin{document}

\begin{abstract}
In this paper, we address the problem of current-vortex sheets in ideal incompressible magnetohydrodynamics. More precisely, we prove a local-in-time existence and uniqueness result for analytic initial data using a Cauchy-Kowalevskaya theorem.
\end{abstract}


\maketitle

\section{Introduction}

\subsection{Motivation}

In this article, we are interested in a free boundary problem arising in magnetohydrodynamics (MHD), namely the current-vortex sheets problem. We consider a homogeneous plasma (the density in constant), assumed to be perfectly conducting, inviscid and incompressible. The equations of ideal incompressible MHD thus read:
\begin{equation}\label{int-equations_MHD}
\left\{
\begin{array}{l}
\p_t u + \dv(u\otimes u - H\otimes H) + \nabla q = 0, \\
\p_t H - \nabla\times(u\times H) = 0, \\
\dv u = \dv H = 0,
\end{array}
\right.
\end{equation}
where $u\in\bR^3$ and $H\in\bR^3$ stand for the velocity and the magnetic field of the plasma respectively. The unknown $q$ defined by $q:= p + \frac{|H|^2}{2}$ is called the ``total pressure'', $p$ being the physical pressure.

We are looking for a special class of (weak) solutions of \eqref{int-equations_MHD}: we want $(u,H,q)$ to be smooth on either side of a hypersurface $\Gamma(t)$ ($t$ is the time-variable), and to give rise to a \textit{tangential} discontinuity across $\Gamma(t)$. We shall assume that the hypersurface $\Gamma(t)$ can be parametrized by $\Gamma(t) := \{ x_3 = f(t,x') \}$, where $x'$ is the tangential variable living in the $2-$dimensional torus $\bT^2 = \raisebox{0.7mm}{$\bR^2$} \big/ \raisebox{-0.7mm}{$\bZ^2$}$, and $x_3$ denotes the normal variable. The unknown $f$ is called the ``front'' of the discontinuity later on. For all $t\in [0,T]$, we shall consider the MHD system \eqref{int-equations_MHD} in the time-dependent domain
\begin{align*}
\Omega(t) := \Omega^+(t) \sqcup \Omega^-(t), \quad \text{ where } \quad \Omega^\pm(t) := \{ x_3 \gtrless f(t,x') \},
\end{align*}
with appropriate boundary conditions on $\Gamma(t)$ :
\begin{align}\label{int-cond_bord}
\p_t f = u^{\pm}\cdot N, \quad H^{\pm}\cdot N = 0, \quad [q] = 0, \quad \forall \, t\in [0,T].
\end{align}
The notation $[q]$ in \eqref{int-cond_bord} stands for the jump of $q$ across $\Gamma(t)$:
\begin{align*}
[q] := \left.q^{+}\right|_{\Gamma(t)} - \left.q^{-}\right|_{\Gamma(t)},
\end{align*}
and $N$ is a normal vector to $\Gamma(t)$, chosen as follows:
\begin{align*}
N := (-\p_1 f, -\p_2 f, 1).
\end{align*}

The boundary conditions \eqref{int-cond_bord} correspond to a tangential discontinuity: the velocity $\p_t f$ of the front is given by the normal component of $u$, the normal magnetic field $H\cdot N$ as well as the total pressure $q$ are continuous across $\Gamma(t)$, and $H\cdot N|_{\Gamma(t)} = 0$ on either side of $\Gamma(t)$. Such boundary conditions model a plasma which does not flow through the hypersurface $\Gamma(t)$, called current-vortex sheet (see Figure \ref{s2-figure_nappe} below). We refer to \cite{BT} for other types of tangential discontinuities in MHD.

For simplicity, we assume that the normal variable $x_3$ belongs to $(-1,1)$. Thus we shall assume that for all $t$ and $x'$ we have $-1<f(t,x')<1$, and impose an additional condition on the ``exterior'' boundaries
\begin{align*}
\Gamma_\pm := \{ (x', \pm 1) \, , \, x'\in\bT^2 \}.
\end{align*}
The system of current-vortex sheets eventually reads as follows:
\begin{equation}\label{int-equations_nappes_MHD}
\left\{
\begin{array}{r l}
\p_t u^\pm \, + \, (u^\pm\sg)u^\pm \, - \, (H^\pm\sg)H^\pm \, + \, \nabla q^\pm \, = \, 0, & \text{ in } \Omega^\pm(t), \quad t\in [0,T], \\[0.5ex]
\p_t H^\pm \, + \, (u^\pm\sg)H^\pm \, - \, (H^\pm\sg)u^\pm \, = \, 0, & \text{ in } \Omega^\pm(t), \quad t\in [0,T], \\[0.5ex]
\dv u^\pm(t) \, = \, \dv H^\pm(t) \, = \, 0, & \text{ in } \Omega^\pm(t), \quad t\in [0,T], \\[0.5ex]
\p_t f \, = \, u^\pm\cdot N, \quad H^\pm\cdot N \, = \, 0, \quad [q] \, = \, 0, & \text{ on } \Gamma(t), \quad t\in [0,T],\\[0.5ex]
u_3^\pm \, = \, H_3^\pm \, = \, 0, & \text{ on } [0,T]\times\Gamma_\pm.
\end{array}
\right.
\end{equation}
The superscript $\,^\pm$ in \eqref{int-equations_nappes_MHD} denotes the unknown $u$, $H$ and $q$ restricted to $\Omega^\pm(t)$. 

In order to solve problem \eqref{int-equations_nappes_MHD}, a common procedure will be to reduce this free boundary problem into the \textbf{fixed} domains $\Omega^+ := \bT^2\times (0,1)$ and $\Omega^- := \bT^2\times (-1,0)$, using a suitable change of variables (see Paragraph \ref{s2-redressement} below). Consequently, the reference domain we shall consider is $\Omega := \bT^2\times (-1,1)$. Besides, the (fixed) boundaries will be noted as follows:
\begin{align}\label{int-def_bords}
\Gamma \, := \, \bT^2\times \{x_3 = 0\} \quad \text{ and } \quad \Gamma_\pm \, := \, \bT^2\times \{x_3=\pm 1\}.
\end{align}
For other free boundary problems, other strategies can be used. For instance, in the water waves theory \cite{Lannes2,Lannes} or for incompressible vortex sheets \cite{SSBF}, the strategy followed by the authors was to reduce the problem on the free surface only. This procedure has been adopted by Sun, Wang and Zhang \cite{SWZ} in their proof of local existence and uniqueness of incompressible current-vortex sheets in the Sobolev regularity scale. Here, the method we advocate is different, since we keep an ``eulerian'' approach considering the equations of \eqref{int-equations_nappes_MHD} in the domains $\Omega^+$ and $\Omega^-$ (up to a diffeomorphism).
 
System \eqref{int-equations_nappes_MHD} is also supplemented with initial data $(u_0^\pm, H_0^\pm, f_0)$ satisfying the constraints:
\begin{equation*}
\left\{
\begin{array}{r l}
\dv u_0^\pm \, = \, \dv H_0^\pm \, = \, 0, & \text{ in } \Omega^\pm(0), \\[0.5ex]
\left.(H^\pm\cdot N)\right|_{t=0} \, = \, 0, & \text{ on } \Gamma(0), \\[0.5ex]
\left. v^+\cdot N \right|_{t=0} \, = \, \left. v^-\cdot N \right|_{t=0}, & \text{ on } \Gamma(0).
\end{array}
\right.
\end{equation*}
\begin{figure}[H]
\begin{center}
\includegraphics[scale=0.25]{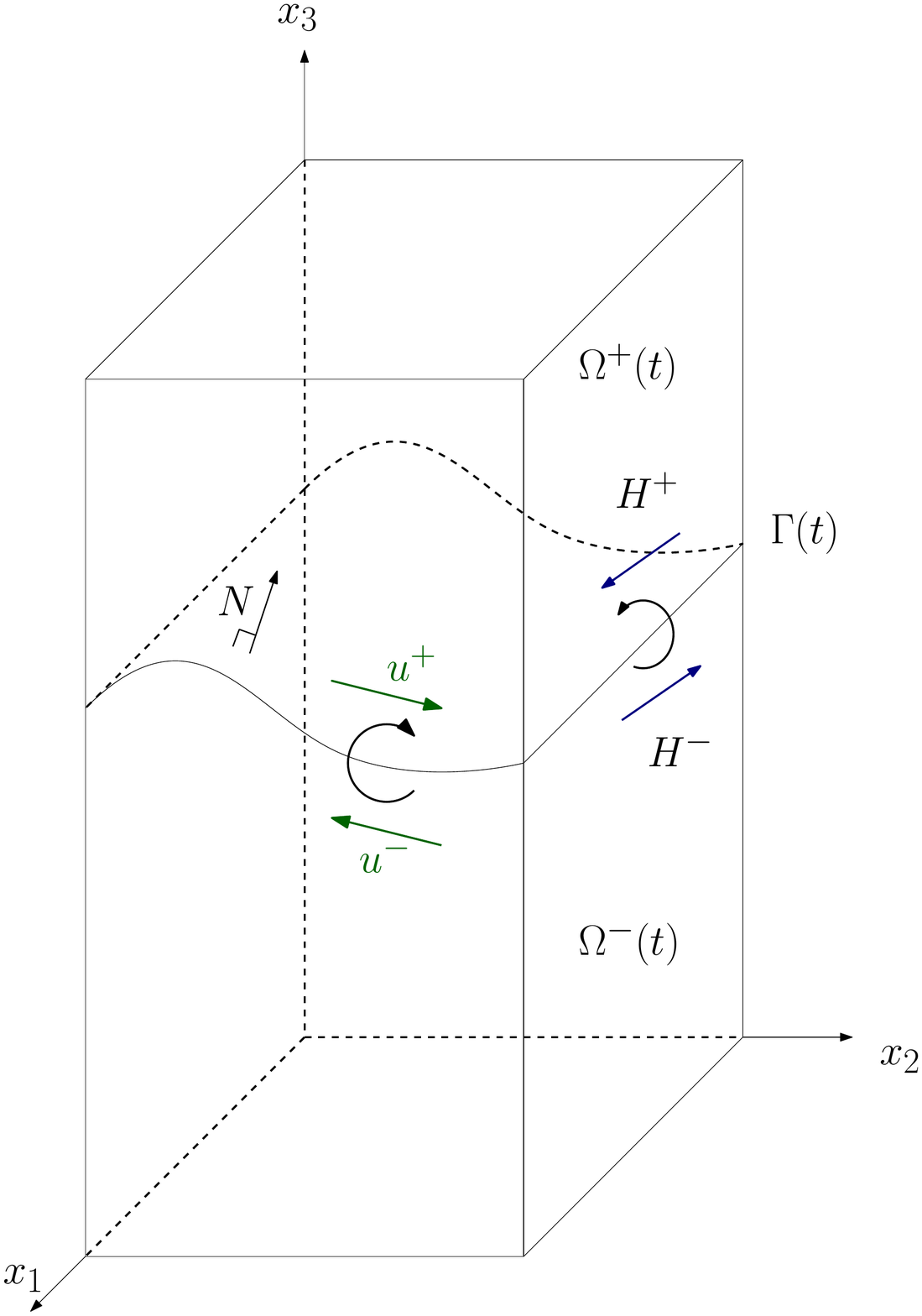}
\caption{Scheme of a current-vortex sheet.}
\label{s2-figure_nappe}
\end{center}
\end{figure}

\subsection{Background}

Current-vortex sheets in ideal incompressible MHD have been a well-known free boundary problem since the 1950's. It has been addressed for instance by Syrovatski\u\i \cite{Syro}, Axford \cite{Axford} or Chandrasekhar \cite{Chandra}. These references in particular deal with the linear stability of \textit{planar} current-vortex sheets using the so-called normal modes analysis. The linear stability criterion derived in these references reads as follows:
\begin{subequations}\label{int-KL_deux_equations}
\begin{align}
& \left| [u] \right|^2 \leq 2 \left( |H^+|^2 + |H^-|^2 \right), \label{int-KL1} \\
& \left| [u]\times H^- \right|^2 + \left| [u]\times H^+ \right|^2 \leq 2 \left| H^+\times H^- \right|^2. \label{int-KL2}
\end{align}
\end{subequations}
In \eqref{int-KL_deux_equations}, $[\cdot]$ denotes the jump across the planar sheet $\{x_3 = 0\}$. Besides, if we assume that $H^+\times H^- \neq 0$ and
\begin{align}
\left| [u]\times H^- \right|^2 + \left| [u]\times H^+ \right|^2 < 2 \left| H^+\times H^- \right|^2, \label{int-KL}
\end{align}
then \eqref{int-KL1} follows from \eqref{int-KL} with a strict inequality. However, if \eqref{int-KL_deux_equations} is not satisfied, then it leads to the so-called \textit{Kelvin-Helmholtz instabilites}. For instance, we refer to \cite{Chandra,Trakhinin}, \cite[p.251-252]{CMST} and references therein for a more detailed discussion about this stability condition. The latter turns out to be a \textit{weak} linear stability condition, \textit{i.e.} the \textit{uniform Kreiss-Lopatinski\u\i\, condition} is not satisfied, leading to \textit{weakly nonlinear surface waves} \cite{AH}. Without the magnetic field, system \eqref{int-equations_nappes_MHD} is reduced to the \textit{vortex sheets problem} for the Euler equations. The tangential discontinuity given by \eqref{int-cond_bord} in this case gives rise to the well-known Kelvin-Helmholtz instabilities in hydrodynamics. This problem is well-posed only in the \textit{analytic} scale. We refer for example to \cite{Chandra,SSBF,Delort,Lebeau} for more details. Nevertheless, when surface tension is involved, the vortex sheets problem is well-posed in the \textit{Sobolev} scale (see e.g. \cite{Ambrose-Masmoudi}).

Under a more restrictive stability condition, namely
\begin{align}\label{int-KL_max}
\max\left( \left| [u]\times H^+ \right| , \left| [u]\times H^- \right| \right) < \left| H^+\times H^-\right|, 
\end{align}
Coulombel, Morando, Secchi and Trebeschi proved in \cite{CMST} an \textit{a priori} estimate without loss of derivatives for the \textit{nonlinear} problem \eqref{int-equations_nappes_MHD}, based on a symmetry argument introduced by Trakhinin \cite{Trakhinin}. Their approach thus gives some hope to state an existence and uniqueness result for system \eqref{int-equations_nappes_MHD}, without using a Nash-Moser iteration.

In that spirit, the aim of this article it to construct \textit{analytic solutions} to \eqref{int-equations_nappes_MHD} using a Cauchy-Kowalevskaya theorem. In other words, we shall use a fixed-point argument in an appropriate \textit{scale of Banach spaces}, adapting the ideas of the proof given by Baouendi and Goulaouic \cite{BG77}. In a future work, we plan to use the main result of \cite{CMST} to exhibit solutions of fixed \textit{Sobolev regularity} to \eqref{int-equations_nappes_MHD} \textit{via} a compactness argument, by approximating Sobolev data by analytic data.

Meanwhile, Sun \textit{et al.} \cite{SWZ} have used a completely different approach to solve the incompressible current-vortex sheets problem. The main idea of their proof is to reduce the whole problem on the free surface $\Gamma(t)$ only, as is common in water waves theory for instance \cite{Lannes}. Then, using some elliptic arguments, they can reconstruct the ``whole'' solution inside both domains $\Omega^\pm(t)$. The advantage of their method is that they can prove a local-in-time existence and uniqueness theorem in the \textit{whole} domain of stability given by \eqref{int-KL}, solving some div-curl systems. Therefore the arguments in \cite{SWZ} require the velocity to be divergence-free. On the opposite, the method we advocate in this paper is based on \textit{a priori} estimates only and could therefore apply to other hyperbolic systems. For instance, as a future work, we might avoid using the Nash-Moser theorem used by Trakhinin \cite{Trakhinin_comp} and by Chen and Wang \cite{Chen-Wang} to prove local-in-time existence and uniqueness of \textit{compressible} current-vortex sheets.

\subsection{Reformulation on a fixed domain, lifting of the front \texorpdfstring{$f$}{f}}\label{s2-redressement}

We begin by recalling Lemma 1 of \cite{CMST}, which shows how to lift the front $f$ defined on $[0,T]\times\Gamma$ into a function $\psi$ defined on $[0,T]\times\Omega$, in order to gain half a derivative in the Sobolev regularity scale. The strategy is inspired from Lannes \cite{Lannes2}.
\begin{lemme}\label{s2-lemme_redressement_CMST}
Let $r\in\bN$ such that $r\geq 2$. Then there exists a continuous linear map
$$ f \in H^{r-\frac{1}{2}}(\bT^2) \mapsto \psi \in H^r(\Omega) $$
such that, for all $x'\in\bT^2$, we have:
\begin{equation}\label{s2-cond_bord_psi}
\psi(x',0) = f(x'), \quad \psi(x',\pm 1) = 0, \quad \p_3 \psi(x',0) = 0.
\end{equation}
\end{lemme}

Lemma \ref{s2-lemme_redressement_CMST} allows to define the map $\Psi : (t,x) \mapsto (x', x_{3} + \psi(t,x))$, where $\psi(t,\cdot)$ is given by the previous lemma applied to the function $f(t,\cdot)$. According to \cite{CMST}, if we impose a smallness condition on the front $f(t,\cdot)$ in the space $H^{\frac{5}{2}}(\bT^2)$, then $\Psi(t,\cdot)$ is a $H^3$-diffeomorphism of $\Omega$. Indeed, the Jacobian $1+\p_3\psi$ of the change of variables satisfies for example $1+\p_3\psi\geq\frac{1}{2}$. We will let $J$ denote the Jacobian of $\Psi$, $A$ the inverse of the Jacobian matrix of $\Psi$ and $a$ the comatrix $JA$ :
\begin{equation}\label{s2-def_J_A_a}
J := 1 + \p_3 \psi, \quad
A := \begin{pmatrix}
1 & 0 & 0 \\
0 & 1 & 0 \\
-\ds\frac{\p_1 \psi}{J} & -\ds\frac{\p_2 \psi}{J} & \ds\frac{1}{J}
\end{pmatrix} \quad \text{ and } \quad 
a := JA = \begin{pmatrix}
J & 0 & 0 \\
0 & J & 0 \\
-\ds\p_1\psi & -\ds\p_2\psi & 1
\end{pmatrix}.
\end{equation}
Let us notice that the comatrix $a$ satisfies the so-called \textit{Piola identities} : the divergence of each column of $a$ vanishes, namely:
\begin{equation}\label{s2-id_piola}
\forall \, j\in\{1,2,3\}, \quad \sum_{i=1}^3 \p_i a_{ij} \, = \, 0,
\end{equation}
where $a_{ij}$ are the coefficients of the matrix $a$. These identities will be useful later on, in order to simplify some expressions.

We also define, for $x$ in the \textit{fixed domains} $\Omega^\pm$, the following new unknowns:
$$ v^{\pm}(t,x) := u^{\pm}(t,\Psi(t,x)), \quad B^{\pm}(t,x) := H^{\pm}(t,\Psi(t,x)), \quad Q^{\pm}(t,x) := q^{\pm}(t,\Psi(t,x)). $$
With these new unknowns, system \eqref{int-equations_nappes_MHD} can be rewritten as follows:
\begin{equation}\label{s2-equations_nappes_MHD_fixe}
\left\{
\begin{array}{r l}
\p_t v^\pm \, + \, (\tildv^\pm\sg)v^\pm \, - \, (\tildB^\pm\sg)B^\pm \, + \, A^T\nabla Q^\pm \, = \, 0, & \text{ in } [0,T]\times\Omega^\pm, \\[0.5ex]
\p_t B^\pm \, + \, (\tildv^\pm\sg)B^\pm \, - \, (\tildB^\pm\sg)v^\pm \, = \, 0, & \text{ in } [0,T]\times\Omega^\pm, \\[0.5ex]
(A^T\nabla)\cdot v^\pm \, = \, (A^T\nabla)\cdot B^\pm \, = \, 0, & \text{ in } [0,T]\times\Omega^\pm, \\[0.5ex]
\p_t f \, = \, v^\pm\cdot N, \quad B^\pm\cdot N \, = \, 0, \quad [Q] \, = \, 0, & \text{ on } [0,T]\times\Gamma, \\[0.5ex]
v_3^\pm \, = \, B_3^\pm \, = \, 0, & \text{ on } [0,T]\times\Gamma_\pm,
\end{array}
\right.
\end{equation}
where we have set:
\begin{equation}\label{s2-def_tildv_tildB}
 N := (-\p_1 \psi, -\p_2 \psi, 1), \quad \tildv^{\pm} := \left( v_{1}^{\pm}, v_{2}^{\pm}, \frac{v^{\pm}\cdot N - \p_t \psi}{J}\right), \quad \tildB^{\pm} := \left( B_{1}^{\pm}, B_{2}^{\pm}, \frac{B^{\pm}\cdot N}{J}\right).
\end{equation}
Here, the vector $N$ is defined on the whole domain $\Omega$, and not only on the interface $\Gamma$. From now on, the notation $[\cdot]$ stands for the jump across the \textbf{fixed} interface $\Gamma$:
$$ [Q] := \left.Q^{+}\right|_{\Gamma} - \left.Q^{-}\right|_{\Gamma}.$$
The total pressure $Q^\pm$ in \eqref{s2-equations_nappes_MHD_fixe} can be implicitly expressed as a function of the unknowns $(v^\pm,B^\pm,f)$, using the ``divergence-free'' constraints on the velocity and the magnetic field. We can show that the couple $(Q^+,Q^-)$ has to satisfy the following elliptic problem (obtained by applying the operator $(A^T\nabla)\cdot$ to the first equation of \eqref{s2-equations_nappes_MHD_fixe}, see \cite[p.266]{CMST} for the full details):
\begin{equation}\label{s2-pb_ellip_pression}
\left\{
\begin{array}{r l}
- \, A_{ji}\p_j (A_{ki}\p_k Q^{\pm}) \, = \, \cF^{\pm}, & \text{ in } [0,T]\times\Omega^\pm, \\[0.5ex]
\, [Q] = 0, & \text{ on } [0,T]\times\Gamma, \\[0.5ex]
\, (1+|\nabla' f|^2)[\p_3 Q] \, = \, \cG, & \text{ on } [0,T]\times\Gamma, \\[0.5ex]
\p_3 Q^{\pm} \, = \, 0, & \text{ on } [0,T]\times\Gamma_\pm,
\end{array}
\right.
\end{equation}
where we have used Einstein's summation convention for repeated indices $i,j,k\in\{1,2,3\}$. The source terms $\cF^\pm$ are defined by:
\begin{align}\label{s2-terme_source_cF^pm}
\cF^\pm :=  -\p_t A_{ki} \, \p_k v_i^\pm + A_{ki} \, \p_k \tildv^\pm\cdot\nabla v_i^\pm - \tildv^\pm\cdot\nabla A_{ki} \, \p_k v_i^\pm - A_{ki} \, \p_k \tildB^\pm\cdot\nabla B_i^\pm + \tildB^\pm\cdot\nabla A_{ki} \, \p_k B_i^\pm, 
\end{align}
and the source term $\cG$ in \eqref{s2-pb_ellip_pression} is defined as follows:
\begin{equation}\label{s2-terme_source_cG}
\cG := - \big[ \, 2v'\cdot\nabla' \p_t f + (v'\cdot\nabla')\nabla' f\cdot v' - (B'\cdot\nabla')\nabla' f\cdot B' \, \big],
\end{equation}
where ``$\, ' \,$'' stands for the tangential parts:
$$ \nabla' := (\p_1, \p_2), \quad v' := (v_1 , v_2 ), \quad B' := (B_1 , B_2 ). $$
In order to solve system \eqref{s2-pb_ellip_pression} (where $t\in [0,T]$ is fixed), the source terms $\cF^\pm$ and $\cG$ need to satisfy the following \textit{necessary compatibility condition} (it suffices to integrate by parts the first equation of \eqref{s2-pb_ellip_pression}):
\begin{equation}\label{s2-cond_nec_comp_cF_cG}
\sum_\pm \int_{\Omega^\pm} J \, \cF^\pm \, dx \, = \, \int_\Gamma \cG \, dx'.
\end{equation}
Using definitions \eqref{s2-terme_source_cF^pm} and \eqref{s2-terme_source_cG}, together with the Piola identities \eqref{s2-id_piola}, we can show after some algebra that the compatibility condition \eqref{s2-cond_nec_comp_cF_cG} can be rewritten as follows:
\begin{equation}\label{s2-cond_nec_comp_cF_cG_2}
\begin{aligned}
\sum_\pm \int_{\Omega^\pm} J \, \Big( \p_t\big( (A^T\nabla)\cdot v^\pm \big) \, & + \,  \tildv^\pm \cdot \nabla \big( (A^T\nabla)\cdot v^\pm \big) \, - \, \tildB \cdot \nabla \big( (A^T\nabla)\cdot B^\pm \big) \Big) \, dx \\
& \, = \, - \, \int_\Gamma \Big[ \big( v\cdot N - \p_t f\big) \, \big( \p_3 v\cdot N\big) \, - \, \big( B\cdot N \big) \, \big( \p_3 B\cdot N \big) \Big] \, dx'.
\end{aligned}
\end{equation}
We can easily see that the compatibility condition \eqref{s2-cond_nec_comp_cF_cG_2} is fulfilled as soon as both the velocity $v^\pm$ and the magnetic field $B^\pm$ are ``divergence-free'', namely $(A^T\nabla)\cdot v^\pm \, = \, (A^T\nabla)\cdot B^\pm \, = \, 0$, and if they satisfy the jump conditions $\p_t f = v^+ \cdot N = v^- \cdot N$ and $B^+ \cdot N = B^- \cdot N = 0$ on the fixed boundary $\Gamma$. Thus, we will need to ensure these \textit{nonlinear} constraints \textit{for all time} $t\in [0,T]$ in order to define the pressure $(Q^+,Q^-)$. 
To overcome this difficulty, we stand out of the works of Coulombel \textit{et al.} \cite{CMST} and Sun \textit{et al.} \cite{SWZ}, giving a new formulation of problem \eqref{s2-equations_nappes_MHD_fixe} inspired by Trakhinin \cite{Trakhinin,Trakhinin_comp}.

\subsection{A new formulation of the problem}\label{s2-sec_nouvelle_formulation}

Let us come back to the formulation \eqref{s2-equations_nappes_MHD_fixe} of the current-vortex sheets problem written in \textit{fixed domains}, whose unknowns are $(v^\pm,B^\pm,f)$ (we omit the pressure $Q^\pm$ since it can be implicitly obtained in function of $(v^\pm,B^\pm,f)$ through the resolution of the Laplace problem \eqref{s2-pb_ellip_pression}). We define new unknowns $(u^\pm,b^\pm,f)$\footnote{Let us be careful, the notation $u^\pm$ no longer refers to the velocity in the original formulation \eqref{int-equations_nappes_MHD}.} as follows:
\begin{equation}\label{s2-def_u_b}
u^\pm \, := \, a \, v^\pm \, = \, (J \, v_1^\pm, J \, v_2^\pm , v^\pm\cdot N) \quad \text{ and } \quad b^\pm \, := \, a \, B^\pm \, = \, (J \, B_1^\pm, J \, B_2^\pm, B^\pm\cdot N),
\end{equation}
where we recall the definition \eqref{s2-def_J_A_a} of the comatrix $a$. We can remark that the terms $\tildv^\pm$ and $\tildB^\pm$ in \eqref{s2-equations_nappes_MHD_fixe} can be easily rewritten in function of $(u^\pm,b^\pm,f)$, since we have:
\begin{align*}
\tildv^\pm \, = \, \frac{u^\pm - \p_t\psi \, e_3}{J}, \quad \tildB^\pm \, = \, \frac{b^\pm}{J},
\end{align*}
where $e_3$ denotes the third vector of the canonical basis of $\bR^3$. With these new unknowns, we will be able to rewrite quite easily both jump conditions on $\Gamma$ and boundary conditions on $\Gamma_\pm$, as well as the ``divergence-free'' constraints in system \eqref{s2-equations_nappes_MHD_fixe}.

\medskip

\noindent
\textbf{The ``divergence-free'' constraints.} Using Piola identities \eqref{s2-id_piola}, we get:
\begin{equation}\label{s2-div_u_div_b}
(A^T\nabla)\cdot v^\pm \, = \, \frac{1}{J} \, \dv u^\pm \quad \text{ and } \quad (A^T\nabla)\cdot B^\pm \, = \, \frac{1}{J} \, \dv b^\pm.
\end{equation}
Thus, we are reduced to the original divergence operator (which is a \textit{linear} and \textit{constant-coefficients} operator). Therefore, the new unknowns $u^\pm$ and $b^\pm$ will have to be divergence-free for all time:
\begin{subequations}\label{s2-contraintes_div_cond_bord_u_b}
\begin{equation}\label{s2-contraintes_div_u_b}
\forall \, t\in [0,T], \quad \dv u^\pm(t) \, = \, \dv b^\pm(t) \, = \, 0 \quad \text{ in } \quad \Omega^\pm. 
\end{equation}

\medskip

\noindent
\textbf{The jump conditions.} We rewrite the jump conditions in \eqref{s2-equations_nappes_MHD_fixe} as follows:
\begin{equation}\label{s2-cond_saut_u_b}
\p_t f \, = \, u_3^+ \, = \, u_3^- \quad \text{ and } \quad b_3^+ \, = \, b_3^- \, = \, 0 \quad \text{ on } \quad \Gamma. 
\end{equation}
In particular, we notice that both conditions on the jump of $u_3$ across $\Gamma$ and on the trace of $b_3^\pm$ on $\Gamma$ are \textit{linear}.

\medskip

\noindent
\textbf{The boundary conditions on $\Gamma_\pm$.} Similarly, using \eqref{s2-cond_bord_psi}, the boundary conditions on $\Gamma_\pm$ in \eqref{s2-equations_nappes_MHD_fixe} can be rewritten as \textit{linear} conditions, namely:
\begin{equation}\label{s2-cond_bord_fixe_u_b}
u_3^+ \, = \, u_3^- \, = \, 0 \quad \text{ and } \quad b_3^+ \, = \, b_3^- \, = \, 0 \quad \text{ on } \quad \Gamma_\pm.
\end{equation}
\end{subequations}

\medskip

\noindent
\textbf{The equations satisfied by $u^\pm$ and $b^\pm$.} Using definition \eqref{s2-def_u_b}, we can deduce the evolution equations satisfied by $u^\pm$ and $b^\pm$ in the domains $\Omega^\pm$. Within this paragraph, we use Einstein's summation convention for repeated indices: we let $i$ denote a tangential index in $\{1,2\}$, and $k$ any index in $\{1,2,3\}$.

Let us begin with the equations satisfied by the velocity $u^\pm$, and let us omit the exponents $\,^\pm$ for the time being. For $i=1,2$, we have:
\begin{subequations}\label{s2-eq_evo_u_b}
\begin{align}\label{s2-eq_evo_u_i}
\p_t u_i \, + \, \frac{u\sg u_i}{J} \, - \, \frac{b\sg b_i}{J} \, - \, \p_3 \Big( \frac{\p_t\psi \, u_i}{J} \Big) \, - \, \big( u_i \, u_k \, - \,  b_i \, b_k \big) \frac{\p_k J}{J^2} \, + \, J \, \p_i Q \, - \, \p_i\psi \, \p_3 Q \, = \, 0,
\end{align}
and for the normal component $u_3$:
\begin{align}\label{s2-eq_evo_u_3}
\p_t u_3 \, + \, \frac{u\sg u_3}{J} \, - \, \frac{b\sg b_3}{J} \, + \, \p_i \Big( \frac{\p_t\psi \, u_i}{J} \Big) \, + \, \big( u_i \, u_k \, - \,  b_i \, b_k \big) \frac{\p_i\p_k\psi}{J^2} \, - \, \p_i\psi \, \p_i Q \, + \, \frac{1 + \p_i\psi\p_i\psi}{J} \, \p_3 Q \, = \, 0.
\end{align}
In the same way, the equations for the magnetic field $b$ read:
\begin{align}
\forall \, i=1,2, \quad & \p_t b_i \, + \, \frac{u\sg b_i}{J} \, - \, \frac{b\sg u_i}{J} \, - \, \p_3 \Big( \frac{\p_t\psi \, b_i}{J} \Big) \, - \, \big( b_i \, u_k \, - \,  u_i \, b_k \big) \frac{\p_k J}{J^2} \, = \, 0, \label{s2-eq_evo_b_i} \\[1ex]
& \p_t b_3 \, + \, \frac{u\sg b_3}{J} \, - \, \frac{b\sg u_3}{J} \, + \, \p_i \Big( \frac{\p_t\psi \, b_i}{J} \Big) \, + \, \big( b_i \, u_3 \, - \,  u_i \, b_3 \big) \frac{\p_i J}{J^2} \, = \, 0. \label{s2-eq_evo_b_3}
\end{align}
\end{subequations}
System \eqref{s2-eq_evo_u_b} will be supplemented with an initial datum $(u_0^\pm,b_0^\pm)$ satisfying the constraints \eqref{s2-contraintes_div_cond_bord_u_b}, together with an initial front $f_0$.

\medskip

\noindent
\textbf{The Laplace problem satisfied by the pressure $Q^\pm$.} The elliptic problem associated with the pressure is the same as \eqref{s2-pb_ellip_pression}; we just rewrite it in terms of the new unknowns $u^\pm$ and $b^\pm$, and we use Piola identities \eqref{s2-id_piola} to make the classical divergence operator appear. We thus have:
\begin{equation}\label{s2-pb_ellip_pression_réécrit}
\left\{
\begin{array}{r l}
- \, \dv \big( a \, A^T \, \nabla Q^\pm \big) \, = \, J \, \cF^\pm, & \text{ in } [0,T]\times\Omega^\pm, \\[0.5ex]
\, [Q] \, = \, 0, & \text{ on } [0,T]\times\Gamma, \\[0.5ex]
\big( 1 + |\nabla' f|^2 \big) \, \big[\p_3 Q] \, = \, \cG, & \text{ on } [0,T]\times\Gamma, \\[0.5ex]
\p_3 Q^\pm \, = \, 0, & \text{ on } [0,T]\times\Gamma_\pm.
\end{array}
\right.
\end{equation}
The source terms $J\,\cF^\pm$ and $\cG$ in \eqref{s2-pb_ellip_pression_réécrit}, defined by \eqref{s2-terme_source_cF^pm}, \eqref{s2-terme_source_cG}, can be expressed in function of $u^\pm$ and $b^\pm$ as follows (we omit the exponents $\,^\pm$ for the sake of clarity):
\begin{subequations}\label{s2-termes_sources_cF_cG_réécrits}
\begin{equation}\label{s2-terme_source_JcF_réécrit}
\begin{aligned}
J \, \cF & \, = \, \frac{1}{J} \big( \p_k u_\ell \, \p_\ell u_k \, - \, \p_k b_\ell \, \p_\ell b_k \big) \, - \, \frac{2 \, \p_k J}{J^2} \big( u_i \, \p_i u_k \, - \, b_i \, \p_i b_k \big) \, + \, \frac{2 \, \p_i J \, \p_k J}{J^3} \big( u_i \, u_k \, - \, b_i \, b_k \big) \\
& \hspace*{0.5cm} + \, \frac{2 \, \p_i \p_j \psi}{J^2} \big( u_i \, \p_3 u_j \, - \, b_i \, \p_3 b_j \big) \, - \, \frac{2 \, \p_j J}{J^2} \big( u_i \, \p_j u_j \, - \, b_i \, \p_j b_j \big).
\end{aligned}
\end{equation}
In \eqref{s2-terme_source_JcF_réécrit}, we have used Einstein's summation convention for repeated indices, where $i, j$ stand for \textit{tangential} indices belonging to $\{1,2\}$ and $k, \ell$ denote any indices belonging to $\{1,2,3\}$. Finally, the source term $\cG$ reads:
\begin{equation}\label{s2-terme_source_cG_réécrit}
\cG = - \big[ \, 2 \, u'\cdot\nabla' \p_t f + (u'\cdot\nabla')\nabla' f\cdot u' - (b'\cdot\nabla')\nabla' f\cdot b' \, \big].
\end{equation}
\end{subequations}
Using the constraints \eqref{s2-contraintes_div_cond_bord_u_b}, we can easily show that the source terms $J\,\cF^\pm$ and $\cG$ in \eqref{s2-pb_ellip_pression_réécrit} are compatible, as stated below.
\begin{lemme}
Assume that $(u^\pm,b^\pm,f)$ satisfy the constraints \eqref{s2-contraintes_div_u_b}, \eqref{s2-cond_saut_u_b}, \eqref{s2-cond_bord_fixe_u_b} and are sufficiently smooth. Then the compatibility condition \eqref{s2-cond_nec_comp_cF_cG} is fulfilled by the source terms $J \, \cF^\pm$ and $\cG$ defined by \eqref{s2-termes_sources_cF_cG_réécrits}.
\end{lemme}

From now on, we focus on system \eqref{s2-eq_evo_u_b}, \eqref{s2-pb_ellip_pression_réécrit}, \eqref{s2-termes_sources_cF_cG_réécrits} together with the divergence-free constraints and the boundary conditions \eqref{s2-contraintes_div_cond_bord_u_b}. Given an initial datum $(u_0^\pm,b_0^\pm,f_0)$, the purpose is to construct an analytic solution $(u^\pm,b^\pm,f)$ to this system, using a fixed-point method of Cauchy-Kowalevskaya type. The functional framework will be defined in the next paragraph.

The main advantage of this formulation is the linearity of the constraints \eqref{s2-contraintes_div_cond_bord_u_b}. Let us recall that we need these constraints to be satisfied \textit{for all time}, in order to solve the Laplace problem \eqref{s2-pb_ellip_pression_réécrit}. It turns out that the ideas of Baouendi and Goulaouic \cite{BG77} seem to be appropriate to manage the constraints \eqref{s2-contraintes_div_cond_bord_u_b}. Indeed, we will be able to include them in the functional spaces defined below.

\subsection{Scales of analytic Banach spaces}\label{s0-sec_analytic_spaces}

First of all, let us recall the basic properties of Sobolev spaces that we shall use later on in order to define some analytic spaces. We refer for instance to \cite{Benzoni-Serre,Brezis,Evans,Triebel,Zuily}.

\subsubsection*{Periodic Sobolev spaces}

Let us denote the $1-$periodic torus by $\ds \bT := \raisebox{0.7mm}{$\bR$} \big/ \raisebox{-0.7mm}{$\bZ$}$. For $s\in\bR_+$, we define the $d-$dimensional periodic Sobolev space by (see \textit{e.g.} \cite{Triebel}):
\begin{align}\label{s0-def_esp_sob_per}
H^s(\bT^d) := \Big\{ u\in L^2(\bT^d) \, \Big| \, \sum_{n\in\bZ^d} (1 + |n|^2)^s \, |c_n(u)|^2 \, < \, +\infty \Big\},
\end{align}
equipped with the norm $\|\cdot\|_{H^s(\bT^d)}$ defined by
\begin{align}\label{s0-def_norm_esp_sob_T^d}
\|u\|_{H^s(\bT^d)}^2 := \sum_{n\in\bZ^d} (1 + |n|^2)^s \, |c_n(u)|^2.
\end{align}
For $n\in\bZ^d$, the quantity $c_n(u)$ corresponds to the $n^{th}$ Fourier coefficient of $u$. Therefore, we have a convenient characterization of the spaces $H^s(\bT^d)$ by means of the (discrete) Fourier transform. The Sobolev space $H^s(\bT^d)$ is a Hilbert space.

\subsubsection*{Sobolev spaces on a bounded domain of \texorpdfstring{$\bR^3$}{Lg}}

Let us now consider $\Omega$, a bounded and smooth domain of $\bR^3$. For all $r\in\bN$, the Sobolev space $H^r(\Omega)$ is defined by:
\begin{align}\label{s0-def_esp_sob_Omega}
H^r(\Omega) := \left\{ u\in L^2(\Omega) \, | \, \forall \, \alpha\in\bN^3, \, |\alpha|\leq r, \, \p^\alpha u\in L^2(\Omega) \right\},
\end{align}
equipped with the norm $\|\cdot\|_{H^r(\Omega)}$ defined by
\begin{align}\label{s0-def_norm_esp_sob_Omega}
\| u \|_{H^r(\Omega)}^2 := \sum_{|\alpha|\leq r} \|\p^\alpha u\|_{L^2(\Omega)}^2.
\end{align}
The space $H^r(\Omega)$ is a Hilbert space.

\subsubsection*{A particular case: the Sobolev space \texorpdfstring{$H^r(\bT^2\times (0,1))$}{$Lg$}}

The equations of magnetohydrodynamics we study are set on a domain of the form $\bT^2\times (0,1)$, \textit{i.e.} a horizontal strip $\{0 < x_3 < 1\}$, with periodic boundary conditions with respect to the tangential variable $x' := (x_1,x_2)\in\bT^2$.
%
%
%
%
%
Using Fubini's theorem, we can write, for $u\in H^r(\bT^2\times (0,1))$,
\begin{align}\label{s0-norm_H^r_T^2_(0,1)}
\|u\|_{H^r(\bT^2\times (0,1))}^2 = \sum_{|\alpha|\leq r} \left\| \p^\alpha u\right\|_{L^2(\bT^2\times (0,1))}^2 
= \sum_{k=0}^r \left\| \p_3^k u \right\|_{L_{x_3}^2(H^{r-k}(\bT^2))}^2.
\end{align}
As a consequence, the space $\ds H^r(\bT^2\times (0,1))$ also takes the form
\begin{align*}
H^r(\bT^2\times (0,1)) = \bigcap_{k=0}^r H_{x_3}^k \left( H^{r-k}(\bT^2) \right).
\end{align*}
From a computational point a view, the norm given by \eqref{s0-norm_H^r_T^2_(0,1)} turns out to be useful. Indeed, for $k\in\{0,\dots,r\}$, we shall compute the quantities $\left\| \p_3^k u \right\|_{L_{x_3}^2(H^{r-k}(\bT^2))}$ thanks to Fourier series in $x'$ and integration with respect to $x_3$.

We eventually give a useful tame estimate for composite functions, in a very particular case (see \textit{e.g.} \cite[p.283]{Lannes}).

\begin{prop}\label{s0-prop_estim_1/1+u_H^r}
Let $r\in\bN$ and $\Omega$ be a smooth bounded domain of $\bR^d$. We consider $u\in H^r(\Omega)$ such that:

\begin{align*}
\exists \, c_0>0,  \quad \forall \, x\in\Omega, \quad c_0^{-1} \leq 1 + u(x) \leq c_0. 
\end{align*}
Then the function $\frac{1}{1+u}$ belongs to $H^r(\Omega)$ and satisfies the estimate:

\begin{align}\label{s0-estim_1/1+u_H^r}
\left\|\frac{1}{1+u}\right\|_{H^r(\Omega)} \, \leq \, C_{r,c_0} \, \big( 1 + \|u\|_{H^r(\Omega)} \big),
\end{align}
where $C_{r,c_0} >0$ depends only on $r$, $c_0$ and $\Omega$.
\end{prop}
The fact that we can estimate the inverse of a given function turns out the be useful when, for instance, we perform a change of variables. More precisely, some computations of derivatives can make the Jacobian of this change of variables appear, as well as \textbf{its inverse}.

Now that we have recalled the main properties of Sobolev spaces, we are able to construct some spaces of analytic functions. Their definition relies on the spaces $H^r(\Omega)$.

\subsubsection*{Analytic spaces on \texorpdfstring{$\bT^2\times (-1,1)$}{T2x(-1,1)}}\label{s0-sec_esp_ana_T^2_(-1,1)}

From now on, $\Omega$ denotes the reference domain $\bT^2\times (-1,1)$. The norm of the Sobolev space $H^r(\Omega)$, for $r\in\bN$, is simply noted as $\|\cdot\|_{H^{r}}$. We also consider a parameter $\rho_0>0$, fixed once and for all. The following definitions and properties of the functional spaces below are inspired from \cite{Sedenko}.

\begin{definition}
Let $(r,k)\in\bN^{2}$ and $\rho\in (0,\rho_{0}]$. We define the space $B_{\rho,r}^{k}$ by
\begin{align*}
B_{\rho,r}^{k} := \Big\{ u\in H_{x_{3}}^{k+r}(H^\infty(\bT^{2})) \, \Big| \, \sum_{n\geq 0} \frac{\rho^{n}}{n!}\max_{\substack{|\alpha|=n \\ \alpha_{3}\leq k}} \|\p^{\alpha}u\|_{H^{r}} < +\infty \Big\},
\end{align*}
where $\alpha$ will always denote a multi-index of $\bN^3$ and $|\alpha|$ its length. We equip this space with the norm $\|\cdot\|_{\rho,r}^{k}$ given by
\begin{align*}
\| u \|_{\rho,r}^{k} := \sum_{n\geq 0} \frac{\rho^{n}}{n!}\max_{{\substack{|\alpha|=n \\ \alpha_{3}\leq k}}} \|\p^{\alpha}u\|_{H^{r}}.
\end{align*}
\end{definition}

\begin{prop}
Given $(r,k)\in\bN^2$, the sequence $(B_{\rho,r}^{k})_{0<\rho\leq\rho_{0}}$ is a scale of Banach spaces, in other words:
\begin{align*}
\forall \, 0<\rho'\leq\rho\leq\rho_{0}, \quad B_{\rho',r}^{k} \supset B_{\rho,r}^{k}, \quad \text{with} \quad \|\cdot\|_{\rho',r}^{k}\leq \|\cdot\|_{\rho,r}^{k}.
\end{align*}
\end{prop}
We refer to \cite{Sedenko} for the proof. In order to apply a Cauchy-Kowalevskaya theorem, we will need algebra and differentiation properties in the spaces $B_{\rho,r}^k$. Both following propositions give such properties.
\begin{prop}\label{s1-prop_prod_B_rho,r^k}
Let $r\in\bN$ such that $r\geq 2$ and $u,v\in B_{\rho,r}^{k}$. Then $u\,v\in B_{\rho,r}^{k}$ and we have the following algebra properties:
\begin{align}\label{s1-norme_alg_rho,r,k}
\|u\,v\|_{\rho,r}^k \, \leq \, C \, \|u\|_{\rho,r}^k \, \|v\|_{\rho,r}^k,
\end{align}
\begin{align}\label{s1-norme_alg_rho,r,k_convolution}
\|u\,v\|_{\rho,r}^k \, \leq \, C \, \sum_{j=0}^{k} \|u\|_{\rho,r}^j \, \|v\|_{\rho,r}^{k-j}, 
\end{align}
where $C>0$ depends only on $r$.
\end{prop}
Estimate \eqref{s1-norme_alg_rho,r,k} gives the algebra property of the spaces $B_{\rho,r}^k$, whereas inequality \eqref{s1-norme_alg_rho,r,k_convolution} will turn out to be more convenient in what follows, because of the appearing Cauchy product. Once again, we refer to \cite{Sedenko} for the details which can be directly adapted to the framework of Sobolev spaces we consider here. The next proposition, also adapted from \cite{Sedenko}, shows how the differentiation behaves in the scale $(B_{\rho,r}^{k})_{0<\rho\leq\rho_{0}}$.
\begin{prop}\label{s1-prop_deriv_B_rho,r^k}
Let $u\in B_{\rho,r}^{k}$. Then for all $\rho'<\rho$, $u$ satisfies $\p_{1}u, \p_{2}u \in B_{\rho',r}^{k}$ with
\begin{equation}\label{s1-deriv_tan_rho,r,k}
\|\p_{j}u\|_{\rho',r}^{k} \leq \frac{1}{\rho-\rho'}\|u\|_{\rho,r}^{k}, \quad j=1,2.
\end{equation}
Besides, if $k\geq 1$, then $\p_{3}u \in B_{\rho',r}^{k-1}$ with
\begin{equation}\label{s1-deriv_normale_rho,r,k}
\|\p_{3}u\|_{\rho',r}^{k-1} \leq \frac{1}{\rho-\rho'}\|u\|_{\rho,r}^{k}.
\end{equation}
We can also express the loss of derivative thanks to the index $r$. If $r\geq 1$, then for $j=1,2,3$, we have $\p_{j}u\in B_{\rho,r-1}^{k}$ and the estimate
\begin{equation}\label{s1-deriv_rho,r,k}
\|\p_{j}u\|_{\rho,r-1}^{k} \leq \|u\|_{\rho,r}^{k}.
\end{equation}
\end{prop}
To take all the normal derivatives $\p_{3}^{\alpha_{3}}$ (with $\alpha_{3}\in\bN$) into account, we now introduce a new scale of Banach spaces. Their construction is based on the previous Banach spaces $B_{\rho,r}^k$.
\begin{definition}
Let $r\in\bN$, $\sigma > 0$ (a small parameter to fix later on) and $\rho\in (0,\rho_{0}]$. We define the space $B_{\rho,r,\sigma}$ by
\begin{align*}
B_{\rho,r,\sigma} := \Big\{ u\in H^\infty(\Omega) \, \Big| \, \sum_{k\geq 0} \sigma^{k} \|u\|_{\rho,r}^{k} < +\infty \Big\}.
\end{align*}
We equip $B_{\rho,r,\sigma}$ with the norm $\|\cdot\|_{\rho,r,\sigma}$ given by
\begin{align*}
\| u \|_{\rho,r,\sigma} := \sum_{k\geq 0} \sigma^{k} \|u\|_{\rho,r}^{k}.
\end{align*}
\end{definition}

Sedenko \cite{Sedenko} applies a Cauchy-Kowalevskaya theorem in a similar scale of Banach spaces, for solving the Euler equations for ideal incompressible nonhomogeneous and barotropic fluids, bounded by free surfaces. Applying the Cauchy-Kowalevskaya theorem relies on the following properties for the scale $(B_{\rho,r,\sigma})_{0<\rho\leq\rho_0}$.
\begin{thm}\label{s1-thm_echelle_B_rho,r,sigma}
Let $r\in\bN$ such that $r\geq 2$ and $\sigma \in (0,1]$. Then the sequence $(B_{\rho,r,\sigma})_{0<\rho\leq\rho_{0}}$ is a scale of Banach spaces:
$$ \forall \, 0<\rho'\leq\rho\leq\rho_{0}, \quad B_{\rho',r,\sigma} \supset B_{\rho,r,\sigma}, \quad \text{with} \quad \|\cdot\|_{\rho',r,\sigma}\leq \|\cdot\|_{\rho,r,\sigma}. $$
Moreover, if $u\in B_{\rho,r,\sigma}$ then for all $\rho'<\rho$, $\p_{i}u \in B_{\rho',r,\sigma}$ $(1\leq i\leq 3)$ with
\begin{equation}\label{s1-deriv_rho,r,sigma}
\|\p_{i} u\|_{\rho',r,\sigma} \leq \frac{C_\sigma}{\rho-\rho'}\|u\|_{\rho,r,\sigma},
\end{equation}
where $C_\sigma := \sigma^{-1}>0$. To finish with, the spaces $B_{\rho,r,\sigma}$ are algebras, more precisely:
\begin{equation}\label{s1-prod_rho,r,sigma}
\forall \, u,v\in B_{\rho,r,\sigma}, \quad u\,v\in B_{\rho,r,\sigma} \quad \text{ and } \quad \| u\,v \|_{\rho,r,\sigma} \leq C_r \|u\|_{\rho,r,\sigma}\|v\|_{\rho,r,\sigma},
\end{equation}
where $C_r>0$ does not depend on $\rho$ and $\sigma$.
\end{thm}

\begin{rmq}\label{art-rmq_density}
The spaces $B_{\rho,r,\sigma}$ turn out to be dense in $H^r(\Omega)$, which could be useful to handle initial data of \textit{Sobolev} regularity in a future work.
\end{rmq}

\medskip

The introduction of such Banach spaces will allow us to construct analytic solutions to the equations of the current-vortex sheets problem in ideal incompressible MHD, using a Cauchy-Kowalevskaya theorem. Such a problem is represented by the \textit{coupled} system of equations \eqref{s2-eq_evo_u_b}, \eqref{s2-contraintes_div_cond_bord_u_b}. Consequently, we similarly define analytic spaces on the sub-domains $\Omega^+ = \bT^2\times (0,1)$ and $\Omega^- = \bT^2\times (-1,0)$, respectively denoted as $B_{\rho,r,\sigma}(\Omega^+)$ and $B_{\rho,r,\sigma}(\Omega^-)$. The norm on $B_{\rho,r,\sigma}(\Omega^+)$ is denoted by $\|\cdot\|_{\rho,r,\sigma}^+$ and so on.

For $(u^{+},u^{-})\in H^{r}(\Omega^{+})\times H^{r}(\Omega^{-})$, we will denote the norms as follows:
\begin{align}
& \|u^{+}\|_{r,+} := \|u^{+}\|_{H^{r}(\Omega^{+})}, \quad \|u^{-}\|_{r,-} := \|u^{-}\|_{H^{r}(\Omega^{-})}, \nonumber \\
& \|u^{\pm}\|_{r,\pm} := \|u^{+}\|_{r,+} + \|u^{-}\|_{r,-}. \label{s2-norme_pm}
\end{align}
Furthermore, if $(u^{+},u^{-})\in B_{\rho,r,\sigma}(\Omega^+)\times B_{\rho,r,\sigma}(\Omega^-)$, we will write:
\begin{align}\label{s2-norme_pm_B_rho,r,sigma}
\|u^\pm\|_{\rho,r,\sigma}^\pm \, := \, \|u^+\|_{\rho,r,\sigma}^+ \, + \, \|u^-\|_{\rho,r,\sigma}^-.
\end{align}
The algebra and differentiation properties previously established still hold in the Banach spaces $B_{\rho,r}^{k}(U)$ and $B_{\rho,r,\sigma}(U)$ introduced above, where $U$ stands for the domains $\Omega$, $\Omega^{+}$ or $\Omega^{-}$.

\subsubsection*{Analytic spaces on \texorpdfstring{$\Gamma\simeq\bT^{2}$}{Gamma}}

\begin{definition}\label{s2-def_esp_ana_T^2}
Let $s\in\bR_{+}$ and $\rho\in (0,\rho_{0}]$. We define
$$ B_{\rho,s}(\bT^{2}) := \Big\{ u\in H^\infty(\bT^{2}) \,\, \Big| \,\, \| u \|_{\rho,s} := \sum_{n\geq 0} \frac{\rho^{n}}{n!}\max_{|\alpha'|=n} \|\partial^{\alpha'}u\|_{H^{s}(\bT^{2})} < +\infty \Big\}, $$
where $\alpha'$ is here a multi-index of $\bN^2$. We recall that the norm on $H^s(\bT^2)$ is defined by \eqref{s0-def_norm_esp_sob_T^d}.
\end{definition}
With such scales of Banach spaces, we shall be able in Section \ref{s2-sec_CK_thm} to construct analytic solutions to \eqref{s2-eq_evo_u_b}, \eqref{s2-contraintes_div_cond_bord_u_b}.

\subsubsection*{A trace estimate}

We end up with a trace estimate from the spaces $B_{\rho,r,\sigma}(\Omega^\pm)$ towards $B_{\rho,r-\frac{1}{2}}(\bT^2)$.
\begin{prop}
Let $\rho\in (0,\rho_0]$, $r\in\bN^*$ and $\sigma\in (0,1]$. Then for all $u\in B_{\rho,r,\sigma}(\Omega^+)$, the trace of $u$ on $\Gamma\simeq\bT^2$ satisfies the following properties:
\begin{align}\label{art-estim_trace}
u|_\Gamma \in B_{\rho,r-\frac{1}{2}}(\bT^2) \quad \text{ and } \quad \| u|_\Gamma \|_{\rho,r-\frac{1}{2}} \, \leq \, C_r \, \|u\|_{\rho,r,\sigma}^+, 
\end{align}
where $C_r > 0$ depends only on $r$. Obviously, estimate \eqref{art-estim_trace} also holds if $u\in B_{\rho,r,\sigma}(\Omega^-)$ with $\|\cdot\|_{\rho,r,\sigma}^-$ instead of $\|\cdot\|_{\rho,r,\sigma}^+$.
\end{prop}
\begin{proof}
Let $u\in B_{\rho,r,\sigma}(\Omega^+)$, $n\in\bN$ and $\alpha'\in\bN^2$ such that $|\alpha'|=n$. Writing $\p^{\alpha'} \big( u|_\Gamma \big) \, = \, \big( \p^{\alpha'} u \big)|_\Gamma$ and using the classical trace lemma \cite{Adams}, we have:
\begin{align*}
\| \p^{\alpha'} \big( u|_\Gamma \big) \|_{H^{r-\frac{1}{2}}(\bT^2)} 
\, \leq \, C_r \, \| \p^{\alpha'} u \|_{H^r(\Omega^+)} 
\, \leq \, C_r \, \max_{\substack{|\beta| = n \\ \beta_3 = 0}} \| \p^\beta u \|_{H^r(\Omega^+)},
\end{align*}
where $C_r > 0$ depends only on $r$. Consequently, we obtain for all $N\in\bN$:
\begin{align*}
\sum_{n=0}^N \frac{\rho^n}{n!} \, \max_{|\alpha'| = n} \| \p^{\alpha'} \big( u|_\Gamma \big) \|_{H^{r-\frac{1}{2}}(\bT^2)} 
\, \leq \, C_r \, \sum_{n=0}^N \frac{\rho^n}{n!} \, \max_{\substack{|\beta| = n \\ \beta_3 = 0}} \| \p^\beta u \|_{H^r(\Omega^+)} 
\, \leq \, C_r \, \|u\|_{\rho,r,\sigma}^+,
\end{align*}
hence estimate \eqref{art-estim_trace} taking the supremum in $N$.
\end{proof}

\subsection{Main theorem}\label{s2-main_thm}

The main theorem of this paper reads:
\begin{thm}\label{s2-thm_solutions_analytiques_nappes}
There exist $\rho_0 \in (0,1]$, $\sigma_0 \in (0,\frac{1}{2}]$ and $\eta_1 > 0$, such that for fixed $\sigma\in (0,\sigma_0]$ and for all $R>0$, we have the following result: let $(u_0^\pm,b_0^\pm,f_0)$ be some initial datum satisfying the initial constraints \eqref{s2-contraintes_u0_b0_f0} (see \mbox{Section \ref{s2-sec_schema_reso}} below) and the conditions:
\begin{align*}
& u_0^\pm, \, b_0^\pm \in B_{\rho_0,3,\sigma}(\Omega^\pm)^3, \quad f_0 \in B_{\rho_0,\frac{7}{2}}(\bT^2), \\[0.5ex]
& \| u_0^\pm, \, b_0^\pm \|_{\rho_0,3,\sigma}^\pm \, < \, R, \quad \| f_0 \|_{\rho_0,\frac{7}{2}} \, < \, \eta_1.
\end{align*}
Then there exist $a>0$, depending only on $\eta_1$ and $R$, and a unique solution $(u^\pm,b^\pm,Q^\pm,f)$ to the system \eqref{s2-eq_evo_u_b}, \eqref{s2-contraintes_div_cond_bord_u_b} such that for all $\rho\in (0,\rho_0)$, we have:
\begin{align*}
& u^\pm, \, b^\pm, \, Q^\pm \, \in \, \cC^1_t \big( \, [0,a(\rho_0-\rho)) \, ; \, B_{\rho,3,\sigma}(\Omega^\pm) \, \big), \\[0.5ex]
& f \, \in \, \cC^2_t \big( \, [0,a(\rho_0 -\rho)) \, ; \, B_{\rho,\frac{5}{2}}(\bT^2) \, \big) \, \cap \, \cC^0_t \big( \, [0,a(\rho_0 -\rho)) \, ; \, B_{\rho,\frac{7}{2}}(\bT^2) \, \big), \\[0.5ex]
& \forall \, t \in [0,a(\rho_{0}-\rho)), \quad \|u^\pm(t,\cdot), \, b^\pm(t,\cdot) \|_{\rho,3,\sigma}^\pm \, < \, C_0 \, R, \quad \| Q^\pm(t,\cdot) \|_{\rho,3,\sigma}^\pm \, < \, C(\eta_1,R), \\[0.5ex]
& \hspace*{3.31cm} \|f(t,\cdot)\|_{\rho,\frac{5}{2}} < C_0 \, \eta_1,
\end{align*}
where $C_0 > 0$ is a numerical constant and $C(\eta_1,R) > 0$ depends only on $R$ and $\eta_1$; besides, the pressure $Q^\pm$ can be chosen to satisfy the following normalization condition:
\begin{align*}
\forall \, t\in [0,a(\rho_0-\rho)), \quad \sum_\pm \int_{\Omega^\pm} Q^\pm(t,x) \, dx \, = \, 0.
\end{align*}
\end{thm}

\subsubsection*{Strategy of the proof} We have reduced the current-vortex sheet problem \eqref{int-equations_nappes_MHD} into the fixed domains $\Omega^\pm$ and the fixed interface $\Gamma$ using the lifting $\psi$ of the front $f$. Consequently, in Section \ref{s2-sec_analytic_estimates} we will give analytic estimates of the new unknown $\psi$ with respect to the front $f$. Afterwards, in Section \ref{s2-sec_estim_pression}, we will deal with the total pressure $Q^\pm$ which is one of the tricky parts of this article. Using the elliptic problem \eqref{s2-pb_ellip_pression_réécrit} satisfied by the pressure, we will give analytic estimates on $Q^\pm$ depending on the unknowns $u^\pm$, $b^\pm$ and $f$. Finally, we will use the version of \cite{BG77} to construct analytic solutions to system \eqref{s2-eq_evo_u_b}, \eqref{s2-contraintes_div_cond_bord_u_b}. The idea relies on a fixed-point method, therefore we shall give in the next section the functional framework and define a suitable mapping $\Upsilon$ whose any fixed point will provide us with a solution to the current-vortex sheets problem. Eventually, Section \ref{s2-sec_CK_thm} will be devoted to proving the contraction of $\Upsilon$ on an appropriate complete metric space.

\section{Sketch of proof}\label{s2-sec_schema_reso}

We suggest to follow the prove of Cauchy-Kowalevskaya theorem given by \cite{BG77} to construct a local-in-time solution to \eqref{s2-eq_evo_u_b}, \eqref{s2-contraintes_div_cond_bord_u_b} in the Banach spaces defined in Paragraph \ref{s0-sec_analytic_spaces}. The idea is to include the constraints \eqref{s2-contraintes_div_cond_bord_u_b} within those functional spaces. The tricky point is the presence of \textit{two} derivatives of the lifting $\psi$ in \eqref{s2-eq_evo_u_b}, which does not allow to directly apply the Cauchy-Kowalevskaya theorem stated in \cite{BG77}. However, we will be able to adapt the arguments of Baouendi and Goulaouic, introducing an appropriate map $\Upsilon$ defined on a suitable Banach space $\bE_a$, where the parameter $a>0$ will have to be fixed small enough. The inverse of this parameter measures the speed with which the analyticity radius of the solution decreases. The argument consists in proving that $\Upsilon$ will be a contraction map on a closed subset of $\bE_a$, as soon as $a$ is chosen small enough. Obviously, the fixed point of $\Upsilon$ will turn out to be a solution of our problem.

Let $(u_0^\pm,b_0^\pm,f_0)$ be some initial datum such that
\begin{align}\label{art-cond_init}
u_0^\pm, \, b_0^\pm \in B_{\rho_0,3,\sigma}(\Omega^\pm)^3, \quad f_0 \in B_{\rho_0,\frac{7}{2}}(\bT^2),
\end{align}
where $\rho_0 \in (0,1]$ and $\sigma\in (0,\frac{1}{2}]$ will be chosen small enough later on, and fixed once and for all (see Theorem \ref{s2-thm_estim_Q_final} in Section \ref{s2-sec_estim_pression} below). Likewise, we will explain in Section \ref{s2-sec_estim_pression} why we have chosen the Sobolev parameters $3$ and $\frac{7}{2}$ in \eqref{art-cond_init}. Moreover, this initial datum is assumed to satisfy the constraints \eqref{s2-contraintes_div_u_b} and the boundary conditions \eqref{s2-cond_saut_u_b}, \eqref{s2-cond_bord_fixe_u_b}, namely:
\begin{align}\label{s2-contraintes_u0_b0_f0}
\left\{
\begin{array}{r l}
\dv u_0^\pm \, = \, \dv b_0^\pm \, = \, 0, & \quad \text{ in } \quad \Omega^\pm, \\[0.75ex]
\big[ u_{0,3} \big] \, = \, 0 \quad \text{ and } \quad b_{0,3}^+ \, = \, b_{0,3}^- \, = \, 0, & \quad \text{ on } \quad \Gamma, \\[0.75ex]
u_{0,3}^\pm \, = \, b_{0,3}^\pm \, = \, 0, & \quad \text{ on } \quad \Gamma_\pm. 
\end{array}
\right.
\end{align}

\medskip

\noindent
\textbf{Notations.} Given a function $u=u(t,x)$, we will denote the \textit{time}-derivative of order one $\p_t u$ (resp. of order two $\p_t^2 u$) by $\dot{u}$ (resp. $\ddot{u}$). Taking \eqref{art-cond_init} into account, the unknowns $(u^\pm,b^\pm,f)$ clearly satisfy:
\begin{equation}\label{art-primitives_du_db_ddf}
\begin{aligned}
& u^\pm(t) \, = \, u_0^\pm \, + \, \int_0^t \dot{u}(t_1) \, dt_1, \quad 
b^\pm(t) \, = \, b_0^\pm \, + \, \int_0^t \dot{b}(t_1) \, dt_1 \\
& f(t) \, = \, f_0 \, + \, t \, \dot{f}|_{t=0} \, + \, \int_0^t \int_0^{t_1} \ddot{f}(t_2) \, dt_2 \, dt_1.
\end{aligned}
\end{equation}
Actually, in order to define a contraction map $\Upsilon$, the ``real'' unknowns we shall consider later on are $(\dot{u}^\pm,\dot{b}^\pm,\ddot{f})$. This explains why we have introduced the notations \eqref{art-primitives_du_db_ddf}.

We are now able to define a map $\Upsilon$, whose any fixed point will solve system \eqref{s2-eq_evo_u_b}, \eqref{s2-contraintes_div_cond_bord_u_b}. Let $(u^\pm,b^\pm,f)$ be sufficiently smooth and satisfy the constraints \eqref{s2-contraintes_div_u_b} and the boundary conditions \eqref{s2-cond_saut_u_b}, \eqref{s2-cond_bord_fixe_u_b} on a time interval $[0,T]$, namely:
\begin{align}\label{s2-contraintes_u_b_f}
\forall \, t\in [0,T], \quad 
\left\{
\begin{array}{r l}
\dv u^\pm(t) \, = \, \dv b^\pm(t) \, = \, 0, & \quad \text{ in } \quad \Omega^\pm, \\[0.75ex]
\dot{f}(t) \, = \, u_3^+(t) \, = \, u_3^-(t) \quad \text{ and } \quad b_3^+(t) \, = \, b_3^-(t) \, = \, 0, & \quad \text{ on } \quad \Gamma, \\[0.75ex]
u_3^\pm(t) \, = \, b_3^\pm(t) \, = \, 0, & \quad \text{ on } \quad \Gamma_\pm. 
\end{array}
\right.
\end{align}
Then, we define a map: 
\begin{equation}\label{s2-def_Upsilon}
\Upsilon(\dot{u}^\pm,\dot{b}^\pm,\ddot{f}) \, := \, (\dot{u}^{\sharp\pm},\dot{b}^{\sharp\pm},\ddot{f}^\sharp),
\end{equation}
following the ideas of \cite{BG77}: the velocity $u^{\sharp\pm}$ and the magnetic field $b^{\sharp\pm}$ are defined ``explicitly''. In other words, in view of \eqref{s2-eq_evo_u_i} and \eqref{s2-eq_evo_u_3}, we set for $i=1,2$ (omitting the exponents $\,^\pm$):
\begin{subequations}\label{s2-eq_evo_u^sharp}
\begin{equation}\label{s2-eq_evo_u_i^sharp}
\dot{u}_i^\sharp \, := \, - \, \frac{u\sg u_i}{J} \, + \, \frac{b\sg b_i}{J} \, + \, \p_3 \Big( \frac{\p_t\psi \, u_i}{J} \Big) \, + \, \big( u_i \, u_k \, - \,  b_i \, b_k \big) \frac{\p_k J}{J^2} \, - \, J \, \p_i Q \, + \, \p_i\psi \, \p_3 Q,
\end{equation}
and for the normal component:
\begin{equation}\label{s2-eq_evo_u_3^sharp}
\dot{u}_3^\sharp \, := \, - \, \frac{u\sg u_3}{J} \, + \, \frac{b\sg b_3}{J} \, - \, \p_i \Big( \frac{\p_t\psi \, u_i}{J} \Big) \, - \, \big( u_i \, u_k \, - \,  b_i \, b_k \big) \frac{\p_i\p_k\psi}{J^2} \, + \, \p_i\psi \, \p_i Q \, - \, \frac{1 + \p_i\psi\p_i\psi}{J} \, \p_3 Q.
\end{equation}
\end{subequations}
Let us notice that we can solve the Laplace problem \eqref{s2-pb_ellip_pression_réécrit} with source terms given by \eqref{s2-termes_sources_cF_cG_réécrits} because the constraints \eqref{s2-contraintes_u_b_f} are satisfied. Thus, the pressure $Q^\pm$ in \eqref{s2-eq_evo_u^sharp} is well-defined. The initial condition for $u^{\sharp\pm}$ is the same as $u^\pm$, \textit{i.e.} $u^{\sharp\pm}|_{t=0} \, := \, u_0^\pm$.

Using definition \eqref{s2-eq_evo_u^sharp}, we shall be able to propagate the constraints associated with the velocity in \eqref{s2-contraintes_u_b_f}, in order to define the ``new'' front $f^\sharp$ as a solution of the ODE $\dot{f}^\sharp = u_3^{\sharp\pm}$ on $\Gamma$.

\medskip

\noindent
\textbf{Propagation of the jump of the normal velocity.} Taking the trace of equation \eqref{s2-eq_evo_u_3^sharp} on $\Gamma$ (for both states $+$ and $-$) and using \eqref{s2-contraintes_u_b_f}, \eqref{s2-cond_bord_psi}, we get:
\begin{align*}
\p_t u_3^{\sharp\pm} \, = \, - \, 2 \, u_j^\pm \, \p_j\p_tf \, - \, u_i^\pm \, u_j^\pm \, \p_i\p_j f \, + \, b_i^\pm \, b_j^\pm \, \p_i\p_j f \, + \, \p_i f \, \p_i Q^\pm \, - \, \big( 1 + |\nabla' f|^2 \big) \, \p_3 Q^\pm.
\end{align*}
From the jump conditions in the Laplace problem \eqref{s2-pb_ellip_pression_réécrit}, we can deduce:
\begin{align*}
\p_t \big[ u_3^\sharp \big] \, = \, \cG \, - \, \big( 1 + |\nabla' f|^2 \big) \, \big[ \p_3 Q \big] \, = \, 0.
\end{align*}
Therefore, the zero jump condition for the normal velocity is propagated by \eqref{s2-eq_evo_u^sharp} because the initial data $u_{0,3}^\pm$ fulfill \eqref{s2-contraintes_u0_b0_f0}. Consequently, we can define $\ddot{f}^\sharp$ as follows:
\begin{equation}\label{s2-eq_evo_f^sharp}
\ddot{f}^\sharp \, := \, \dot{u}_3^{\sharp +} \, = \, \dot{u}_3^{\sharp -} \quad \text{ on } \quad \Gamma.
\end{equation}
The initial condition for $f^\sharp$ is the same as $f$, that is: $f^\sharp|_{t=0} \, := \, f_0$.

\medskip

\noindent
\textbf{Propagation of the boundary conditions on $\Gamma_\pm$.} Using the boundary conditions on $\Gamma_\pm$ in \eqref{s2-pb_ellip_pression_réécrit} together with \eqref{s2-contraintes_u_b_f}, we easily deduce that:
\begin{align*}
\p_t u_3^{\sharp\pm} \, = \, 0 \quad \text{ on } \quad \Gamma_\pm.
\end{align*}
Likewise, thanks to \eqref{s2-contraintes_u0_b0_f0}, we can see that the boundary conditions \eqref{s2-cond_bord_fixe_u_b} related to the velocity are propagated.

\medskip

\noindent
\textbf{Propagation of the divergence-free constraint for the velocity.} Applying the divergence operator to system \eqref{s2-eq_evo_u_i^sharp}, \eqref{s2-eq_evo_u_3^sharp}, we obtain after some algebra:
\begin{align}\label{s2-propagation_div_u}
\p_t \big( \dv u^{\sharp\pm} \big) \, = \, \p_i \big( J \, \p_i Q^\pm \, - \, \p_i\psi \, \p_3 Q^\pm \big) \, + \, \p_3 \Big( \frac{1 + |\nabla'\psi|^2 }{J} \, \p_3 Q^\pm \, - \, \p_i\psi \, \p_i Q^\pm \Big) \, + \, J \, \cF^\pm,
\end{align}
where we recall the expression \eqref{s2-terme_source_JcF_réécrit} of $J\,\cF^\pm$ in \eqref{s2-propagation_div_u}. Then, using definitions \eqref{s2-def_J_A_a}, we have:
\begin{align*}
\p_t \big( \dv u^{\sharp\pm} \big) \, = \, \dv \big( a \, A^T \, \nabla Q^\pm \big) \, + \, J \, \cF^\pm \, = \, 0,
\end{align*}
thanks to the first equation of \eqref{s2-pb_ellip_pression_réécrit}. Using \eqref{s2-contraintes_u0_b0_f0} once again, we can conclude that the divergence-free constraint on the velocity is propagated by \eqref{s2-eq_evo_u^sharp} as well.

\medskip

Let us now focus on the magnetic field $b^{\sharp\pm}$. As we did for the velocity, we define it ``explicitly'':
\begin{subequations}\label{s2-eq_evo_b^sharp}
\begin{align}
& \dot{b}_i^\sharp \, := \, -\, \frac{u\sg b_i}{J} \, + \, \frac{b\sg u_i}{J} \, + \, \p_3 \Big( \frac{\p_t\psi \, b_i}{J} \Big) \, + \, \big( b_i \, u_k \, - \,  u_i \, b_k \big) \frac{\p_k J}{J^2}, \qquad i=1,2, \label{s2-eq_evo_b_i^sharp} \\
& \dot{b}_3^\sharp \, := \, - \, \frac{u\sg b_3}{J} \, + \, \frac{b\sg u_3}{J} \, - \, \p_i \Big( \frac{\p_t\psi \, b_i}{J} \Big) \, - \, \big( b_i \, u_3 \, - \,  u_i \, b_3 \big) \frac{\p_i J}{J^2}, \label{s2-eq_evo_b_3^sharp}
\end{align}
\end{subequations}
where we have willingly omitted the exponents $\,^\pm$. The initial condition associated with $b^{\sharp\pm}$ is the same as $b^\pm$, \textit{i.e.} $b^{\sharp\pm}|_{t=0} \, := \, b_0^\pm$.

\medskip

\noindent
\textbf{Propagation of the jump conditions on $\Gamma$.} Using \eqref{s2-contraintes_u_b_f}, we easily obtain:
\begin{align*}
\p_t b_3^{\sharp\pm} \, = \, 0 \quad \text{ on } \quad \Gamma,
\end{align*}
which shows that the boundary conditions \eqref{s2-cond_saut_u_b} associated with the magnetic field are propagated.

\medskip

\noindent
\textbf{Propagation of the boundary conditions on $\Gamma_\pm$.} In the same way, we can check that we have the following equation on $\Gamma_\pm$:
\begin{align*}
\p_t b_3^{\sharp\pm} \, = \, 0.
\end{align*}
Therefore, the boundary conditions \eqref{s2-cond_bord_fixe_u_b} on the fixed boundaries $\Gamma_\pm$ are propagated.

\medskip

\noindent
\textbf{Propagation of the divergence-free constraint for the magnetic field.} Let us finish with the divergence-free constraint on $b^{\sharp\pm}$. As we did for the velocity $u^{\sharp\pm}$, we apply the divergence operator to system \eqref{s2-eq_evo_b^sharp}, and we get:
\begin{equation}\label{s2-propagation_div_b}
\begin{aligned}
\p_t \big( \dv b^\sharp \big) & \, = \, - \, \p_i \Big( \frac{u_k \, \p_k b_i - b_k \, \p_k u_i}{J} \Big) \, + \, \p_i \Big( \big( b_i \, u_k \, - \, u_i \, b_k \big) \, \frac{\p_k J}{J^2} \Big) \\[0.75ex]
& \hspace*{0.5cm} - \, \p_3 \Big( \frac{u_k \, \p_k b_3 \, - \, b_k \, \p_k u_3}{J} \Big) \, - \, \p_3 \Big( \big( b_i \, u_3 \, - \, u_i \, b_3 \big) \, \frac{\p_i J}{J^2} \Big).
\end{aligned}
\end{equation}
Taking advantage of the skewsymmetric terms on the right-hand side of \eqref{s2-propagation_div_b}, and using the divergence-free constraints in \eqref{s2-contraintes_u_b_f}, it follows:
\begin{align*}
\p_t \big( \dv b^\sharp \big) \, = \, 0.
\end{align*}
Thus, the divergence-free constraint on the magnetic field is propagated.

\medskip

Let us sum up: we have defined a map $\Upsilon \, : \, (\dot{u}^\pm,\dot{b}^\pm,\ddot{f}) \, \mapsto \, (\dot{u}^{\sharp\pm},\dot{b}^{\sharp\pm},\ddot{f}^\sharp)$ through \eqref{s2-eq_evo_u^sharp}, \eqref{s2-eq_evo_f^sharp} and \eqref{s2-eq_evo_b^sharp}. Besides, the algebraic constraints \eqref{s2-contraintes_u_b_f} are invariant under $\Upsilon$, as soon as these are initially required (which is the case, see \eqref{s2-contraintes_u0_b0_f0}). However, we did not give the functional framework in which we will seek the unknown $(u^\pm,b^\pm,f)$.

From now on, the purpose is to define a suitable complete metric space on which the map $\Upsilon$ will be a contraction. To do so, we follow the ideas of \cite{BG77}, and we will construct such a complete metric space using the Banach spaces introduced in Paragraph \ref{s0-sec_analytic_spaces}.

In the following, the small parameters $\rho_0, \sigma_0 > 0$ are fixed and given by Theorem \ref{s2-thm_estim_Q_final} below. The parameter $\sigma \in (0,\sigma_0]$ is also fixed once and for all. We introduce the Banach space $E_a(\Omega)$, where $a>0$ will have to be chosen small enough later on:
\begin{equation}\label{s2-def_E_a,Omega}
E_a(\Omega) \, := \, \bigg\{ \dot{w} \in \bigcap_{\rho\in (0,\rho_0)} \cC\big( \big[0 , a(\rho_0-\rho) \big) \, ; \, B_{\rho,3,\sigma}(\Omega) \big) \, \, \bigg| \, \, \nt{\dot{w}}_{a,\Omega} \, < \, +\infty \bigg\},
\end{equation}
where the norm $\nt{\cdot}_{a,\Omega}$ is defined by:
\begin{equation}\label{s2-def_nt_a_Omega}
\nt{\dot{w}}_{a,\Omega} \, := \, \sup_{\substack{0<\rho<\rho_0 \\ 0\leq t < a(\rho_0-\rho)}} \| \dot{w}(t) \|_{\rho,3,\sigma} \, (\rho_0 - \rho) \sqrt{1 - \frac{|t|}{a(\rho_0-\rho)}}.
\end{equation}
Similarly, we define Banach spaces $E_a(\Omega^\pm)$ and $E_a(\bT^2)$, whose definitions are analogous to \eqref{s2-def_E_a,Omega}: it suffices to replace $B_{\rho,3,\sigma}(\Omega)$ in \eqref{s2-def_E_a,Omega} by $B_{\rho,3,\sigma}(\Omega^\pm)$ and $B_{\rho,\frac{5}{2}}(\bT^2)$ respectively. The norms will be denoted by $\nt{\cdot}_{a,\Omega^\pm}$ and $\nt{\cdot}_{a,\bT^2}$.

\begin{rmq}\label{art-rmq_E_a(bT^2)} The construction of $E_a(\bT^2)$ is based on the space $B_{\rho,\frac{5}{2}}(\bT^2)$ but not on $B_{\rho,\frac{7}{2}}(\bT^2)$. Although we assume that the initial front $f_0$ belongs to $B_{\rho,\frac{7}{2}}(\bT^2)$ (see Theorem \ref{s2-thm_solutions_analytiques_nappes}), the front $f$ will be sought in $B_{\rho,\frac{5}{2}}(\bT^2)$. Indeed, the latter is defined by the \textit{transport} equation $\p_t f = u_3^\pm|_\Gamma$. Clearly, since $u^\pm$ will belong to $B_{\rho,3,\sigma}(\Omega^\pm)$, we will have $\dot{f}\in B_{\rho,\frac{5}{2}}(\bT^2)$ thanks to the trace estimate \eqref{art-estim_trace}. Actually, by mean of time integrations, we shall be able to gain some regularity in space, hence the assumption $f_0 \in B_{\rho,\frac{7}{2}}(\bT^2)$.
\end{rmq}

On the one hand, the advantage of the space $E_a(\Omega)$ is that it allows the norm $\|\dot{w}(t)\|_{\rho,3,\sigma}$ to ``diverge'' as a time-integrable singularity for $t\to a(\rho_0-\rho)$, which is crucial in the proof of \cite{BG77}. On the other hand, Baouendi and Goulaouic give another vision of the way we can solve a Cauchy problem of the form:
\begin{align}\label{s2-pb_cauchy}
\left\{
\begin{array}{l}
\ds \frac{dw}{dt} = F(w), \\[1ex]
w(0) = w_0.
\end{array}
\right.
\end{align}
Indeed, we usually look for a fixed point of the map:
\begin{align*}
w \, \mapsto \, \Big[ t \, \mapsto \, w_0 \, + \, \int_0^t F\big( w(\tau) \big) \, d\tau \Big],
\end{align*}
which gives the proof of Cauchy-Kowalevskaya theorem by Nishida \cite{Nishida} or Nirenberg \cite{Nirenberg}. Baouendi and Goulaouic reformulate the Cauchy problem \eqref{s2-pb_cauchy} in order to look for a fixed point associated with the time-derivative $\dot{w}$. In other terms, they exhibit a fixed point of the map:
\begin{align*}
\dot{w} \, \mapsto \, \Big[ t \, \mapsto \, F\Big( w_0 \, + \, \int_0^t \dot{w}(\tau) \, d\tau \Big) \Big].
\end{align*}

Now, we want to take the constraints \eqref{s2-contraintes_u_b_f} into account which have to be satisfied by the unknowns $(u^\pm,b^\pm,f)$. We incorporate these constraints within the definition of the functional spaces. More precisely, we set:
\begin{equation}\label{s2-def_bE_a}
\bE_a \, := \, \Big\{ (\dot{u}^\pm,\dot{b}^\pm,\ddot{f}) \in E_a(\Omega^\pm)^3 \times E_a(\Omega^\pm)^3 \times E_a(\bT^2) \, \, \Big| \, \, (\dot{u}^\pm,\dot{b}^\pm,\ddot{f}) \, \text{ satisfies } \, \eqref{s2-contraintes_dotu_dotb_dotf} \text{ below } \Big\}.
\end{equation}
The constraints \eqref{s2-contraintes_dotu_dotb_dotf} readily follow from \eqref{s2-contraintes_u_b_f} :
\begin{align}\label{s2-contraintes_dotu_dotb_dotf} 
\left\{
\begin{array}{r l}
\dv \dot{u}^\pm \, = \, \dv \dot{b}^\pm \, = \, 0, & \quad \text{ in } \quad \Omega^\pm, \\[0.75ex]
\ddot{f} \, = \, \dot{u}_3^+ \, = \, \dot{u}_3^- \quad \text{ and } \quad \dot{b}_3^+ \, = \, \dot{b}_3^- \, = \, 0, & \quad \text{ on } \quad \Gamma, \\[0.75ex]
\dot{u}_3^\pm \, = \, \dot{b}_3^\pm \, = \, 0, & \quad \text{ on } \quad \Gamma_\pm. 
\end{array}
\right.
\end{align}

The goal is to show that $\Upsilon$, defined by \eqref{s2-def_Upsilon}, is a contraction map on a suitable closed subset of $\bE_a$. Previously, we proved that the constraints \eqref{s2-contraintes_u_b_f} were propagated by $\Upsilon$. It remains to study the \textit{regularity} of $\Upsilon(\dot{u}^\pm,\dot{b}^\pm,\ddot{f})$ as soon as $(\dot{u}^\pm,\dot{b}^\pm,\ddot{f})$ belongs to $\bE_a$.

More precisely, in Paragraph \ref{art-sec_espace_invariant} we will exhibit an invariant subset of $\bE_a$ under $\Upsilon$, choosing $a>0$ small enough. Afterwards, in Paragraph \ref{art-sec_contraction} we will prove that $\Upsilon$ is a contraction map when $a>0$ is chosen small enough as well.

Within the next section, we focus on the lifting $\psi$ of the front defined by Lemma \ref{s2-lemme_redressement_CMST}, which allowed us to rewrite the current-vortex sheet problem \eqref{int-equations_nappes_MHD} into the fixed domains $\Omega^\pm$ and the fixed interface $\Gamma$. We need to give analytic estimates of the lifting $\psi$ (defined on the whole domain $\Omega$) with respect to the front $f$ (defined on the boundary $\Gamma\simeq\bT^2$).

\section{Lifting of the front and analytic estimates}\label{s2-sec_analytic_estimates}

\subsection{Analytic estimate of the lifting \texorpdfstring{$\psi$}{Lg}}

For the time being, $\rho_0$ stands for any real number in the interval $(0,1]$, but it will be required to be small enough later on.

We will show that the gain of half a derivative for $\psi$ with respect to $f$ (see Lemma \ref{s2-lemme_redressement_CMST}) persists within the analytic norms $\|\cdot\|_{\rho,r,\sigma}$ defined in Paragraph \ref{s0-sec_analytic_spaces}.

\begin{prop}\label{s2-prop_redressement_front}
Let $\rho\in (0,\rho_0]$, $r\geq 2$ be an integer and $\sigma\in (0,\frac{1}{2}]$. Then we can choose the linear map of \mbox{Lemma \ref{s2-lemme_redressement_CMST}} such that, if it is restricted to the space $B_{\rho,r-\frac{1}{2}}(\bT^2)$, then it maps continuously $B_{\rho,r-\frac{1}{2}}(\bT^2)$ into $B_{\rho,r,\sigma}(\Omega)$. In other terms, for all $f\in B_{\rho,r-\frac{1}{2}}(\bT^2)$, we have $\psi\in B_{\rho,r,\sigma}(\Omega)$ with the following estimate:
\begin{equation}\label{s2-estim_psi_f}
\| \psi \|_{\rho,r,\sigma} \, \leq \, C_r \, \| f \|_{\rho,r-\frac{1}{2}},
\end{equation}
where $C_r>0$ depends only on $r$. The choice of this linear map is independent of the parameters $\rho$, $r$ and $\sigma$.
\end{prop}

\begin{proof}
We adapt the proof of Lemma 1 in \cite{CMST}. We begin by recalling the construction of $\psi$ from $f$.

\medskip

\noindent
\textbf{Step 1.} For $g\in B_{\rho,r-\frac{1}{2}}(\bT^2)$, we define the function $\Phi_g$ by
\begin{equation}\label{s2-Phi_g}
\Phi_{g}(x',x_{3}) \, := \, \varphi(x_{3} |D|) \, g(x'), \quad \forall \, (x',x_{3})\in\Omega,
\end{equation}
where the operator $|D|$ corresponds to the Fourier multiplier by $|k|$ (with $k\in\bZ^2$) in the tangential variable $x'\in\bT^2$. For the time being, we want the function $\varphi$ to satisfy two conditions:
\begin{align*}
\varphi(0) = 1 \quad \text{ and } \quad \varphi\in H^\infty(\bR).
\end{align*}
The aim of this step is to bound from above $\|\Phi_{g}\|_{\rho,r,\sigma}$ by $\|g\|_{\rho,r-\frac{1}{2}}$. First of all, we are going to estimate the norms $\| \Phi_g \|_{\rho,r}^k$, for all $k\in\bN$. In the following computations, $C_0>0$ will denote any numerical constant, and $C_r>0$ any constant that depends only on $r$. Let us start estimating $\|\Phi_g\|_{H^r(\Omega)}$. To do so, we use Fubini's theorem:
\begin{align*}
\|\Phi_{g}\|_{H^{r}(\Omega)}^{2} = \sum_{j=0}^{r} \|\p_{3}^{j}\Phi_{g}\|_{L_{x_{3}}^{2}((-1,1) ; H^{r-j}(\bT^{2}))}^{2},
\end{align*}
and we handle each term $\ds \|\p_3^j\Phi_g\|_{L_{x_3}^2(H^{r-j}(\bT^2))}^2$ separately.
\begin{itemize}[label=$\blacktriangleright$]
\item For $j=0$: by definition of the operator $|D|$, we can write:
\begin{align*}
\|  \Phi_{g}(\cdot,x_{3}) \|_{H^{r}(\bT^{2})}^{2} & = \sum_{k\in\bZ^{2}} (1+|k|^{2})^{r} \, |\varphi(x_{3}|k|)|^{2} \, |c_{k}(g)|^{2} \\
& = |c_0(g)|^{2} + \sum_{k\neq 0} (1+|k|^{2})^{r} \, |\varphi(x_{3}|k|)|^{2} \, |c_{k}(g)|^{2}.
\end{align*}
Integrating over $x_3\in (-1,1)$, we get:
\begin{align*}
\|\Phi_{g}\|_{L_{x_{3}}^{2}(H^{r}(\bT^{2}))}^{2} & \, = \, 2|c_0(g)|^2 \, + \, \sum_{k\neq 0} (1+|k|^{2})^{r} \, |c_{k}(g)|^{2} \,  \int_{-1}^{1} |\varphi(x_{3}|k|)|^{2} dx_{3} \\
& \, = \, 2|c_0(g)|^{2} \, + \, \sum_{k\neq 0} (1+|k|^{2})^{r} \, |c_{k}(g)|^{2} \, \int_{-|k|}^{|k|} |\varphi(s)|^{2} \frac{ds}{|k|} \\
& \, \leq \, C_0 \, \Big( |c_0(g)|^{2} + \|\varphi\|_{L^{2}(\bR)}^{2} \sum_{k\neq 0} (1+|k|^{2})^{r-\frac{1}{2}} \, |c_{k}(g)|^{2} \Big) \\
& \, \leq \, C_0 \, \left( 1 + \|\varphi\|_{L^{2}(\bR)}^{2} \right) \| g \|_{H^{r-\frac{1}{2}}(\bT^{2})}^{2}.
\end{align*}

\item For $j\in\{1,\dots,r\}$, differentiating the function $\Phi_g$, we have:
\begin{align*}
\p_{3}^{j}\Phi_{g}(x) \, = \, \varphi^{(j)}(x_{3}|D|) \, (|D|^{j}g)(x').
\end{align*}
Thus,
\begin{align*}
\|  \p_{3}^{j}\Phi_{g}(\cdot,x_{3}) \|_{H^{r-j}(\bT^{2}))}^{2} & \, = \, \sum_{k\in\bZ^{2}} (1+|k|^{2})^{r-j} \, |\varphi^{(j)}(x_{3}|k|)|^{2} \, |k|^{2j} \, |c_{k}(g)|^{2} \\
& \, = \, \sum_{k\neq 0} (1+|k|^{2})^{r-j} \, |\varphi^{(j)}(x_{3}|k|)|^{2} \, |k|^{2j} \, |c_{k}(g)|^{2} \\
& \, \leq \, \sum_{k\neq 0} (1+|k|^{2})^{r} \, |\varphi^{(j)}(x_{3}|k|)|^{2} \, |c_{k}(g)|^{2}.
\end{align*}
Integrating over $x_3$, we obtain:
\begin{align*}
\|\p_{3}^{j}\Phi_{g}\|_{L_{x_{3}}^{2}(H^{r-j}(\bT^{2}))}^{2} & \, \leq \, \sum_{k\neq 0} (1+|k|^{2})^{r} \, |c_{k}(g)|^{2} \,  \int_{-|k|}^{|k|} |\varphi^{(j)}(s)|^{2} \frac{ds}{|k|} \\
& \, \leq \, C_0 \, \| \varphi^{(j)} \|_{L^{2}(\bR)}^{2} \sum_{k\neq 0} (1+|k|^{2})^{r-\frac{1}{2}} \, |c_{k}(g)|^{2} \\
& \, \leq \, C_0 \, \| \varphi^{(j)} \|_{L^{2}(\bR)}^{2} \|g\|_{H^{r-\frac{1}{2}}(\bT^{2})}^{2}.
\end{align*}
\end{itemize}
Summing over $j$, we finally have:
\begin{align*}
\|\Phi_{g}\|_{H^{r}(\Omega)}^{2} \, \leq \, C_0 \, \left( 1 + \|\varphi\|_{H^{r}(\bR)}^{2} \right) \| g \|_{H^{r-\frac{1}{2}}(\bT^{2})}^{2},
\end{align*}
hence
\begin{equation}\label{s2-estim_Phi_g_Hr}
\|\Phi_{g}\|_{H^{r}(\Omega)} \, \leq \, C_0 \, \left( 1 + \|\varphi\|_{H^{r}(\bR)} \right) \| g \|_{H^{r-\frac{1}{2}}(\bT^{2})}.
\end{equation}

\medskip

\noindent
\textbf{Step 2.} Now, let us focus on the norms $\|\Phi_g\|_{\rho,r}^k$, for all $k\in\bN$. We begin with the case $k=0$, for which the tangential derivatives $\p^\alpha$, with $\alpha=(\alpha_1,\alpha_2,0)=(\alpha',0)$, commute with the operator $\varphi(x_{3}|D|)$. We easily have:
\begin{align*}
\|\Phi_{g}\|_{\rho,r}^{0} \, = \, \sum_{n\geq 0} \frac{\rho^{n}}{n!} \max_{\substack{|\alpha|=n \\ \alpha_{3}=0}} \|\p^{\alpha} \Phi_{g} \|_{H^{r}(\Omega)}  \, = \, \sum_{n\geq 0} \frac{\rho^{n}}{n!} \max_{\substack{|\alpha|=n \\ \alpha_{3}=0}} \| \Phi_{\p^{\alpha'}g} \|_{H^{r}(\Omega)}.
\end{align*}
Therefore, we can apply estimate \eqref{s2-estim_Phi_g_Hr} to the function $\p^{\alpha'}g$ (instead of $g$), and we get:
\begin{align}
\|\Phi_{g}\|_{\rho,r}^{0} & \, \leq \, C_0 \, \left(1 + \|\varphi\|_{H^{r}(\bR)} \right) \sum_{n\geq 0} \frac{\rho^{n}}{n!} \max_{|\alpha'|=n} \| \p^{\alpha'} g \|_{H^{r-\frac{1}{2}}(\bT^{2})} \nonumber \\
& \, \leq \, C_0 \left(1 + \|\varphi\|_{H^{r}(\bR)} \right) \|g\|_{\rho,r-\frac{1}{2}} \label{s2-estim_Phi_g_0}.
\end{align}
In order to estimate the norms $\|\Phi_{g}\|_{\rho,r}^{k+1}$, with $k\geq 0$, we proceed by induction:
\begin{align}
\|\Phi_{g}\|_{\rho,r}^{k+1} & = \sum_{n\geq 0} \frac{\rho^{n}}{n!} \max_{\substack{|\alpha|=n \\ \alpha_{3}\leq k+1}} \|\p^{\alpha} \Phi_{g} \|_{H^{r}(\Omega)} \nonumber \\
& \leq \|\Phi_{g}\|_{\rho,r}^{k} + \sum_{n\geq k+1} \frac{\rho^{n}}{n!} \max_{\substack{|\alpha|=n \\ \alpha_{3} = k+1}} \|\p^{\alpha} \Phi_{g} \|_{H^{r}(\Omega)} \nonumber \\
& =  \|\Phi_{g}\|_{\rho,r}^{k} + \sum_{n\geq k+1} \frac{\rho^{n}}{n!} \max_{|\alpha'|=n-k-1} \|\p_{3}^{k+1} \Phi_{\p^{\alpha'}g} \|_{H^{r}(\Omega)}. \label{s2-estim_Phi_g_k+1_rec} 
\end{align}
As previously, we estimate the norms $\ds \|\p_3^j \p_3^{k+1} \Phi_{\p^{\alpha'}g} \|_{L_{x_{3}}^{2}(H^{r-j}(\bT^{2}))}$, for all $j\in\{0,\dots,r\}$:
\begin{align*}
\| \p_3^{j+k+1} \Phi_{\p^{\alpha'}g}(\cdot, x_3) \|_{H^{r-j}(\bT^{2})}^{2}
& = \sum_{\ell\neq 0} (1+|\ell|^2)^{r-j} \, |\varphi^{(j+k+1)}(x_{3}|\ell|)|^2 \, |\ell|^{2(j+k+1)} \, |c_\ell(\p^{\alpha'}g)|^{2} \\
& \leq \sum_{\ell\neq 0} (1+|\ell|^2)^r \, |\ell|^{2(k+1)} \, (2\pi|\ell_1|)^{2\alpha_1} \, (2\pi|\ell_2|)^{2\alpha_2} \, |c_\ell(g)|^2 \, |\varphi^{(j+k+1)}(x_{3}|\ell|)|^2.
\end{align*}
Integrating over $x_3$, we have:
\begin{align}\label{s2-ineg_j+k+1}
\| \p_3^{j+k+1}  \Phi_{\p^{\alpha'}g} &\|_{L_{x_3}^{2}(H^{r-j}(\bT^{2}))}^{2} \leq C_0 \, \|\varphi^{(j+k+1)}\|_{L^{2}(\bR)}^{2} \, \sum_{\ell\neq 0} (1+|\ell|^{2})^{r-\frac{1}{2}} \, |\ell|^{2(k+1)} (2\pi|\ell_1|)^{2\alpha_1} (2\pi|\ell_2|)^{2\alpha_2} \, |c_\ell(g)|^2.
\end{align}
Then, we rewrite the sum in \eqref{s2-ineg_j+k+1} as $\sum_{|\ell_1| \leq |\ell_2|} + \sum_{|\ell_2| < |\ell_1|}$ (if $\ell =0$, the general term of the series vanishes). Thus,
\begin{align}
\sum_{|\ell_1|\leq|\ell_2|} (1+|\ell|^{2})^{r-\frac{1}{2}} \, |\ell|^{2(k+1)} & (2\pi|\ell_1|)^{2\alpha_1} (2\pi|\ell_2|)^{2\alpha_2} \, |c_\ell(g)|^2 \nonumber\\
& \leq \sum_{|\ell_1|\leq|\ell_2|} (1+|\ell|^{2})^{r-\frac{1}{2}} \, (2\ell_2^2)^{k+1} \, (2\pi|\ell_1|)^{2\alpha_1} (2\pi|\ell_2|)^{2\alpha_2} \, |c_\ell(g)|^2 \nonumber\\
& \leq \sum_{|\ell_1|\leq|\ell_2|} (1+|\ell|^{2})^{r-\frac{1}{2}} \, (2\pi|\ell_1|)^{2\alpha_1} (2\pi|\ell_2|)^{2(\alpha_2+k+1)} \, |c_\ell(g)|^2.
\label{s2-ineg_ell1_leq_ell2}
\end{align}
Let us set $\gamma' := (\alpha_1, \alpha_2+k+1)$. Then we have $|\gamma'|=n$, and inequality \eqref{s2-ineg_ell1_leq_ell2} gives:
\begin{align}
\sum_{|\ell_1|\leq|\ell_2|} (1+|\ell|^{2})^{r-\frac{1}{2}} \, |\ell|^{2(k+1)} \, (2\pi|\ell_1|)^{2\alpha_1} (2\pi|\ell_2|)^{2\alpha_2} \, |c_\ell(g)|^2 & \leq \sum_{|\ell_1|\leq|\ell_2|} (1+|\ell|^{2})^{r-\frac{1}{2}} \, \big|c_\ell(\p^{\gamma'}g)\big|^2 \nonumber\\
& \leq \, \|\p^{\gamma'}g\|_{H^{r-\frac{1}{2}}(\bT^2)}^2 \, \leq \, \max_{|\beta'|=n} \|\p^{\beta'}g\|_{H^{r-\frac{1}{2}}(\bT^2)}^2. \label{s2-ineg_ell1_leq_ell2_2}
\end{align}
The case of the second sum $\sum_{|\ell_2| < |\ell_1|}$ is symmetric, and gives the same upper bound as \eqref{s2-ineg_ell1_leq_ell2_2}. Back to \eqref{s2-ineg_j+k+1}, we obtain:
\begin{align*}
\| \p_3^{j+k+1}  \Phi_{\p^{\alpha'}g} &\|_{L_{x_3}^{2}(H^{r-j}(\bT^{2}))}^{2} \, \leq \, C_0 \, \|\varphi^{(j+k+1)}\|_{L^{2}(\bR)}^{2} \, \max_{|\beta'|=n} \| \p^{\beta'}g \|_{H^{r-\frac{1}{2}}(\bT^{2})}^2.
\end{align*}
Summing over $j$, it follows that
\begin{align}\label{s2-estim_d^k+1_Phi_g}
\|\p_3^{k+1} \Phi_{\p^{\alpha'}g}\|_{H^r(\Omega)} \, \leq \, C_0 \, \|\varphi^{(k+1)}\|_{H^r(\bR)} \, \max_{|\beta'|=n} \|\p^{\beta'}g\|_{H^{r-\frac{1}{2}}(\bT^2)}.
\end{align}
Using \eqref{s2-estim_d^k+1_Phi_g} in estimate \eqref{s2-estim_Phi_g_k+1_rec}, we have:
\begin{align*}
\|\Phi_{g}\|_{\rho,r}^{k+1} \, \leq \, \|\Phi_{g}\|_{\rho,r}^{k} \, + \, C_0 \, \|\varphi^{(k+1)}\|_{H^{r}(\bR)} \, \|g\|_{\rho,r-\frac{1}{2}}.
\end{align*}
By induction and using \eqref{s2-estim_Phi_g_0}, we end up with the following estimate, holding for all $k\geq 0$:
\begin{equation}\label{s2-estim_Phi_g_k}
\|\Phi_{g}\|_{\rho,r}^k \, \leq \, C_0 \, \Big( 1 + \sum_{j=0}^k \|\varphi^{(j)}\|_{H^{r}(\bR)} \Big) \, \|g\|_{\rho,r-\frac{1}{2}}.
\end{equation}

\medskip

\noindent
\textbf{Step 3.} To conclude about the estimate of the norm $\|\Phi_{g}\|_{\rho,r,\sigma}$, it remains to multiply \eqref{s2-estim_Phi_g_k} by $\sigma^k$ and to sum over $k$. To do so, the Sobolev norms of the functions $\varphi^{(j)}$ must not increase ``more'' than geometrically. Thus, we choose $\varphi\in H^\infty(\bR)$ such that:
\begin{align*}
\widehat{\varphi}(\xi) := \pi \, \mathds{1}_{[-1,1]}(\xi), \quad \forall \, \xi\in\bR.
\end{align*}
The function $\varphi$ is given by the following formula:
\begin{align*}
\varphi(x) = \frac{1}{2\pi} \, \int_\bR \mathrm{e}^{i x\xi} \, \widehat{\varphi}(\xi) \, d\xi = \frac{\sin(x)}{x}.
\end{align*}
In particular, we have $\varphi(0)=1$ and $\varphi\in H^\infty(\bR)$. Besides, we easily estimate the norms $\|\varphi^{(j)}\|_{H^{r}(\bR)}$ using the Fourier transform:
\begin{equation}\label{s2-norme_Sob_varphi_j}
\|\varphi^{(j)}\|_{H^{r}(\bR)}^{2} \leq C_r \int_\bR (1+\xi^2)^r \, \xi^{2j} \, \widehat{\varphi}(\xi)^{2} \, d\xi \leq C_r \, \int_0^1 (1+\xi^2)^r  \, \xi^{2j} \, d\xi \leq C_r.
\end{equation}
Estimate \eqref{s2-estim_Phi_g_k} finally reads:
\begin{align*}
\|\Phi_{g}\|_{\rho,r}^{k} \leq C_r \, (k+1) \, \|g\|_{\rho,r-\frac{1}{2}}.
\end{align*}
Now, let $K\in\bN$ and let us write $\|\Phi_g\|_{\rho,r,\sigma}^K$ the partial sum of order $K$ of $\|\Phi_g\|_{\rho,r,\sigma}$. Then, for $\sigma\in (0,\frac{1}{2}]$,
\begin{align*}
\|\Phi_g\|_{\rho,r,\sigma}^K \, = \, \sum_{k=0}^K \, \sigma^k \, \|\Phi_g\|_{\rho,r}^k & \, \leq \, C_r \, \sum_{k=0}^K \, (k+1) \, \sigma^k \, \|g\|_{\rho,r-\frac{1}{2}} \, \leq \, C_r \, \|g\|_{\rho,r-\frac{1}{2}}.
\end{align*}
Taking the supremum over $K$, we conclude that $\Phi_g\in B_{\rho,r,\sigma}$ and
\begin{align}\label{s2-estim_Phi_g_B_rho,r,sigma}
\|\Phi_g\|_{\rho,r,\sigma} \, \leq \, C_r \, \|g\|_{\rho,r-\frac{1}{2}}.
\end{align}

\medskip

\noindent
\textbf{Conclusion.} Given $f\in B_{\rho,r-\frac{1}{2}}(\bT^2)$, we define the function $\Psi_f = \psi$ by:
\begin{equation}\label{s2-Psi_f}
\Psi_{f}(x',x_{3}) := (1-x_{3}^{2}) \, \Phi_{f}(x) = (1-x_{3}^{2}) \, \varphi(x_{3} |D|) \, f(x'), \quad \forall \, (x',x_{3})\in\Omega.
\end{equation}
The definition of $\psi$ above can also be considered in the proof done by \cite{CMST}, and does not change any result about the lifting of the front $f$ recalled by Lemma \ref{s2-lemme_redressement_CMST}. 

We remind that we wish to estimate $\|\Psi_{f}\|_{\rho,r,\sigma}$ by $\|f\|_{\rho,r-\frac{1}{2}}$. To do so, it suffices to use the algebra property of the spaces $B_{\rho,r,\sigma}(\Omega)$ given by Theorem \ref{s1-thm_echelle_B_rho,r,sigma} and estimate \eqref{s2-estim_Phi_g_B_rho,r,sigma} obtained in Step 3. Setting $\omega(x_{3}) := 1 - x_{3}^{2}$, we have:
\begin{align*}
\|\Psi_f\|_{\rho,r,\sigma} \, \leq \, C_r \, \|\omega\|_{\rho,r,\sigma} \, \|\Phi_f\|_{\rho,r,\sigma} \, \leq \, C_r \, \|\omega\|_{\rho,r,\sigma} \, \|f\|_{\rho,r-\frac{1}{2}}.
\end{align*}
To finish with, we can explicitly estimate $\|\omega\|_{\rho,r,\sigma}$, since all the tangential derivatives of $\omega$ vanish. For all $k\geq 0$, we get:
\begin{align*}
\|\omega\|_{\rho,r}^k = \sum_{n=0}^k \frac{\rho^n}{n!} \|\p_3^n\omega\|_{H^{r}(\Omega)} \leq \sum_{n=0}^2 \frac{\rho^n}{n!} \|\p_3^n\omega\|_{H^{r}(\Omega)} \leq C_0,
\end{align*}
because all the derivatives of $\omega$ of order larger than 3 vanish, and $\rho$ can be directly bounded above by $1$. As a consequence, we have the estimate
\begin{align*}
\|\omega\|_{\rho,r,\sigma} \leq C_0.
\end{align*}
Thus, we can conclude that there exists a constant $C_r >0$, depending only on $r$, such that:
\begin{equation}\label{s2-estim_Psi_f}
\|\Psi_{f}\|_{\rho,r,\sigma} \, \leq \, C_r \, \|f\|_{\rho,r-\frac{1}{2}}.
\end{equation}
\end{proof}

\subsection{Analytic estimate of the inverse of the Jacobian \texorpdfstring{$J$}{Lg}}

The change of unknown $f\mapsto\psi$ required the continuous estimate \eqref{s2-estim_psi_f}. With the new unknowns $(u^\pm,b^\pm,f)$ defined by \eqref{s2-def_u_b}, the equations of \eqref{s2-eq_evo_u_b} make the inverse of the Jacobian $J$ appear (recall that $J = 1 + \p_3\psi$). To construct analytic solutions to this problem, we will have to be able to estimate the analytic norm $\|\frac{1}{J}\|_{\rho,r,\sigma}$. To do so, we will use the algebra property of the Banach spaces $B_{\rho,r,\sigma}(\Omega)$.
\begin{prop}\label{s2-prop_estim_analytique_1/1+g}
Let $\rho\in (0,\rho_0]$, $r\geq 2$ be an integer and $\sigma\in (0,\frac{1}{2}]$. There exists $\varepsilon_1 = \varepsilon_1(r) > 0$ depending only on $r$, such that for all $g\in B_{\rho,r,\sigma}(\Omega)$ satisfying $\|g\|_{\rho,r,\sigma}\leq\varepsilon_1$, we have:
\begin{align}\label{s2-estim_analytique_1/1+g}
\frac{1}{1+g} \in B_{\rho,r,\sigma}(\Omega) \quad \text{ and } \quad \left\|\frac{1}{1+g}\right\|_{\rho,r,\sigma} \leq C_r,
\end{align}
where $C_r>0$ depends only on $r$.
\end{prop}
\begin{proof}
We start recalling the algebra property of the spaces $B_{\rho,r,\sigma}(\Omega)$. There exists a constant $C_r>0$ depending only on $r$, such that for all $u,v \in B_{\rho,r,\sigma}(\Omega)$, we have:
\begin{align*}
u \, v \in B_{\rho,r,\sigma}(\Omega) \quad \text{ and } \quad \|u \, v\|_{\rho,r,\sigma} \, \leq \, C_r \, \|u\|_{\rho,r,\sigma} \|v\|_{\rho,r,\sigma}.
\end{align*}
We define $\ds\varepsilon_1 := (2C_r)^{-1}$ and consider $g\in B_{\rho,r,\sigma}(\Omega)$ such that $\|g\|_{\rho,r,\sigma}\leq\varepsilon_1$. Then the series $\sum_{m\geq 0} (-1)^m g^m$ is absolutely convergent in the Banach space $B_{\rho,r,\sigma}(\Omega)$. Indeed,
\begin{align*}
\sum_{m\geq 0} \| (-1)^m g^m \|_{\rho,r,\sigma} \leq \|1\|_{\rho,r,\sigma} + \sum_{m\geq 1} C_r^{m-1} \|g\|_{\rho,r,\sigma}^m \leq \|1\|_{\rho,r,\sigma} + C_r^{-1}.
\end{align*}
Then, we estimate the norm of the constant function $1$ in $B_{\rho,r,\sigma}(\Omega)$ by a straightforward computation:
\begin{align*}
\|1\|_{\rho,r,\sigma} \leq C_0,
\end{align*}
where $C_0 >0$ is a numerical constant. Eventually, we deduce that
\begin{align*}
\sum_{m\geq 0} \| (-1)^m g^m \|_{\rho,r,\sigma} \leq C_r' < +\infty,
\end{align*}
where $C_r'>0$ depends only on $r$. The completeness of the spaces $B_{\rho,r,\sigma}(\Omega)$ allows to write the identity
\begin{align*}
\frac{1}{1+g} = \sum_{m\geq 0} (-1)^m g^m.
\end{align*}
Consequently, $\ds\frac{1}{1+g}$ belongs to the space $B_{\rho,r,\sigma}(\Omega)$, and satisfies the estimate:
\begin{align*}
\left\|\frac{1}{1+g}\right\|_{\rho,r,\sigma} \leq C_r'.
\end{align*}
\end{proof}

Using Proposition \ref{s2-prop_redressement_front}, we deduce a straightforward corollary about the analytic estimate of the inverse of the Jacobian $J$.
\begin{coro}\label{s2-coro_estim_1/J}
Let $\rho\in (0,\rho_0]$, $r\geq 2$ be an integer and $\sigma\in (0,\frac{1}{2}]$. Let also $f\in B_{\rho,r+\frac{1}{2}}(\bT^2)$. There exists a constant $\eta_0 = \eta_0(r) > 0$ such that:
\begin{equation}\label{s2-estim_1/J_rho,r,sigma}
\| f \|_{\rho,r+\frac{1}{2}} \leq \eta_0 \quad \Longrightarrow \quad \frac{1}{J}\in B_{\rho,r,\sigma}(\Omega) \quad \text{ and } \quad \left\|\frac{1}{J}\right\|_{\rho,r,\sigma} \leq M(\eta_0),
\end{equation}
where $M(\eta_0)>0$ depends only on $\eta_0$.
\end{coro}
\begin{proof}
Using Proposition \ref{s2-prop_redressement_front}, there exists a constant $\widetilde{C}_r > 0$ such that:
\begin{align}\label{s2-estim_Psi_f_unif}
\|\p_3\psi\|_{\rho,r,\sigma} \, \leq \, \|\psi\|_{\rho,r+1,\sigma} \, \leq \, \widetilde{C}_r \, \|f\|_{\rho,r+\frac{1}{2}}.
\end{align}
Let us set $\ds\eta_0 := \frac{\varepsilon_1}{\widetilde{C}_r}$, where $\varepsilon_1$ is given by Proposition \ref{s2-prop_estim_analytique_1/1+g}. Assuming $\ds \| f \|_{\rho,r+\frac{1}{2}} \leq \eta_0$ and using \eqref{s2-estim_Psi_f_unif}, we obtain the estimate $\|\p_3\psi\|_{\rho,r,\sigma} \leq \varepsilon_1$. Therefore, applying Proposition \ref{s2-prop_estim_analytique_1/1+g} to the function $g := \p_3\psi$, we achieve the proof of Corollary \ref{s2-coro_estim_1/J}.
\end{proof}

\section{Analytic estimate of the total pressure}\label{s2-sec_estim_pression}

In Paragraphs \ref{s2-redressement} and \ref{s2-sec_nouvelle_formulation}, we have seen that the total pressure $(Q^+,Q^-)$ can be implicitly expressed as a function of the velocity, the magnetic field and the front through the resolution of a coupled elliptic problem (we refer to both equivalent systems \eqref{s2-pb_ellip_pression} and \eqref{s2-pb_ellip_pression_réécrit}). Consequently, we will need to get analytic estimates for $Q^\pm$ depending on the analytic norms of $(u^\pm,b^\pm,f)$. Within this section, we choose to work on the first version \eqref{s2-pb_ellip_pression}. 

For the time being, we assume that $\cF^\pm$ and $\cG$ in \eqref{s2-pb_ellip_pression} are \textit{any} source terms. We will study the specific case where $\cF^\pm$ and $\cG$ are given by \eqref{s2-terme_source_cF^pm} and \eqref{s2-terme_source_cG} later on (or by \eqref{s2-termes_sources_cF_cG_réécrits} if we focus on the second version \eqref{s2-pb_ellip_pression_réécrit} using the new unknowns $(u^\pm,b^\pm,f)$). We recall that $t\in [0,T]$ is a parameter in \eqref{s2-pb_ellip_pression}, so we can omit it in this section.

To begin with, we will admit the well-posedness of system \eqref{s2-pb_ellip_pression} in the Sobolev regularity scale. More precisely,
\begin{itemize}[label=$\sbt$]
\item if the source terms $\cF^\pm$ and $\cG$ satisfy the compatibility condition \eqref{s2-cond_nec_comp_cF_cG},
\item if the front $f$ satisfies the smallness condition $\|f\|_{H^{\frac{5}{2}}(\bT^2)} < \varepsilon_0$, for a suitable numerical constant $\varepsilon_0 > 0$,
\item if the source terms $\cF^\pm$ and $\cG$ belong to $H^1(\Omega^\pm)$ and $H^{\frac{3}{2}}(\bT^2)$ respectively,
\end{itemize}
then problem \eqref{s2-pb_ellip_pression} is well-posed in $H^3(\Omega^+)\times H^3(\Omega^-)$. In other terms, there exists a unique solution $(Q^+,Q^-)$ to \eqref{s2-pb_ellip_pression}, with zero mean on $\Omega$, and belonging to $H^3(\Omega^+)\times H^3(\Omega^-)$. Besides, it satisfies the following elliptic regularity estimate:
\begin{equation}\label{s2-reg_ellip_H^3}
\|Q^\pm\|_{3,\pm} \, \leq \, C_0 \, \Big( \|\cF^\pm\|_{1,\pm} + \|\cG\|_{H^{\frac{3}{2}}(\bT^2)} \Big),
\end{equation}
where $C_0>0$ is a constant depending only on $\varepsilon_0$. We do not detail the proof of the well-posedness of \eqref{s2-pb_ellip_pression}, a sketch of proof is given in \cite{CMST}. The main idea relies on classical tools of elliptic analysis, such as the Lax-Milgram theorem applied in the Hilbert space
\begin{align*}
\mathscr{H} \, := \, \Big\{ (u^+,u^-) \in H^1(\Omega^+) \times H^1(\Omega^-) \,\, \Big| \,\, [u] \, = \, 0 \,\, \text{ and } \,\, \sum_\pm \int_{\Omega^\pm} u^\pm \, dx \, = \, 0 \Big\}.
\end{align*}
Then, to get an estimate in $H^2(\Omega^\pm)$, we estimate the difference quotients of order 2 (see \cite{Evans}). Finally, we proceed by induction to gain one more derivative, estimating the appearing commutators properly (see \cite[p.268]{CMST}).

Furthermore, we admit that problem \eqref{s2-pb_ellip_pression} is still well-posed in $H^\infty(\Omega^+)\times H^\infty(\Omega^-)$. In other words, if the coefficients $A_{ji}$ and the source terms $\cF^\pm$ and $\cG$ are from now on in $H^\infty$, then there exists a unique zero mean solution $(Q^+,Q^-)$ to problem \eqref{s2-pb_ellip_pression} that belongs to $H^\infty$ (see \cite{Evans} for the general method). In the following, we shall require more than the $H^\infty$ regularity, since we will take an analytic front $f$ in the space $B_{\rho,\frac{7}{2}}(\bT^2)$. Therefore the lifting $\psi$ will be in $B_{\rho,4,\sigma}(\Omega)$ according to Proposition \ref{s2-prop_redressement_front}. Thus, the coefficients $A_{ji}$ will belong to $B_{\rho,3,\sigma}(\Omega)$. The source terms  $\cF^\pm$ and $\cG$ will be taken respectively in the analytic spaces $B_{\rho,1,\sigma}(\Omega^\pm)$ and $B_{\rho,\frac{3}{2}}(\bT^2)$. Working in such spaces remains consistent with estimate \eqref{s2-reg_ellip_H^3}. Intuitively, if the radius of analyticity $\rho$ and the parameter $\sigma$ go to zero, these spaces ``degenerate'' into the Sobolev spaces $H^3(\Omega)$ (for the coefficients $A_{ji}$), $H^1(\Omega^\pm)$ and $H^\frac{3}{2}(\bT^2)$ (for the source terms $\cF^\pm$ and $\cG$). The purpose of this section is to extend estimate \eqref{s2-reg_ellip_H^3} to the analytic spaces mentioned above.

We sum up the main result of this section about the pressure estimate in the following theorem.
\begin{thm}\label{s2-thm_estim_Q_final}
There exist $\eta_0 > 0$, $\rho_0 \in (0,1]$ and $\sigma_0\in (0,\frac{1}{2}]$ such that, for all $\rho\in (0,\rho_0]$ and $\sigma\in (0,\sigma_0]$, if we consider a front $f\in B_{\rho,\frac{7}{2}}(\bT^2)$ satisfying
\begin{align*}
\|f\|_{\rho,\frac{7}{2}} < \eta_0,
\end{align*}
and analytic source terms $\cF^\pm\in B_{\rho,1,\sigma}(\Omega^\pm)$ and $\cG\in B_{\rho,\frac{3}{2}}(\bT^2)$, then the unique solution $(Q^+,Q^-)$ to \eqref{s2-pb_ellip_pression} with zero mean on $\Omega$ belongs to the space $B_{\rho,3,\sigma}(\Omega^+)\times B_{\rho,3,\sigma}(\Omega^-)$. Moreover, it satisfies the following estimate:
\begin{align}\label{s2-estim_Q_final}
\|Q^\pm\|_{\rho,3,\sigma}^{\pm} \, \leq \, C(\eta_0) \, \Big( \|\cF^\pm\|_{\rho,1,\sigma}^\pm \, + \, \|\cG\|_{\rho,\frac{3}{2}} \Big),
\end{align}
where $C(\eta_0)>0$ is a constant depending only on $\eta_0$.
\end{thm}
\begin{rmq}
Within this section, we shall assume that the front $f$ belongs to $B_{\rho,\frac{7}{2}}(\bT^2)$, which could seem to disagree with the definition of $E_a(\bT^2)$ (see Remark \ref{art-rmq_E_a(bT^2)}). Actually, we shall need the coefficients $A_{ji}$ in \eqref{s2-pb_ellip_pression} to have the same regularity than the velocity and the magnetic field; in other words, we will want $A_{ji}$ to belong to $B_{\rho,3,\sigma}(\Omega)$. Therefore we shall assume that $f \in B_{\rho,\frac{7}{2}}(\bT^2)$ in virtue of Proposition \ref{s2-prop_redressement_front}. Nevertheless, this point will turn out not to be as tricky as it seems and we will see in Section \ref{s2-sec_CK_thm} how to estimate $f$ in $B_{\rho,\frac{7}{2}}(\bT^2)$ as soon as we assume that $\ddot{f}\in E_a(\bT^2)$ (we refer specifically to \eqref{art-estim_rho,7/2_rho,5/2}). We do not claim that this analytic regularity on the coefficients $A_{ji}$ is sharp, however it will be sufficient for our purpose.

We can also explain the assumption $f \in B_{\rho,\frac{7}{2}}(\bT^2)$ looking at the source term $\cG$ defined by \eqref{s2-terme_source_cG_réécrit}. The latter contains \textbf{two} derivatives of the front $f$, but the analytic estimate \eqref{s2-estim_Q_final} shows that we will need $\cG$ to be estimated in $B_{\rho,\frac{3}{2}}(\bT^2)$. Consequently, the front $f$ should belong to $B_{\rho,\frac{7}{2}}(\bT^2)$ and it seems quite difficult to get rid of this assumption.
\end{rmq}

\subsection{Estimate of the tangential derivatives}

In the following, the notation $f \lesssim g$ means that there exists a numerical constant $C_0>0$ (that does not depend on the functions $f$ and $g$, neither on other parameters like $\rho$ or $\sigma$), such that $f\,\leq\, C_0 \, g$.

Within this paragraph, we focus on estimating $Q^\pm$ in the space $B_{\rho,3}^0(\Omega^\pm)$. Thus, we wish to estimate the following quantity:
\begin{align*}
\|Q^\pm\|_{\rho,3}^{0,\pm} = \sum_{n\geq 0} \frac{\rho^n}{n!} \max_{|\alpha'|=n} \|\p^{\alpha'} Q^\pm\|_{3,\pm}.
\end{align*}

Using \eqref{s2-reg_ellip_H^3}, we shall be able to estimate the tangential derivatives $\p^{\alpha'}Q^\pm$ in $H^3(\Omega^\pm)$. Multi-indices $\alpha$ in $\bN^3$ will be noted $\alpha := (\alpha', \alpha_3)$, where $\alpha'\in\bN^2$. In the following, the notation $\alpha'$ will always stand for a multi-index of $\bN^2$, to refer to the tangential derivatives $\p_1$ and $\p_2$. For all $N\in\bN$, we will note $\|Q\|_{\rho,3}^{0,\pm,N}$ the partial sum of order $N$ of the norm $\|Q^\pm\|_{\rho,3}^{0,\pm}$. To estimate $\p^{\alpha'}Q^\pm$ in $H^3(\Omega^\pm)$, we shall commute problem \eqref{s2-pb_ellip_pression} with $\p^{\alpha'}$. Thus, estimate \eqref{s2-reg_ellip_H^3} will hold for $\p^{\alpha'}Q^\pm$, with new source terms $\cF'^\pm$ and $\cG'$. The latter contain commutators, composed of derivatives of $Q^\pm$ and coefficients of the matrix $A$. To estimate these commutators in a suitable way, we will have to take the flatness condition required on the front $f$ in the space $H^{\frac{5}{2}}(\bT^2)$ into account. Therefore the norm $\|\psi\|_{H^3(\Omega)}$ will also be small according to Lemma \ref{s2-lemme_redressement_CMST}. Consequently, if we write $A = I_3 - \tildA$, with
\begin{equation}\label{s2-def_tildA}
\tildA := \begin{pmatrix}
0 & 0 & 0 \\
0 & 0 & 0 \\
\ds\frac{\p_1\psi}{J} & \ds\frac{\p_2\psi}{J} & \ds\frac{\p_3\psi}{J}
\end{pmatrix},
\end{equation}
then $\|\tildA\|_{H^2(\Omega)}$ will also be a small quantity. It explains why during the commutators estimates below, we will have to distinguish the cases where low derivatives of $A$ appear, to take advantage of the flatness of the front $f$.

Now, let $\alpha'\in\bN^2$ such that $|\alpha'| = n \in \bN$. The case $n=0$ is a straightforward application of estimate \eqref{s2-reg_ellip_H^3}, which gives:
\begin{align}\label{s2-estim_Q_n=0}
\|Q^\pm\|_{3,\pm} \, \lesssim \, \|\cF^\pm\|_{1,\pm} + \|\cG\|_{H^\frac{3}{2}(\bT^2)} \, \lesssim \, \|\cF^\pm\|_{\rho,1}^{0,\pm} + \|\cG\|_{\rho,\frac{3}{2}}.
\end{align}
From now on, we assume that $n\geq 1$ and we set $Q'^\pm := \p^{\alpha'} Q^\pm$. Let us commute problem \eqref{s2-pb_ellip_pression} with $\p^{\alpha'}$:
\begin{equation}\label{s2-pb_ellip_pression_tang}
\left\{
\begin{array}{r l}
-A_{ji}\p_j (A_{ki}\p_k Q'^\pm) \, = \, \cF'^\pm, & \text{ in } [0,T]\times\Omega^\pm, \\[0.5ex]
\, [Q'^\pm] \, = \, 0, & \text{ on } [0,T]\times\Gamma, \\[0.5ex]
\, (1+|\nabla' f|^2)[\p_3 Q'] \, = \, \cG', & \text{ on } [0,T]\times\Gamma, \\[0.5ex]
\p_3 Q' \, = \, 0, & \text{ on } [0,T]\times\Gamma_\pm.
\end{array}
\right.
\end{equation}
The source term $\cF'^\pm$ is defined by
\begin{align*}
\cF'^\pm := \p^{\alpha'}\cF^\pm + \big[ \p^{\alpha'} \, ; \, A_{ji}\p_j (A_{ki}\p_k \cdot) \big] Q^\pm.
\end{align*}
Writing $A = I_3 - \tildA$, we can split $\cF'^\pm$ into four quantities as follows:
\begin{align}
\cF'^\pm & = \p^{\alpha'}\cF^\pm + \big[ \p^{\alpha'} \, ; \, \tildA_{ji}\p_j(\tildA_{ki}\p_k\cdot)\big]Q^\pm - \big[ \p^{\alpha'} \, ; \, \p_i\tildA_{ki}\p_k\cdot \big]Q^\pm - 2 \big[ \p^{\alpha'} \, ; \, \tildA_{ji}\p_{ji}\cdot \big]Q^\pm. \nonumber \\
& =: F_1 + F_2 - F_3 - 2F_4. \label{s2-def_F'} 
\end{align}
The boundary source term $\cG'$ is given by
\begin{align}
\cG' & := \p^{\alpha'}\cG - \big[ [\p^{\alpha'} \, ; \, (1+|\nabla' f|^2)\p_3] \, Q \big] \nonumber \\
& =: G_1 - G_2. \label{s2-def_G'}
\end{align}
Let us be careful: the exterior bracket in \eqref{s2-def_G'} stands for the jump across $\Gamma$, whereas the interior bracket corresponds to a commutator. Applying estimate \eqref{s2-reg_ellip_H^3} to the problem \eqref{s2-pb_ellip_pression_tang}, we have:
\begin{equation}\label{s2-estim_Q'_H^3}
\| Q'^\pm \|_{3,\pm} \lesssim \|\cF'\|_{1,\pm} + \|\cG'\|_{H^\frac{3}{2}(\bT^2)}.
\end{equation}
Thus, the partial sum of order $N$ of $\|Q^\pm\|_{\rho,3}^{0,\pm}$ satisfies:
\begin{equation}\label{s2-estim_som_part_Q'}
\|Q^\pm\|_{\rho,3}^{0,\pm,N} \lesssim \sum_{n=0}^N \frac{\rho^n}{n!} \max_{|\alpha'|=n} \|\cF'^\pm\|_{1,\pm} + \sum_{n=0}^N \frac{\rho^n}{n!} \max_{|\alpha'|=n} \|\cG'\|_{H^\frac{3}{2}(\bT^2)}.
\end{equation}

\paragraph*{\textbf{Estimate of $\cF'^\pm$.}}

$\,$

The case of $F_1$ is easy. Indeed, we have assumed that $\cF^\pm$ belongs to the space $B_{\rho,1,\sigma}(\Omega^\pm)$. In particular, the quantity $\|\cF^\pm\|_{\rho,1}^{0,\pm}$ is finite, and we can write:
\begin{align}\label{s2-estim_F1}
\sum_{n=0}^N \frac{\rho^n}{n!} \max_{|\alpha'|=n} \|F_1\|_{1,\pm} \, = \, \sum_{n=0}^N \frac{\rho^n}{n!} \max_{|\alpha'|=n} \|\p^{\alpha'}\cF^\pm\|_{1,\pm} \, \leq \, \|\cF^\pm\|_{\rho,1}^{0,\pm}.
\end{align}

Let us deal with the first commutator given by $\ds F_2 := \big[ \p^{\alpha'} \, ; \, \tildA_{ji}\p_j(\tildA_{ki}\p_k\cdot)\big]Q^\pm$. Expanding the latter expression, we have:
\begin{equation}\label{s2-def_F2}
F_2 = \sum_{\substack{\beta'+\gamma'+\delta'=\alpha' \\ |\beta'+\gamma'|\geq 1}} \frac{\alpha'!}{\beta'! \, \gamma'! \, \delta'!} \left( \p^{\beta'}\tildA_{ji} \, \p^{\gamma'}\p_j\tildA_{ki} \, \p^{\delta'}\p_k Q^\pm \, + \, \p^{\beta'}\tildA_{ji} \, \p^{\gamma'}\tildA_{ki} \, \p^{\delta'}\p_{jk}Q^\pm \right).
\end{equation}
Let us set  $\ds T_{01} := \p^{\beta'}\tildA_{ji} \, \p^{\gamma'}\p_j\tildA_{ki} \, \p^{\delta'}\p_k Q^\pm$ and $\ds T_{02} := \p^{\beta'}\tildA_{ji} \, \p^{\gamma'}\tildA_{ki} \, \p^{\delta'}\p_{jk}Q^\pm$. In order to estimate the terms $T_{01}$ and $T_{02}$ in $H^1(\Omega^\pm)$, we first estimate the $0-$order derivative and then the $1-$order derivatives in $L^2(\Omega^\pm)$. In the following, $\p^m$ denotes any $m-$order derivative, for all $m\geq 1$. All the following estimates are computed using Hölder's inequality and Sobolev embeddings recalled below ($U$ stands for any domain among $\Omega$, $\Omega^+$ or $\Omega^-$):
\begin{align*}
& H^1(U) \hookrightarrow L^p(U), \quad \forall \, 1\leq p\leq 6, \\
& H^2(U) \hookrightarrow L^\infty(U).
\end{align*}
We start with the $0-$order derivative:
\begin{align*}
\| T_{01} \|_{0,\pm} \, & = \, \| \p^{\beta'}\tildA_{ji} \, \p^{\gamma'}\p_j\tildA_{ki} \, \p^{\delta'}\p_k Q^\pm \|_{0,\pm} \\
& \lesssim \|\p^{\beta'}\tildA\|_{H^2} \, \|\p^{\gamma'}\tildA\|_{H^1} \, \| \p^{\delta'}Q^\pm \|_{3,\pm}, & \quad (\textit{Hölder's inequality in } L^\infty\times L^2\times L^\infty).
\end{align*}
In the norms, we will write $H^r$ to denote the space $H^r(\Omega)$ (and only for the domain $\Omega$). Likewise, we get:
\begin{align*}
\| T_{02} \|_{0,\pm} \, & \lesssim \, \|\p^{\beta'}\tildA\|_{H^2} \, \|\p^{\gamma'}\tildA\|_{H^2} \, \| \p^{\delta'}Q^\pm \|_{2,\pm}, & \quad (\textit{Hölder's inequality in } L^\infty\times L^\infty\times L^2).
\end{align*}
Then, we compute the $1-$order derivatives. We write $\ds \p^1 T_{01} = T_{11} + T_{12} + T_{13}$, with:
\begin{align*}
T_{11} = \p^{\beta'}\p^1\tildA_{ji} \, \p^{\gamma'}\p^1\tildA_{ki} \, \p^{\delta'}\p^1 Q^\pm, \quad 
T_{12} = \p^{\beta'}\tildA_{ji} \, \p^{\gamma'}\p^2\tildA_{ki} \, \p^{\delta'}\p^1 Q^\pm, \quad 
T_{13} = \p^{\beta'}\tildA_{ji} \, \p^{\gamma'}\p^1\tildA_{ki} \, \p^{\delta'}\p^2 Q^\pm.
\end{align*}
As done before, applying Hölder's inequality in $L^4\times L^4\times L^\infty$ for $T_{11}$, in $L^\infty\times L^2\times L^\infty$ for $T_{12}$ and in $L^\infty\times L^4\times L^4$ for $T_{13}$, we get:
\begin{equation}\label{s2-estim_T11_T12_T13}
\| T_{11} \|_{0,\pm}, \, \| T_{12} \|_{0,\pm}, \, \| T_{13} \|_{0,\pm} \, \lesssim \, \|\p^{\beta'}\tildA\|_{H^2} \, \|\p^{\gamma'}\tildA\|_{H^2} \, \| \p^{\delta'}Q^\pm \|_{3,\pm}.
\end{equation}
The case of $T_{02}$ is analogous. We write $\ds \p^1 T_{02} = T_{14} + T_{15} + T_{16}$, with:
\begin{align*}
T_{14} = \p^{\beta'}\p^1\tildA_{ji} \, \p^{\gamma'}\tildA_{ki} \, \p^{\delta'}\p^2 Q^\pm, \quad 
T_{15} = \p^{\beta'}\tildA_{ji} \, \p^{\gamma'}\p^1\tildA_{ki} \, \p^{\delta'}\p^2 Q^\pm, \quad 
T_{16} = \p^{\beta'}\tildA_{ji} \, \p^{\gamma'}\tildA_{ki} \, \p^{\delta'}\p^3 Q^\pm.
\end{align*}
Using Hölder's inequality in $L^4\times L^\infty\times L^4$ for $T_{14}$, in $L^\infty\times L^4\times L^4$ for $T_{15}$ and in $L^\infty\times L^\infty\times L^2$ for $T_{16}$, we also show that the terms $T_{14}$, $T_{15}$ and $T_{16}$ satisfy the estimate \eqref{s2-estim_T11_T12_T13}. To sum up, we can write:
\begin{equation}\label{s2-estim_T01+T02}
\| T_{01} + T_{02} \|_{1,\pm} \, \lesssim \, \, \|\p^{\beta'}\tildA\|_{H^2} \, \|\p^{\gamma'}\tildA\|_{H^2} \, \| \p^{\delta'}Q^\pm \|_{3,\pm}.
\end{equation}
From now on, we can estimate the commutator $F_2$ given by \eqref{s2-def_F2}. Applying \eqref{s2-estim_T01+T02}, we have:
\begin{align*}
\| F_2 \|_{1,\pm} & \lesssim \sum_{\substack{\beta' +\gamma' +\delta' = \alpha' \\ |\beta' +\gamma'| \geq 1}} \frac{\alpha'!}{\beta'! \, \gamma'! \, \delta'!} \, \|\p^{\beta'}\tildA\|_{H^2} \, \|\p^{\gamma'}\tildA\|_{H^2} \, \|\p^{\delta'} Q^\pm\|_{3,\pm}.
\end{align*}
Summing over the length of multi-indices (recall that $|\alpha'|=n$), we obtain:
\begin{equation}\label{s2-estim_F2_1}
\| F_2 \|_{1,\pm} \lesssim \sum_{\substack{i_1 + i_2 + i_3 = n \\ i_1 + i_2 \geq 1}} \Bigg( \sum_{\substack{\beta' + \gamma' + \delta' = \alpha' \\ |\beta'| = i_1 \\ |\gamma'| = i_2 \\ |\delta'| = i_3}} \frac{\alpha'!}{\beta'! \, \gamma'! \, \delta'!} \Bigg) \max_{\substack{\beta' + \gamma' + \delta' = \alpha' \\ |\beta'| = i_1 \\ |\gamma'| = i_2 \\ |\delta'| = i_3}} \Bigg( \|\p^{\beta'}\tildA\|_{H^2} \, \|\p^{\gamma'}\tildA\|_{H^2} \, \|\p^{\delta'} Q^\pm\|_{3,\pm} \Bigg).
\end{equation}
Next, we use the following identity to simplify the second sum in \eqref{s2-estim_F2_1}:
\begin{align*}
\sum_{\substack{\beta' + \gamma' + \delta' = \alpha' \\ |\beta'| = i_1, \, |\gamma'| = i_2, \, |\delta'| = i_3}} \frac{\alpha'!}{\beta'! \, \gamma'! \, \delta'!} \, = \, \frac{n!}{i_1!\, i_2!\, i_3!}.
\end{align*}
Then, bounding from above the maximum of the products by the product of the maximums, we get:
\begin{align*}
\| F_2 \|_{1,\pm} \lesssim \sum_{\substack{i_1 + i_2 + i_3 = n \\ i_1 + i_2 \geq 1}} \frac{n!}{i_1!\, i_2!\, i_3!} \, \max_{|\beta'|=i_1} \|\p^{\beta'}\tildA\|_{H^2} \, \max_{|\gamma'|=i_2} \|\p^{\gamma'}\tildA\|_{H^2} \, \max_{|\delta'|=i_3} \|\p^{\delta'} Q^\pm\|_{3,\pm}.
\end{align*}
Now we have to use that $i_3$ can not be larger than $n-1$. So, we rewrite the previous sum using a summation over the values of $i_1 + i_2$:
\begin{align}
\frac{1}{n!} \|F_2\|_{1,\pm} & \lesssim \sum_{i=1}^n \sum_{i_1+i_2=i} \frac{1}{i_1!} \, \max_{|\beta'|=i_1} \|\p^{\beta'}\tildA\|_{H^2} \, \frac{1}{i_2!} \, \max_{|\gamma'|=i_2} \|\p^{\gamma'}\tildA\|_{H^2} \, \frac{1}{(n-i)!} \, \max_{|\delta'|=n-i} \|\p^{\delta'} Q^\pm\|_{3,\pm}  \nonumber \\
& \lesssim \sum_{i=0}^{n-1} \sum_{i_1+i_2=i+1} \frac{1}{i_1!} \, \max_{|\beta'|=i_1} \|\p^{\beta'}\tildA\|_{H^2} \, \frac{1}{i_2!} \, \max_{|\gamma'|=i_2} \|\p^{\gamma'}\tildA\|_{H^2} \, \frac{1}{(n-1-i)!} \, \max_{|\delta'|=n-1-i} \|\p^{\delta'} Q^\pm\|_{3,\pm}, \label{s2-estim_F2_2}
\end{align}
where the sum in \eqref{s2-estim_F2_2} has been re-indexed over $i$. For the sake of clarity, we rewrite the sum $\sum_{i_1+i_2=i+1}$ using only one index $i_1 \in \{0, \dots, i+1\}$:
\begin{align}\label{s2-estim_F2_3}
\frac{1}{n!} \|F_2\|_{1,\pm} \lesssim \sum_{i=0}^{n-1} \sum_{i_1 = 0}^{i+1} \frac{1}{i_1!} \, \max_{|\beta'|=i_1} \|\p^{\beta'}\tildA\|_{H^2} \, \frac{1}{(i+1-i_1)!} & \, \max_{|\gamma'|=i+1-i_1} \|\p^{\gamma'}\tildA\|_{H^2} \\
&\times \frac{1}{(n-1-i)!} \, \max_{|\delta'|=n-1-i} \|\p^{\delta'} Q^\pm\|_{3,\pm}. \nonumber
\end{align}
We recall that we have to distinguish the cases where we can exhibit the norms $\|\tildA\|_{H^2}$ because of their smallness. So, in estimate  \eqref{s2-estim_F2_3}, we deal with the case $i=0$ for which we have the norm $\|\tildA\|_{H^2}$, and the case $i\geq 1$ which will give some analytic norms. The latter will be handled thereafter using the radius of analyticity $\rho>0$ as a small parameter. We rewrite the double sum in \eqref{s2-estim_F2_3} as $S_{i=0} + S_{i\geq 1}$. The term $S_{i=0}$ corresponds to the case $i=0$, and the term $S_{i\geq 1}$ to the case where we sum over $i\geq 1$.

\paragraph*{\textbf{$\blacktriangleright$ Treatment of $S_{i=0}$.}}

$\,$

\noindent
We directly have the following estimate:
\begin{align}
S_{i=0} & = 2\|\tildA\|_{H^2} \, \max_{|\gamma'|=1} \|\p^{\gamma'}\tildA\|_{H^2} \, \frac{1}{(n-1)!} \, \max_{|\delta'|=n-1} \|\p^{\delta'} Q^\pm\|_{3,\pm} \nonumber\\
& \lesssim \|\tildA\|_{H^2} \, \|\tildA\|_{H^3} \, \frac{1}{(n-1)!} \, \max_{|\delta'|=n-1} \|\p^{\delta'} Q^\pm\|_{3,\pm}. \label{s2-estim_F2_9}
\end{align}

\paragraph*{\textbf{$\blacktriangleright$ Treatment of $S_{i\geq 1}$.}}

\paragraph*{\textbf{Remark.}} If $n=1$, the term $S_{i\geq 1}$ does not appear in the right-hand side of \eqref{s2-estim_F2_3}. Consequently, the case $n=1$ is completely treated using estimate \eqref{s2-estim_F2_9}. So, we can assume in the following that $n\geq 2$.

\noindent
Re-indexing over $i$, we can write:
\begin{align*}
S_{i\geq 1} = \sum_{i=0}^{n-2} \sum_{i_1=0}^{i+2} \frac{1}{i_1!} \, \max_{|\beta'|=i_1} \|\p^{\beta'}\tildA\|_{H^2} \, \frac{1}{(i+2-i_1)!} \, \max_{|\gamma'|=i+2-i_1} \|\p^{\gamma'}\tildA\|_{H^2} \, \frac{1}{(n-2-i)!} \, \max_{|\delta'|=n-2-i} \|\p^{\delta'} Q^\pm\|_{3,\pm}.
\end{align*}
Once more, we isolate the case $i_1\in\{0, i+2\}$, which gives the norm $\|\tildA\|_{H^2}$. Thus we get:
\begin{align}
S_{i\geq 1} & = 2 \|\tildA\|_{H^2} \sum_{i=0}^{n-2} \frac{1}{(i+2)!} \, \max_{|\gamma'|=i+2} \|\p^{\gamma'}\tildA\|_{H^2} \, \frac{1}{(n-2-i)!} \, \max_{|\delta'|=n-2-i} \|\p^{\delta'} Q^\pm\|_{3,\pm} \label{s2-estim_F2_4}\\
& \quad + \sum_{i=0}^{n-2} \sum_{i_1=1}^{i+1}  \frac{1}{i_1!} \, \max_{|\beta'|=i_1} \|\p^{\beta'}\tildA\|_{H^2} \, \frac{1}{(i+2-i_1)!} \, \max_{|\gamma'|=i+2-i_1} \|\p^{\gamma'}\tildA\|_{H^2} \, \frac{1}{(n-2-i)!} \, \max_{|\delta'|=n-2-i} \|\p^{\delta'} Q^\pm\|_{3,\pm}. \label{s2-estim_F2_5}
\end{align}
We shall note $S_{i_1=0,i+2}$ (resp.  $S_{1\leq i_1\leq i+1}$) the right term of \eqref{s2-estim_F2_4} (resp. \eqref{s2-estim_F2_5}).
\begin{itemize}[label=$\blacklozenge$]
\item \textbf{Treatment of $S_{i_1=0,i+2}$.}

$\,$

\noindent
Using both following trivial estimates
\begin{align}\label{s2-ineg_perte_der_H2-H3}
\max_{|\gamma'|=i+2} \|\p^{\gamma'}\tildA\|_{H^2} \leq \max_{|\gamma'|=i+1} \|\p^{\gamma'}\tildA\|_{H^3} \quad \text{ and } \quad (i+2)!\geq (i+1)!,
\end{align}
and re-indexing over $i$, we obtain:
\begin{align}
S_{i_1=0,i+2} \lesssim \|\tildA\|_{H^2} \sum_{i=1}^{n-1} \frac{1}{i!} \, \max_{|\gamma'|=i} \|\p^{\gamma'}\tildA\|_{H^3} \, \frac{1}{(n-1-i)!} \, \max_{|\delta'|=n-1-i} \|\p^{\delta'} Q^\pm\|_{3,\pm}. \label{s2-estim_F2_6}
\end{align}
Eventually, estimate \eqref{s2-estim_F2_6} is rewritten as follows:
\begin{align}\label{s2-estim_F2_7}
S_{i_1=0,i+2} \lesssim \|\tildA\|_{H^2} \sum_{i_1+i_2 = n-1} \frac{1}{i_1!} \, \max_{|\gamma'|=i_1} \|\p^{\gamma'}\tildA\|_{H^3} \, \frac{1}{i_2!} \, \max_{|\delta'|=i_2} \|\p^{\delta'} Q^\pm\|_{3,\pm}.
\end{align}
This concludes the case of $S_{i_1=0,i+2}$.
\item \textbf{Treatment of $S_{1\leq i_1\leq i+1}$.}

$\,$

\noindent
Re-indexing over $i_1$, we have:
\begin{align*}
S_{1\leq i_1\leq i+1} = \sum_{i=0}^{n-2} \sum_{i_1=0}^i \frac{1}{(i_1+1)!} \, \max_{|\beta'|=i_1+1} \|\p^{\beta'}\tildA\|_{H^2} \, & \frac{1}{(i+1-i_1)!} \, \max_{|\gamma'|=i+1-i_1} \|\p^{\gamma'}\tildA\|_{H^2} \\
 \times & \frac{1}{(n-2-i)!} \, \max_{|\delta'|=n-2-i} \|\p^{\delta'} Q^\pm\|_{3,\pm}.
\end{align*}
Using the same kind of inequalities given by \eqref{s2-ineg_perte_der_H2-H3}, we deduce that
\begin{align*}
S_{1\leq i_1\leq i+1} \lesssim \sum_{i=0}^{n-2} \sum_{i_1=0}^i \frac{1}{i_1!} \, \max_{|\beta'|=i_1} \|\p^{\beta'}\tildA\|_{H^3} \, \frac{1}{(i-i_1)!} \, \max_{|\gamma'|=i-i_1} \|\p^{\gamma'}\tildA\|_{H^3} \, \frac{1}{(n-2-i)!} \, \max_{|\delta'|=n-2-i} \|\p^{\delta'} Q^\pm\|_{3,\pm}.
\end{align*}
Finally, rewriting the sum over $i_1$ with two indices $i_1$ and $i_2$ such that $i_1+i_2=i$, we end up with:
\begin{align}\label{s2-estim_F2_8}
S_{1\leq i_1\leq i+1} & \, \lesssim \, \sum_{i=0}^{n-2} \sum_{i_1+i_2=i} \frac{1}{i_1!} \, \max_{|\beta'|=i_1} \|\p^{\beta'}\tildA\|_{H^3} \, \frac{1}{i_2!} \, \max_{|\gamma'|=i_2} \|\p^{\gamma'}\tildA\|_{H^3} \, \frac{1}{(n-2-i)!} \, \max_{|\delta'|=n-2-i} \|\p^{\delta'} Q^\pm\|_{3,\pm} \nonumber\\
& \, \lesssim \, \sum_{i_1+i_2+i_3=n-2} \frac{1}{i_1!} \, \max_{|\beta'|=i_1} \|\p^{\beta'}\tildA\|_{H^3} \, \frac{1}{i_2!} \, \max_{|\gamma'|=i_2} \|\p^{\gamma'}\tildA\|_{H^3} \, \frac{1}{i_3!} \, \max_{|\delta'|=i_3} \|\p^{\delta'} Q^\pm\|_{3,\pm},
\end{align}
which concludes the estimate of $S_{1\leq i_1\leq i+1}$.
\end{itemize}
Let us come back to estimate \eqref{s2-estim_F2_3} of the commutator $F_2$. We use \eqref{s2-estim_F2_9} together with \eqref{s2-estim_F2_7} and \eqref{s2-estim_F2_8} to get, for $|\alpha'| = n \geq 2$:
\begin{align}
\frac{\rho^n}{n!} \max_{|\alpha'|=n} \|F_2\|_{1,\pm} & \, \lesssim \, \rho \, \|\tildA\|_{H^2} \, \|\tildA\|_{H^3} \, \frac{\rho^{n-1}}{(n-1)!} \, \max_{|\delta'| = n-1} \|\p^{\delta'} Q^\pm\|_{3,\pm} \label{s2-estim_F2_12}\\
& \quad + \rho\, \|\tildA\|_{H^2} \sum_{i_1+i_2=n-1} \frac{\rho^{i_1}}{i_1!} \, \max_{|\gamma'|=i_1} \|\p^{\gamma'}\tildA\|_{H^3} \, \frac{\rho^{i_2}}{i_2!} \, \max_{|\delta'|=i_2} \|\p^{\delta'} Q^\pm\|_{3,\pm} \label{s2-estim_F2_10}\\
& \quad + \rho^2 \, \sum_{i_1+i_2+i_3=n-2} \frac{\rho^{i_1}}{i_1!} \, \max_{|\beta'|=i_1} \|\p^{\gamma'}\tildA\|_{H^3} \, \frac{\rho^{i_2}}{i_2!} \, \max_{|\gamma'|=i_2} \|\p^{\gamma'}\tildA\|_{H^3} \, \frac{\rho^{i_3}}{i_3!} \, \max_{|\delta'|=i_3} \|\p^{\delta'} Q^\pm\|_{3,\pm}. \label{s2-estim_F2_11}
\end{align}
If $n=1$, we recall that we have the simpler estimate:
\begin{align*}
\frac{\rho^n}{n!} \max_{|\alpha'|=n} \|F_2\|_{1,\pm} & \, \lesssim \, \rho \, \|\tildA\|_{H^2} \, \|\tildA\|_{H^3} \, \frac{\rho^{n-1}}{(n-1)!} \, \max_{|\delta'| = n-1} \|\p^{\delta'} Q^\pm\|_{3,\pm},
\end{align*}
which turns out to be the same as \eqref{s2-estim_F2_12}. In the sums appearing in \eqref{s2-estim_F2_10} and \eqref{s2-estim_F2_11}, we recognize ``partial'' Cauchy products. Thus, summing over $n\in\{1,\dots,N\}$ and using the straightforward inequality $\|\tildA\|_{H^3}\leq\|\tildA\|_{\rho,3}^{0}$, we obtain:
\begin{align*}
\sum_{n=1}^N \frac{\rho^n}{n!} \max_{|\alpha'|=n} \|F_2\|_{1,\pm} \, \lesssim \, \rho \, \|\tildA\|_{H^2} \, \|\tildA\|_{\rho,3}^{0} \|Q^\pm\|_{\rho,3}^{0,\pm,N-1} & + \, \rho \, \|\tildA\|_{H^2} \, \|\tildA\|_{\rho,3}^{0,N-1} \|Q^\pm\|_{\rho,3}^{0,\pm,N-1} \\
&  + \rho^2 \, \|\tildA\|_{\rho,3}^{0,N-2} \, \|\tildA\|_{\rho,3}^{0,N-2} \, \|Q^\pm\|_{\rho,3}^{0,\pm,N-2}.
\end{align*}
Since the coefficients of $\tildA$ belong to the space $B_{\rho,3,\sigma}(\Omega)$, we can bound from above the partial sums  $\|\tildA\|_{\rho,3}^{0,M}$ by the finite quantity $\|\tildA\|_{\rho,3}^{0}$. Eventually, the commutator $F_2$ satisfies the estimate:
\begin{align}\label{s2-estim_commut_F2}
\sum_{n=1}^N \frac{\rho^n}{n!} \max_{|\alpha'|=n} \|F_2\|_{1,\pm} \, \lesssim \, \rho\, \|\tildA\|_{H^2} \, \|\tildA\|_{\rho,3}^0 \, \|Q^\pm\|_{\rho,3}^{0,\pm,N} \, + \, \left( \rho \, \|\tildA\|_{\rho,3}^0 \right)^2 \, \|Q^\pm\|_{\rho,3}^{0,\pm,N}.
\end{align}
This concludes the case of the first commutator $F_2$.

Let us move on to the second commutator $F_3$ given by (see \eqref{s2-def_F'}):
\begin{equation}\label{s2-def_F3}
F_3 := \big[ \p^{\alpha'} \, ; \, \p_i\tildA_{ki}\p_k\cdot \big] Q^\pm = \sum_{\substack{\beta' +\gamma' = \alpha' \\ |\beta'|\geq 1}} \frac{\alpha'!}{\beta'! \, \gamma'!} \, \p^{\beta'}\p_i\tildA_{ki} \, \p^{\gamma'}\p_k Q^\pm.
\end{equation}
We estimate $F_3$ with the same tools used for $F_2$. First of all, considering the derivatives of order 0 and 1, we show that
\begin{align}\label{s2-estim_F3_1}
\| \p^{\beta'}\p_i\tildA_{ki} \, \p^{\gamma'}\p_k Q^\pm \|_{1,\pm} \, \lesssim \, \|\p^{\beta'}\tildA\|_{H^2} \, \|\p^{\gamma'} Q^\pm\|_{3,\pm}.
\end{align}
It leads to the following estimate ($n\geq 1$):
\begin{align*}
\frac{1}{n!} \max_{|\alpha'|=n} \|F_3\|_{1,\pm} \, \lesssim \, \sum_{i_1=1}^n \frac{1}{i_1!} \, \max_{|\beta'|=i_1} \|\p^{\beta'}\tildA\|_{H^2} \, \frac{1}{(n-i_1)!} \, \max_{|\gamma'| = n-i_1} \|\p^{\gamma'} Q^\pm\|_{3,\pm}.
\end{align*}
Using similar inequalities as in \eqref{s2-ineg_perte_der_H2-H3}, we can write
\begin{align*}
\frac{1}{n!} \max_{|\alpha'|=n} \|F_3\|_{1,\pm} \, \lesssim \, \sum_{i_1 + i_2 = n-1} \frac{1}{i_1!} \, \max_{|\beta'|=i_1} \|\p^{\beta'}\tildA\|_{H^3} \, \frac{1}{i_2!} \, \max_{|\gamma'| = i_2} \|\p^{\gamma'} Q^\pm\|_{3,\pm}.
\end{align*}
Summing over $n\in\{1,\dots,N\}$, we end up with:
\begin{align}\label{s2-estim_commut_F3}
\sum_{n=1}^N \frac{\rho^n}{n!} \, \max_{|\alpha'|=n} \|F_3\|_{1,\pm} \, \lesssim \, \rho \, \|\tildA\|_{\rho,3}^0 \, \|Q^\pm\|_{\rho,3}^{0,\pm,N}.
\end{align}

Let us finish with the last commutator $F_4$ of $\cF'$:
\begin{align}\label{s2-def_F4}
F_4 := \big[ \p^{\alpha'} \, ; \, \tildA_{ji}\p_{ji}\cdot \big] Q^\pm = \sum_{\substack{\beta' + \gamma' = \alpha' \\ |\beta'|\geq 1}} \frac{\alpha!}{\beta'! \, \gamma'!} \, \p^{\beta'}\tildA_{ki} \, \p^{\gamma'}\p_{ki}Q^\pm.
\end{align}
In the same spirit as \eqref{s2-estim_F3_1}, we show that
\begin{align*}
\| \p^{\beta'}\tildA_{ki} \, \p^{\gamma'}\p_{ki}Q^\pm \|_{1,\pm} \, \lesssim \, \|\p^{\beta'}\tildA\|_{H^2} \, \|\p^{\gamma'} Q^\pm\|_{3,\pm}.
\end{align*}
This estimate is the same as \eqref{s2-estim_F3_1}. Consequently, we also conclude that
\begin{align}\label{s2-estim_commut_F4}
\sum_{n=1}^N \frac{\rho^n}{n!} \, \max_{|\alpha'|=n} \|F_4\|_{1,\pm} \, \lesssim \, \rho \, \|\tildA\|_{\rho,3}^0 \, \|Q^\pm\|_{\rho,3}^{0,\pm,N}.
\end{align}

This achieves the estimate of the source term $\cF'^\pm$. Let us sum up below the final estimate satisfied by $\cF'^\pm$, combining \eqref{s2-estim_F1}, \eqref{s2-estim_commut_F2}, \eqref{s2-estim_commut_F3} and \eqref{s2-estim_commut_F4}:
\begin{align}
\sum_{n=1}^N \frac{\rho^n}{n!} \max_{|\alpha'|=n} \| \cF'^\pm \|_{1,\pm} \, \lesssim \, \|\cF^\pm\|_{\rho,1}^0 \, & + \rho \, \|\tildA\|_{\rho,3}^0 \, \left( 1 + \|\tildA\|_{H^2} \right) \, \|Q^\pm\|_{\rho,3}^{0,\pm,N} \nonumber \\
& + \, \left( \rho \, \|\tildA\|_{\rho,3}^0 \right)^2 \, \|Q^\pm\|_{\rho,3}^{0,\pm,N}. \label{s2-estim_F'}
\end{align}
Now, we shall proceed in the same way to estimate the boundary source term $\cG'$.

\medskip

\paragraph*{\textbf{Estimate of $\cG'$.}}

$\,$

We recall that $\cG'$ is defined by \eqref{s2-def_G'}. The case of $G_1$ is identical to the interior source term $F_1$. We obtain the same type of estimate as \eqref{s2-estim_F1}, namely:
\begin{align}\label{s2-estim_G1}
\sum_{n=0}^N \frac{\rho^n}{n!} \max_{|\alpha'|=n} \|G_1\|_{H^\frac{3}{2}(\bT^2)} \, = \, \sum_{n=0}^N \frac{\rho^n}{n!} \max_{|\alpha'|=n} \|\p^{\alpha'}\cG\|_{H^\frac{3}{2}(\bT^2)} \, \leq \, \|\cG\|_{\rho,\frac{3}{2}}.
\end{align}

Now, we treat the commutator $G_2$. To do so, we use the continuity of the trace map from $H^2(\Omega^\pm)$ to $H^\frac{3}{2}(\bT^2)$ (see \eqref{art-estim_trace}), in order to eliminate the jump across $\Gamma$. Then, we also use Proposition \ref{s2-prop_redressement_front} to make $\psi$ appear in the estimates:
\begin{align*}
\| G_2 \|_{H^\frac{3}{2}(\bT^2)} = \left\|\left[\big[ \p^{\alpha'} \, ; \, (1+|\nabla' f|^2)\p_3 \big] Q \right]\right\|_{H^\frac{3}{2}(\bT^2)} \, \lesssim \, \left\| \big[ \p^{\alpha'} \, ; \, (1+\p_h\psi\p_h\psi)\p_3 \big] Q^\pm \right\|_{2,\pm},
\end{align*}
where we have adopted Einstein's summation convention over the index $h\in\{1,2\}$. We set
\begin{align*}
G_2^\pm := \big[ \p^{\alpha'} \, ; \, (1+\p_h\psi\p_h\psi)\p_3 \big] Q^\pm,
\end{align*}
and we expand it to obtain:
\begin{align*}
G_2^\pm \, = \, \sum_{\substack{\beta'+\gamma'+\delta'=\alpha' \\ |\beta'+\gamma'|\geq 1}} \frac{\alpha'!}{\beta'! \, \gamma'! \, \delta'!} \, \p^{\beta'}\p_h\psi \, \p^{\gamma'}\p_h\psi \, \p^{\delta'}\p_3 Q^\pm.
\end{align*}
Using the algebra property of the space $H^2(\Omega^\pm)$, we have
\begin{align*}
\|G_2 \|_{H^\frac{3}{2}(\bT^2)} \, \lesssim \, \sum_{\substack{\beta'+\gamma'+\delta'=\alpha' \\ |\beta'+\gamma'|\geq 1}} \frac{\alpha'!}{\beta'! \, \gamma'! \, \delta'!} \, \|\p^{\beta'}\psi\|_{H^3} \, \|\p^{\gamma'}\psi\|_{H^3} \, \|\p^{\delta'} Q^\pm\|_{3,\pm}.
\end{align*}
The computations are exactly the same as previously for the commutator $F_2$ defined by \eqref{s2-def_F2}. Here, $\psi$ plays the role of $\tildA$, whose tangential derivatives are estimated in $H^3$ instead of $H^2$ for $\tildA$. It remains consistent, because the definition of $\tildA$ only gives $1-$order derivatives of $\psi$. Thus the functions $\tildA$ and $\psi$ are estimated in the same Sobolev scale. We get the following estimate for $G_2$ (analogous to \eqref{s2-estim_commut_F2}):
\begin{align}\label{s2-estim_commut_G2}
\sum_{n=2}^N \frac{\rho^n}{n!} \max_{|\alpha'|=n} \| G_2 \|_{H^\frac{3}{2}(\bT^2)} \, \lesssim \, \rho \|\psi\|_{H^3} \, \|\psi\|_{\rho,4}^0 \, \|Q^\pm\|_{\rho,3}^{0,\pm,N} \, + \, \Big( \rho \, \|\psi\|_{\rho,4}^0 \Big)^2 \, 
 \|Q^\pm\|_{\rho,3}^{0,\pm,N}.
\end{align}
We notice that the norm $\|\psi\|_{H^3}$ as well as $\|\tildA\|_{H^2}$ will be small quantities, because of the flatness of the front $f$ in $H^\frac{5}{2}(\bT^2)$. Below, we sum up the final estimate of $\cG'$, combining \eqref{s2-estim_G1} together with \eqref{s2-estim_commut_G2}:
\begin{align}
\sum_{n=1}^N \frac{\rho^n}{n!} \max_{|\alpha'|=n} \|\cG'\|_{H^\frac{3}{2}(\bT^2)} \, \lesssim \, \|\cG\|_{\rho,\frac{3}{2}} \, + \, \rho \, \|\psi\|_{H^3} \, \|\psi\|_{\rho,4}^0 \, \|Q^\pm\|_{\rho,3}^{0,\pm,N} \, + \, 
\Big( \rho \, \|\psi\|_{\rho,4}^0 \Big)^2 \, \|Q^\pm\|_{\rho,3}^{0,\pm,N}. \label{s2-estim_G'}
\end{align}

\medskip

\paragraph*{\textbf{Estimate of $Q'^\pm$.}}

$\,$

Let us come back to estimate \eqref{s2-estim_som_part_Q'}. Using inequalities \eqref{s2-estim_F'} and \eqref{s2-estim_G'} holding for $n\geq 1$, and \eqref{s2-estim_Q_n=0} holding for $n=0$, we deduce that:
\begin{align}\label{s2-estim_Q_rho,3^0,N}
\|Q^\pm\|_{\rho,3}^{0,\pm,N} & \, \lesssim \, \|\cF^\pm\|_{\rho,1}^{0,\pm} \, + \, \|\cG\|_{\rho,\frac{3}{2}} \nonumber \\
& \quad + \, \rho \left( \big( 1 + \|\tildA\|_{H^2} \big) \|\tildA\|_{\rho,3}^0 \, + \, \|\psi\|_{H^3} \, \|\psi\|_{\rho,4}^0 \right) \|Q^\pm\|_{\rho,3}^{0,\pm,N} \\
& \quad + \, \rho^2 \left( \|\tildA\|_{\rho,3}^0 \, \|\tildA\|_{\rho,3}^0 \, + \, \|\psi\|_{\rho,4}^0 \, \|\psi\|_{\rho,4}^0 \right) \|Q^\pm\|_{\rho,3}^{0,\pm,N}. \nonumber
\end{align}

\medskip

\paragraph*{\textbf{Conclusion.}}

$\,$

Now we complete estimate \eqref{s2-estim_Q_rho,3^0,N} absorbing the analytic terms on the right-hand side. We consider $f\in B_{\rho,\frac{7}{2}}(\bT^2)$ satisfying the following smallness condition:
\begin{align}\label{s2-hyp_petitesse_front}
\|f\|_{\rho,\frac{7}{2}} < \eta_0, 
\end{align}
where the numerical constant $\eta_0>0$ is given by Corollary \ref{s2-coro_estim_1/J}. Decreasing $\eta_0$ if necessary, we can assume that we have the range $\frac{1}{2}\leq J\leq\frac{3}{2}$ and the interior regularity estimate \eqref{s2-reg_ellip_H^3} (see \cite[p.250, 271]{CMST}). From now on, $C_0>0$ will denote any numerical constant that may depend on $\eta_0$, but not on the parameters $\rho$ and $\sigma$. Using assumption \eqref{s2-hyp_petitesse_front}, we have the inequality $\ds\|J^{-1}\|_{H^2}\leq C_0$. Therefore the algebra property of $H^2(\Omega)$ together with Lemma \ref{s2-lemme_redressement_CMST} give:
\begin{align*}
\|\tildA\|_{H^2} \leq C_0.
\end{align*}
Then we estimate the analytic norms in the following way:
\begin{align*}
\|\tildA\|_{\rho,3}^0 \, = \, \|J^{-1} \nabla\psi \|_{\rho,3}^0 \, & \leq \, \|J^{-1} \nabla\psi \|_{\rho,3,\sigma} \\
& \leq \, C_0 \, \|J^{-1}\|_{\rho,3,\sigma} \, \|\psi\|_{\rho,4,\sigma} & \text{ (algebra property of } B_{\rho,3,\sigma}(\Omega) ) \\
& \leq \, C_0 \, \|J^{-1}\|_{\rho,3,\sigma} \, \|f\|_{\rho,\frac{7}{2}} & \text{ (Proposition \ref{s2-prop_redressement_front}) } \\
& \leq \, C_0. & \text{ (Corollary \ref{s2-coro_estim_1/J} and assumption \eqref{s2-hyp_petitesse_front}) }
\end{align*}
So, we have first:
\begin{align*}
\rho \left( \big( 1 + \|\tildA\|_{H^2} \big) \|\tildA\|_{\rho,3}^0 \, + \, \|\psi\|_{H^3} \, \|\psi\|_{\rho,4}^0 \right) \, \leq \, C_0 \, \rho.
\end{align*}
Then, we also get:
\begin{align*}
\rho^2 \left( \|\tildA\|_{\rho,3}^0 \, \|\tildA\|_{\rho,3}^0 \, + \, \|\psi\|_{\rho,4}^0 \, \|\psi\|_{\rho,4}^0 \right) \, \leq \, 
C_0 \, \rho^2.
\end{align*}
For instance, let us define
\begin{align}\label{s2-def_rho_0}
\rho_0 := \min\left\{1, \frac{1}{2 C_0}\right\} \in (0,1],
\end{align}
so that for all $\rho\in (0,\rho_0]$, we eventually get:
\begin{align}\label{s2-estim_Q_rho,3^0,N_absorbe}
\|Q^\pm\|_{\rho,3}^{0,\pm,N} \, \lesssim \, \|\cF^\pm\|_{\rho,1}^{0,\pm} + \|\cG\|_{\rho,\frac{3}{2}}.
\end{align}
Taking the supremum over $N$, we obtain $Q^\pm \in B_{\rho,3}^0(\Omega^\pm)$ and the same estimate as above for the norm $\ds\|Q^\pm\|_{\rho,3}^{0,\pm}$.

We summarize the main result of this paragraph in the following theorem.
\begin{thm}\label{s2-thm_estim_Q_rho,3^0}
There exists $\eta_0 > 0$ and $\rho_0 \in (0,1]$ such that, for all $\rho\in (0,\rho_0]$ and $\sigma\in (0,\frac{1}{2}]$, if we consider a front $f\in B_{\rho,\frac{7}{2}}(\bT^2)$ satisfying
\begin{align}\label{s2-hyp_f_petit_thm_Q_tang}
\|f\|_{\rho,\frac{7}{2}} < \eta_0,
\end{align}
and analytic source terms $\cF^\pm\in B_{\rho,1,\sigma}(\Omega^\pm)$ and $\cG\in B_{\rho,\frac{3}{2}}(\bT^2)$, then the unique solution $(Q^+,Q^-)$ to the elliptic problem \eqref{s2-pb_ellip_pression}, with zero mean on $\Omega$, belongs to the space $B_{\rho,3}^0(\Omega^\pm)$; furthermore, it satisfies the following estimate:
\begin{align}\label{s2-estim_Q_rho,3^0}
\|Q^\pm\|_{\rho,3}^{0,\pm} \, \leq \, C(\eta_0) \, \Big( \|\cF^\pm\|_{\rho,1}^{0,\pm} + \|\cG\|_{\rho,\frac{3}{2}} \Big),
\end{align}
where $C(\eta_0)>0$ is a constant depending only on $\eta_0$.
\end{thm}
Let us observe that we have already admitted that the elliptic problem \eqref{s2-pb_ellip_pression} has a solution $Q^\pm$ in $H^\infty(\Omega^\pm)$. Theorem \ref{s2-thm_estim_Q_rho,3^0} gives an additional information about the tangential derivatives of $Q^\pm$ when the source terms and the coefficients of the elliptic operator are in appropriate analytic spaces. The purpose of the next paragraph is to handle all the remaining normal derivatives.

\subsection{Estimate of the normal derivatives}

Within this paragraph, we still consider analytic source terms \mbox{$\cF^\pm\in B_{\rho,1,\sigma}(\Omega^\pm)$} and $\cG\in B_{\rho,\frac{3}{2}}(\bT^2)$, together with a front $f\in B_{\rho,\frac{7}{2}}(\bT^2)$ satisfying \eqref{s2-hyp_f_petit_thm_Q_tang}. From now on, the purpose is to show that the zero mean solution $Q^\pm$ of problem \eqref{s2-pb_ellip_pression} belongs to the space $B_{\rho,3,\sigma}(\Omega^\pm)$, for $\sigma >0$ small enough. To do so, we first have to handle all the normal derivatives in order to show that for all $k\geq 1$, we have $Q^\pm\in B_{\rho,3}^k(\Omega^\pm)$. Eventually, to prove that $Q^\pm\in B_{\rho,3,\sigma}(\Omega^\pm)$, it will remain to estimate the partial sums $\|Q^\pm\|_{\rho,3,\sigma}^{\pm,N}$.

We proceed by induction over $k\geq 0$. Let us assume that for all $j\in\{0,\dots,k\}$, we have $Q^\pm\in B_{\rho,3}^j(\Omega^\pm)$. First, we know that $Q^\pm\in H^\infty(\Omega^\pm)$, so that we have $Q^\pm\in H_{x_3}^{k+4}(H^\infty(\bT^2))$. It remains to prove that the norm $\|Q^\pm\|_{\rho,3}^{k+1,\pm}$ is a finite quantity. For all $N\geq 0$, we shall write $\|Q^\pm\|_{\rho,3}^{k+1,\pm,N}$ its partial sum of order $N$. Using the induction assumption, we split the partial sum as follows (we assume that $N\geq k+1$):
\begin{align}
\|Q^\pm\|_{\rho,3}^{k+1,\pm,N} & := \sum_{k=0}^N \frac{\rho^n}{n!} \, \max_{\substack{|\alpha|=n \\ \alpha_3\leq k+1}} \|\p^\alpha Q^\pm \|_{3,\pm} \nonumber \\
& \leq \|Q^\pm\|_{\rho,3}^{k,\pm} \, + \, \sum_{n=k+1}^N \frac{\rho^n}{n!} \, \max_{|\alpha'|=n-k-1} \|\p^{\alpha'}\p_3^{k+1} Q^\pm\|_{3,\pm}. \label{s2-ineg_rec_Q_rho,3^k+1,N}
\end{align}
In order to estimate the right sum in \eqref{s2-ineg_rec_Q_rho,3^k+1,N}, we rewrite the norm $\ds\|\p^{\alpha'}\p_3^{k+1} Q^\pm\|_{3,\pm}$ as follows:
\begin{align}\label{s2-norm_H^3_eclatee}
\|\p^{\alpha'}\p_3^{k+1} Q^\pm\|_{3,\pm}^2 = \|\p^{\alpha'}\p_3^{k+1} Q^\pm\|_{2,\pm}^2 \, + \, \sum_{\substack{|\beta|=3 \\ |\beta'|\geq 1}} \|\p^\beta\p^{\alpha'}\p_3^{k+1} Q^\pm\|_{0,\pm}^2 \, + \, \|\p^{\alpha'}\p_3^{k+4} Q^\pm\|_{0,\pm}^2.
\end{align}
Both first right terms of \eqref{s2-norm_H^3_eclatee} will allow to get the quantity $\|Q^\pm\|_{\rho,3}^{k,\pm}$. On the one hand, we begin by writing the estimate
\begin{align*}
\|\p^{\alpha'}\p_3^{k+1} Q^\pm\|_{2,\pm} \leq \|\p^{\alpha'}\p_3^{k} Q^\pm\|_{3,\pm} \leq \max_{|\alpha'|=n-k-1} \|\p^{\alpha'}\p_3^k Q^\pm\|_{3,\pm}.
\end{align*}
Thus summing over $n$ we have
\begin{align*}
\sum_{n=k+1}^N \frac{\rho^n}{n!} \, \max_{|\alpha'|=n-k-1} \|\p^{\alpha'}\p_3^{k+1} Q^\pm\|_{2,\pm} 
\leq \sum_{n=k+1}^N \frac{\rho^n}{n!} \, \max_{|\alpha'|=n-k-1} \|\p^{\alpha'}\p_3^k Q^\pm\|_{3,\pm} \leq \rho \, \|Q^\pm\|_{\rho,3}^{k,\pm}.
\end{align*}
On the other hand, for $|\beta|=3$ and $|\beta'|\geq 1$, the terms $\ds \|\p^\beta\p^{\alpha'}\p_3^{k+1} Q^\pm\|_{0,\pm}$ of \eqref{s2-norm_H^3_eclatee} give the same kind of estimate. Indeed, since $\beta$ is a multi-index of length 3, with at least one tangential  derivative $\bp$, we can write $\p^\beta = \bp\p^2$. Therefore we get:
\begin{align*}
\|\p^\beta\p^{\alpha'}\p_3^{k+1} Q^\pm\|_{0,\pm} = \|\bp\p^2 \, (\p^{\alpha'}\p_3^{k+1} Q^\pm)\|_{0,\pm} \leq \|\bp\p^{\alpha'}\p_3^k Q^\pm\|_{3,\pm} \leq \max_{|\alpha'|=n-k} \|\p^{\alpha'}\p_3^k Q^\pm\|_{3,\pm}.
\end{align*}
After summing over $n$, we obtain:
\begin{align*}
\sum_{\substack{|\beta|=3 \\ |\beta'|\geq 1}} \sum_{n=k+1}^N \frac{\rho^n}{n!} \max_{|\alpha'|=n-k-1} \|\p^\beta\p^{\alpha'}\p_3^{k+1} Q^\pm\|_{0,\pm} \, \lesssim \, \sum_{n=k+1}^N \frac{\rho^n}{n!} \, \max_{|\alpha'|=n-k} \|\p^{\alpha'}\p_3^k Q^\pm\|_{3,\pm} \, \lesssim \, \|Q^\pm\|_{\rho,3}^{k,\pm}.
\end{align*}
Back to \eqref{s2-ineg_rec_Q_rho,3^k+1,N}, we deduce that
\begin{align}\label{s2-ineg_rec_Q_rho,3^k+1,N_2}
\|Q^\pm\|_{\rho,3}^{k+1,\pm,N} \, \lesssim \, (1+\rho) \, \|Q^\pm\|_{\rho,3}^{k,\pm} \, + \, \sum_{n=k+1}^N \frac{\rho^n}{n!} \, \max_{|\alpha'|=n-k-1} \|\p^{\alpha'}\p_3^{k+4} Q^\pm\|_{0,\pm}.
\end{align}
The only tricky term to estimate is the one containing the most normal derivatives, namely $\ds \p^{\alpha'}\p_3^{k+4} Q^\pm$. We have to estimate the latter in $L^2(\Omega^\pm)$. So, we begin by using the following bound:
\begin{align*}
\|\p^{\alpha'}\p_3^{k+4} Q^\pm\|_{0,\pm} \leq \|\p^{\alpha'}\p_3^k \, (\p_3^2 Q^\pm)\|_{2,\pm}.
\end{align*}
This highlights exactly 2 normal derivatives of $Q^\pm$, that we will be able to ``eliminate'' using the equation satisfied by $Q^\pm$ in problem \eqref{s2-pb_ellip_pression}. Then, it will suffice to apply the operator $\p^{\alpha'}\p_3^k$ to this equation: it will exhibit at most $k$ normal derivatives. Consequently, using the induction assumption, we shall be able to use all the quantities $\|Q^\pm\|_{\rho,3}^{j,\pm}$, for $j\leq k$. The difference of exactly one normal derivative between the steps $k$ and $k+1$ will allow to gain one power of $\sigma$ later on, after summing over $k$. Eventually, this gain will turn out to be crucial to absorb some analytic norms, in order to complete the estimate of the pressure. This method is similar as the previous paragraph, when the computations of $\ds\|Q^\pm\|_{\rho,3}^{0,\pm}$ gave us some powers of $\rho$.

Expanding the first equation of \eqref{s2-pb_ellip_pression}, we get:
\begin{align}\label{s2-eq_Q}
\frac{1+|\nabla'\psi|^2}{J^2} \, \p_3^2 Q^\pm \, = \, - \Delta' Q^\pm \, + \, 2 \frac{\p_h\psi}{J} \, \p_h\p_3 Q^\pm \, - \, \cF^\pm \, - \, A_{ji} \, \p_j\tildA_{ki} \, \p_k Q^\pm.
\end{align}
Let us define
\begin{align}\label{s2-def_zeta}
\zeta := \frac{J^2}{1+|\nabla'\psi|^2},
\end{align}
so that we can rewrite \eqref{s2-eq_Q} as follows:
\begin{align}\label{s2-eq_Q_zeta}
\p_3^2 Q^\pm \, = \, - \zeta \, \Delta' Q^\pm \, + \, 2 \frac{\zeta}{J} \, \p_h\psi \, \p_h\p_3 Q^\pm \, - \, \zeta \, \cF^\pm \, - \, \zeta \, A \, \p^1\tildA \, \p^1 Q^\pm.
\end{align}
Let us set $\alpha := (\alpha',k)$ (remark that $|\alpha| = n-1$), and apply $\p^\alpha$ to equation \eqref{s2-eq_Q_zeta}:
\begin{align}
\p^\alpha \p_3^2 Q^\pm & \, = \, - \p^\alpha (\zeta \, \Delta' Q^\pm) \, + \, 2 \p^\alpha\left(\frac{\zeta}{J} \, \p_h\psi \, \p_h\p_3 Q^\pm\right) \, - \, \p^\alpha (\zeta \, \cF^\pm) \, - \, \p^\alpha \left(\zeta \, A \, \p^1\tildA \, \p^1 Q^\pm\right) \nonumber \\
& \, =: \, - T_1 + 2 T_2 - T_3 - T_4. \label{s2-eq_Q_zeta_d^alpha}
\end{align}
To finish with, it remains to estimate each term $T_i$ in $H^2(\Omega^\pm)$.

\medskip

\paragraph*{\textbf{$\blacktriangleright$ Treatment of $T_1$.}}

$\,$

\noindent
We have:
\begin{align*}
T_1 := \p^\alpha (\zeta \, \Delta' Q^\pm) = \sum_{\beta + \gamma = \alpha} \frac{\alpha!}{\beta! \, \gamma!} \, \p^\beta\zeta \, \p^\gamma\Delta' Q^\pm.
\end{align*}
As for the case of $\|Q^\pm\|_{\rho,3}^{0,\pm}$, we get:
\begin{align}\label{s2-estim_T1_1}
\| T_1 \|_{2,\pm} \, \lesssim \, \sum_{i_1 + i_2 = n-1} \frac{(n-1)!}{i_1! \, i_2!} \, \max_{\substack{\beta+\gamma=\alpha \\ |\beta|=i_1 \\ |\gamma|=i_2}} \Big( \|\p^\beta\zeta\|_{H^2} \, \|\p^\gamma\Delta' Q^\pm\|_{2,\pm} \Big).
\end{align}
Now, we have to use the fact that $\beta_3 + \gamma_3 = k$. Consequently, we can bound from above the right term of \eqref{s2-estim_T1_1} as follows:
\begin{align}\label{s2-estim_T1_2}
\| T_1 \|_{2,\pm} \, \lesssim \, \sum_{j_1+j_2=k} \, \sum_{i_1+i_2=n-1} \frac{(n-1)!}{i_1! \, i_2!} \, \max_{\substack{|\beta|=i_1 \\ \beta_3\leq j_1}} \|\p^\beta\zeta\|_{H^2} \, \max_{\substack{|\gamma|=i_2 \\ \gamma_3\leq j_2}} \|\p^\gamma\Delta' Q^\pm\|_{2,\pm}.
\end{align}
The estimate of $Q^\pm$ is reduced to $H^3(\Omega^\pm)$ using the following inequality:
\begin{align*}
\max_{\substack{|\gamma|=i_2 \\ \gamma_3\leq j_2}} \|\p^\gamma\Delta' Q^\pm\|_{2,\pm} \lesssim \max_{\substack{|\gamma|=i_2+1 \\ \gamma_3\leq j_2}} \|\p^\gamma Q^\pm\|_{3,\pm}.
\end{align*}
Then, re-indexing over $i_2$, we obtain:
\begin{align*}
\| T_1 \|_{2,\pm} \, \lesssim \, \sum_{j_1+j_2=k} \, \sum_{\substack{i_1+i_2=n \\ i_2\geq 1}} \frac{(n-1)!}{i_1! \, (i_2-1)!} \, \max_{\substack{|\beta|=i_1 \\ \beta_3\leq j_1}} \|\p^\beta\zeta\|_{H^2} \, \max_{\substack{|\gamma|=i_2 \\ \gamma_3\leq j_2}} \|\p^\gamma Q^\pm\|_{3,\pm}.
\end{align*}
Finally, multiplying by $\frac{\rho^n}{n!}$ and summing over $n\in\{k+1,\dots,N\}$, we have:
\begin{align*}
\sum_{n=k+1}^N \frac{\rho^n}{n!} \, \max_{|\alpha'|=n-k-1} \|T_1\|_{2,\pm} & \, \lesssim \, \sum_{j_1+j_2=k} \, \sum_{n=k+1}^N \,  \sum_{\substack{i_1+i_2=n \\ i_2\geq 1}} \frac{(n-1)! \, i_2}{n!} \, \frac{\rho^{i_1}}{i_1!} \, \max_{\substack{|\beta|=i_1 \\ \beta_3\leq j_1}} \|\p^\beta\zeta\|_{H^2} \, \frac{\rho^{i_2}}{i_2!} \, \max_{\substack{|\gamma|=i_2 \\ \gamma_3\leq j_2}} \|\p^\gamma Q^\pm\|_{3,\pm} \\
& \, \lesssim \, \sum_{j_1+j_2=k} \, \sum_{n=k+1}^N \,  \sum_{i_1+i_2=n} \frac{\rho^{i_1}}{i_1!} \, \max_{\substack{|\beta|=i_1 \\ \beta_3\leq j_1}} \|\p^\beta\zeta\|_{H^2} \, \frac{\rho^{i_2}}{i_2!} \, \max_{\substack{|\gamma|=i_2 \\ \gamma_3\leq j_2}} \|\p^\gamma Q^\pm\|_{3,\pm},
\end{align*}
where we have used the inequality $\frac{(n-1)! \, i_2}{n!} \leq 1$. To conclude, we recognize a ``partial'' Cauchy product with respect to $n$, and write:
\begin{align} \label{s2-estim_T1_3}
\sum_{n=k+1}^N \frac{\rho^n}{n!} \, \max_{|\alpha'|=n-k-1} \|T_1\|_{2,\pm}
\, \lesssim \, \sum_{j_1+j_2=k} \|\zeta\|_{\rho,2}^{j_1,N} \, \|Q^\pm\|_{\rho,3}^{j_2,\pm,N}
\, \lesssim \, \sum_{j_1+j_2=k} \|\zeta\|_{\rho,2}^{j_1} \, \|Q^\pm\|_{\rho,3}^{j_2,\pm},
\end{align}
which ends the case of $T_1$.

\medskip

\paragraph*{\textbf{$\blacktriangleright$ Treatment of $T_2$.}}

$\,$

\noindent
In the same way as $T_1$, we get:
\begin{align*}
T_2 := \p^\alpha\left(\frac{\zeta}{J} \, \p_h\psi \, \p_h\p_3 Q^\pm\right) = \sum_{\beta+\gamma=\alpha} \frac{\alpha!}{\beta! \, \gamma!} \, \p^\beta\left( \frac{\zeta}{J} \p_h\psi \right) \, \p^\gamma \p_h\p_3 Q^\pm.
\end{align*}
Then, the estimate of $\|T_2\|_{2,\pm}$ is identical to $T_1$, estimating $Q^\pm$ in $H^3(\Omega^\pm)$ \textit{via} the inequality
\begin{align*}
\max_{\substack{|\gamma|=i_2 \\ \gamma_3\leq j_2}} \|\p^\gamma \p_h\p_3 Q^\pm\|_{2,\pm} \, \leq \, \max_{\substack{|\gamma|=i_2+1 \\ \gamma_3\leq j_2}} \|\p^\gamma  Q^\pm\|_{3,\pm}.
\end{align*}
We conclude that
\begin{align}\label{s2-estim_T2}
\sum_{n=k+1}^N \frac{\rho^n}{n!} \, \max_{|\alpha'|=n-k-1} \|T_2\|_{2,\pm}
\, \lesssim \, \sum_{j_1+j_2=k} \left\|\frac{\zeta}{J}\p_h\psi\right\|_{\rho,2}^{j_1} \, \|Q^\pm\|_{\rho,3}^{j_2,\pm}.
\end{align}

\medskip

\paragraph*{\textbf{$\blacktriangleright$ Treatment of $T_3$.}}

$\,$

\noindent
By definition, we have
\begin{align*}
T_3 := \p^\alpha \left( \zeta \, \cF^\pm \right) = \sum_{\beta+\gamma=\alpha} \frac{\alpha!}{\beta! \, \gamma!} \, \p^\beta\zeta \, \p^\gamma\cF^\pm.
\end{align*}
As previously, we get:
\begin{align*}
\|T_3\|_{2,\pm} \, \lesssim \, \sum_{j_1+j_2=k} \, \sum_{i_1+i_2=n-1} \frac{(n-1)!}{i_1! \, i_2!} \, \max_{\substack{|\beta|=i_1 \\ \beta3\leq j_1}} \|\p^\beta\zeta\|_{H^2} \, \max_{\substack{|\gamma|=i_2 \\ \gamma_3\leq j_2}} \|\p^\gamma\cF^\pm\|_{2,\pm}.
\end{align*}
By assumption, we have $\cF^\pm\in B_{\rho,1,\sigma}(\Omega^\pm)$. As a consequence, we are reduced to the norm $\|\cdot\|_{1,\pm}$ using the following inequality:
\begin{align*}
\max_{\substack{|\gamma|=i_2 \\ \gamma_3\leq j_2}} \|\p^\gamma\cF^\pm\|_{2,\pm} \, \lesssim \, \max_{\substack{|\gamma|=i_2 \\ \gamma_3\leq j_2+1}} \|\p^\gamma\cF^\pm\|_{1,\pm} \, + \, \max_{\substack{|\gamma|=i_2+1 \\ \gamma_3\leq j_2+1}} \|\p^\gamma\cF^\pm\|_{1,\pm}.
\end{align*}
Next, the arguments are completely analogous to the cases of $T_1$ and $T_2$, and we end up with:
\begin{align}\label{s2-estim_T3}
\sum_{n=k+1}^N \frac{\rho^n}{n!} \, \max_{|\alpha'|=n-k-1} \|T_3\|_{2,\pm}
\, \lesssim \, (1+\rho) \, \sum_{j_1+j_2=k} \|\zeta\|_{\rho,2}^{j_1} \, \|\cF^\pm\|_{\rho,1}^{j_2+1,\pm}.
\end{align}

\medskip

\paragraph*{\textbf{$\blacktriangleright$ Treatment of $T_4$.}}

$\,$

\noindent
The case of $T_4$ is the same as $T_1$, but with an additional term to differentiate:
\begin{align*}
T_4 := \p^\alpha\left( \zeta \, A \, \p^1\tildA \, \p^1 Q^\pm \right) = \sum_{\beta+\gamma+\delta=\alpha} \frac{\alpha!}{\beta! \, \gamma! \, \delta!} \, \p^\beta (\zeta \, A) \, \p^\gamma\p^1\tildA \, \p^\delta\p^1 Q^\pm.
\end{align*}
We generalize estimate \eqref{s2-estim_T1_2} with three indices $i_1$, $i_2$ and $i_3$:
\begin{align*}
\|T_4\|_{2,\pm} \, \lesssim \, \sum_{j_1+j_2+j_3=k} \, \sum_{i_1+i_2+i_3=n-1} \frac{(n-1)!}{i_1! \, i_2! \, i_3!} \, \max_{\substack{|\beta|=i_1 \\ \beta_3\leq j_1}} \|\p^\beta (\zeta\, A)\|_{H^2} \, \max_{\substack{|\gamma|=i_2 \\ \gamma_3\leq j_2}} \|\p^\gamma\p^1\tildA\|_{H^2} \, \max_{\substack{|\delta|=i_3 \\ \beta_3\leq j_3}} \|\p^\delta\p^1 Q^\pm\|_{2,\pm}.
\end{align*}
Then, using both estimates
\begin{align*}
\|\p^\gamma\p^1\tildA\|_{H^2} \, \leq \, \|\p^\gamma\tildA\|_{H^3} \quad \text{ and } \quad \|\p^\delta\p^1 Q^\pm\|_{2,\pm} \, \leq \,
\|\p^\delta Q^\pm\|_{3,\pm},
\end{align*}
we can deduce that
\begin{align*}
\|T_4\|_{2,\pm} \, \lesssim \, \sum_{j_1+j_2+j_3=k} \, \sum_{i_1+i_2+i_3=n-1} \frac{(n-1)!}{i_1! \, i_2! \, i_3!} \, \max_{\substack{|\beta|=i_1 \\ \beta_3\leq j_1}} \|\p^\beta (\zeta\, A)\|_{H^2} \, \max_{\substack{|\gamma|=i_2 \\ \gamma_3\leq j_2}} \|\p^\gamma\tildA\|_{H^3} \, \max_{\substack{|\delta|=i_3 \\ \beta_3\leq j_3}} \|\p^\delta Q^\pm\|_{3,\pm}.
\end{align*}
Eventually, we obtain:
\begin{align}\label{s2-estim_T4}
\sum_{n=k+1}^N \frac{\rho^n}{n!} \, \max_{|\alpha'|=n-k-1} \|T_4\|_{2,\pm}
\, \lesssim \, \rho \, \sum_{j_1+j_2+j_3=k} \|\zeta \, A\|_{\rho,2}^{j_1} \, \|\tildA\|_{\rho,3}^{j_2} \, \|Q^\pm\|_{\rho,3}^{j_3,\pm}.
\end{align}

Back to estimate \eqref{s2-ineg_rec_Q_rho,3^k+1,N_2}, and using \eqref{s2-estim_T1_3}-\eqref{s2-estim_T4}, we have proved that
\begin{align}
\|Q^\pm\|_{\rho,3}^{k+1,\pm,N}  \, \lesssim \, (1+\rho) \, \|Q^\pm\|_{\rho,3}^{k,\pm} & + \sum_{j_1+j_2=k} \left( \|\zeta\|_{\rho,2}^{j_1} + \left\|\frac{\zeta}{J}\nabla'\psi\right\|_{\rho,2}^{j_1} \right) \|Q^\pm\|_{\rho,3}^{j_2,\pm} \nonumber\\
& + (1+\rho) \, \sum_{j_1+j_2=k} \|\zeta\|_{\rho,2}^{j_1} \, \|\cF^\pm\|_{\rho,1}^{j_2+1,\pm} \nonumber\\
& + \rho \sum_{j_1+j_2+j_3=k} \|\zeta \, A\|_{\rho,2}^{j_1} \, \|\tildA\|_{\rho,3}^{j_2} \, \|Q^\pm\|_{\rho,3}^{j_3,\pm}. \label{s2-estim_Q_rho,3^k+1,pm,N}
\end{align}
In particular, we can see that the norm $\|Q^\pm\|_{\rho,3}^{k+1,\pm,N}$ is uniformly bounded with respect to $N$. Thus, taking the supremum over $N$, we have $\|Q^\pm\|_{\rho,3}^{k+1,\pm} < +\infty$. By induction, we obtain $\ds Q^\pm\in \cap_{k\geq 0} B_{\rho,3}^{k}(\Omega^\pm)$, with estimate \eqref{s2-estim_Q_rho,3^k+1,pm,N} which also holds for $\|Q^\pm\|_{\rho,3}^{k+1,\pm}$. Multiplying by $\sigma^{k+1}$ and summing over $k\in\{0,\dots,K-1\}$ (where $K\geq 1$ is any integer), we end up with the following estimate for the partial sum $\ds\|Q^\pm\|_{\rho,3,\sigma}^{\pm,K}$ of the norm $\ds\|Q^\pm\|_{\rho,3,\sigma}^{\pm}$:
\begin{align}
\|Q^\pm\|_{\rho,3,\sigma}^{\pm,K} & \lesssim \|Q^\pm\|_{\rho,3}^{0,\pm} \, + \, (1+\rho) \, \|\zeta\|_{\rho,2,\sigma} \, \|\cF^\pm\|_{\rho,1,\sigma}^\pm \label{s2-estim_Q_rho,3,sigma^K_1} \\
& \quad + \sigma \left( \|\zeta\|_{\rho,2,\sigma} \, + \, \left\|\frac{\zeta}{J}\nabla'\psi\right\|_{\rho,2,\sigma} \, + \, \rho \, \|\zeta\|_{\rho,2,\sigma} \, \|A\|_{\rho,2,\sigma} \, \|\tildA\|_{\rho,3,\sigma} \right) \|Q^\pm\|_{\rho,3,\sigma}^{\pm,K}. \label{s2-estim_Q_rho,3,sigma^K_2}
\end{align}
Using \eqref{s2-estim_Q_rho,3^0} satisfied by $\ds\|Q^\pm\|_{\rho,3}^{0,\pm}$, and bounding from above $\rho$ by $\rho_0\leq 1$, we can estimate the right term of \eqref{s2-estim_Q_rho,3,sigma^K_1} by
\begin{align*}
C_0 \left( \|Q^\pm\|_{\rho,3}^{0,\pm} \, + \, \|\zeta\|_{\rho,2,\sigma} \, \|\cF^\pm\|_{\rho,1,\sigma}^\pm \right),
\end{align*}
where $C_0>0$ is a numerical constant. Estimate \eqref{s2-estim_Q_rho,3,sigma^K_1}-\eqref{s2-estim_Q_rho,3,sigma^K_2} is rewritten as follows:
\begin{align}
\|Q^\pm\|_{\rho,3,\sigma}^{\pm,K} & \, \leq \, C_1 \, \big( \|Q^\pm\|_{\rho,3}^{0,\pm} \, + \, \|\zeta\|_{\rho,2,\sigma} \, \|\cF^\pm\|_{\rho,1,\sigma}^\pm \big) \label{s2-estim_Q_rho,3,sigma^K_4} \\
& \quad + \sigma \, C_1 \, \left( \|\zeta\|_{\rho,2,\sigma} \, + \, \left\|\frac{\zeta}{J}\nabla'\psi\right\|_{\rho,2,\sigma} \, + \, \rho \, \|\zeta\|_{\rho,2,\sigma} \, \|A\|_{\rho,2,\sigma} \, \|\tildA\|_{\rho,3,\sigma} \right) \|Q^\pm\|_{\rho,3,\sigma}^{\pm,K}, \label{s2-estim_Q_rho,3,sigma^K_3}
\end{align}
where $C_1>0$ is a numerical constant. From now on, it remains to absorb the analytic terms in \eqref{s2-estim_Q_rho,3,sigma^K_3}, in order to complete the estimate of $\ds\|Q^\pm\|_{\rho,3,\sigma}^{\pm,K}$.

\medskip

\paragraph*{\textbf{Absorption of the analytic terms.}}

$\,$

\medskip

\paragraph*{\textbf{$\blacktriangleright$ Treatment of $\ds\|\zeta\|_{\rho,2,\sigma}$.}}

$\,$

\noindent
We proceed in the same way as $\ds\|Q^\pm\|_{\rho,3}^{0,\pm}$. We use the gain of one power of $\sigma$ in order to absorb the analytic terms in \eqref{s2-estim_Q_rho,3,sigma^K_3}, choosing $\sigma$ small enough. Let $f\in B_{\rho,\frac{7}{2}}(\bT^2)$ still satisfy assumption \eqref{s2-hyp_f_petit_thm_Q_tang} of Theorem \ref{s2-thm_estim_Q_rho,3^0}. As a consequence, we obtain both following estimates (we have just rewritten \eqref{s2-estim_1/J_rho,r,sigma} and \eqref{s2-estim_Psi_f_unif}):
\begin{align}
& \|\nabla\psi\|_{\rho,3,\sigma} \leq \|\psi\|_{\rho,4,\sigma} \leq C_0 \|f\|_{\rho,\frac{7}{2}} \leq C_0, \label{s2-estim_redressement_rappel}\\
& \left\|\frac{1}{J}\right\|_{\rho,3,\sigma} \leq C_0. \label{s2-estim_1/J_rappel}
\end{align}
Here, and from now on, $C_0>0$ will stand for any numerical constant. We notice that both estimates \eqref{s2-estim_redressement_rappel} and \eqref{s2-estim_1/J_rappel} hold for all $\rho\in (0,\rho_0]$ and $\sigma\in (0,\frac{1}{2}]$. Using the algebra property of $B_{\rho,2,\sigma}(\Omega)$, we have:
\begin{align*}
\|\zeta\|_{\rho,2,\sigma} \, \leq \, C_0 \, \left\|\frac{1}{1+|\nabla'\psi|^2}\right\|_{\rho,2,\sigma} \, \|J\|_{\rho,2,\sigma}^2.
\end{align*}
On the one hand, using \eqref{s2-estim_redressement_rappel}, we can write:
\begin{align*}
\|J\|_{\rho,2,\sigma} \, = \, \|1 + \p_3\psi\|_{\rho,2,\sigma} \, \leq \, C_0 \, ( 1 + \|\psi\|_{\rho,3,\sigma}) \, \leq \, C_0.
\end{align*}
On the other hand, decreasing $\eta_0>0$ if necessary (see assumption \eqref{s2-hyp_f_petit_thm_Q_tang}), we get:
\begin{align*}
\left\|\frac{1}{1+|\nabla'\psi|^2}\right\|_{\rho,2,\sigma} \, \leq \, C_0.
\end{align*}
Indeed, it suffices to use Proposition \ref{s2-prop_estim_analytique_1/1+g} and to adapt the proof of Corollary \ref{s2-coro_estim_1/J}. Eventually, we can state that $\zeta$ is bounded in $B_{\rho,2,\sigma}(\Omega)$:
\begin{align}\label{s2-estim_zeta_rho,2,sigma}
\|\zeta\|_{\rho,2,\sigma} \, \leq \, C_0.
\end{align}

\medskip

\paragraph*{\textbf{$\blacktriangleright$ Treatment of $\ds\left\|\frac{\zeta}{J}\nabla'\psi\right\|_{\rho,2,\sigma}$ and $\rho \, \|A\|_{\rho,2,\sigma} \, \|\zeta\|_{\rho,2,\sigma} \, \|\tildA\|_{\rho,3,\sigma}$.}}

$\,$

\noindent
In the following, we do not detail the computations of these terms, since we use exactly the same method as before. It gives the same type of estimate as \eqref{s2-estim_zeta_rho,2,sigma}, namely:
\begin{align}\label{s2-estim_zeta_rho,2,sigma_bis}
\left\|\frac{\zeta}{J} \nabla'\psi\right\|_{\rho,2,\sigma} \, \leq \, C_0 \quad \text{ and } \quad \rho \, \|\zeta\|_{\rho,2,\sigma} \, \|A\|_{\rho,2,\sigma} \, \|\tildA\|_{\rho,3,\sigma} \, \leq \, C_0,
\end{align}
where we have used $\rho\leq\rho_0\leq 1$.

\medskip

\paragraph*{\textbf{$\blacktriangleright$ Limitation of the parameter $\sigma$.}}

$\,$

\noindent
From now on, we redefine $\sigma_0$ as follows:
\begin{align*}
\sigma_0 := \frac{1}{6 C_0 C_1},
\end{align*}
where $C_0>0$ is given by \eqref{s2-estim_zeta_rho,2,sigma}-\eqref{s2-estim_zeta_rho,2,sigma_bis} and $C_1$ by \eqref{s2-estim_Q_rho,3,sigma^K_3}. Without loss of generality, we can increase $C_0$ to still have $\sigma_0 \leq \frac{1}{2}$. Therefore all the previous estimates obtained with $\sigma\in (0,\frac{1}{2}]$ remain valid. For all $\rho\in (0,\rho_0]$ and $\sigma\in (0,\sigma_0]$, the right term of \eqref{s2-estim_Q_rho,3,sigma^K_3} is estimated as follows:
\begin{align*}
\sigma \, C_1 \, \left( \|\zeta\|_{\rho,2,\sigma} \, + \, \left\|\frac{\zeta}{J}\nabla'\psi\right\|_{\rho,2,\sigma} \, + \, \rho \, \|\zeta\|_{\rho,2,\sigma} \, \|A\|_{\rho,2,\sigma} \, \|\tildA\|_{\rho,3,\sigma} \right) \|Q^\pm\|_{\rho,3,\sigma}^{\pm,K} & \, \leq \, 3 \, C_0 \, C_1 \, \sigma \, \|Q^\pm\|_{\rho,3,\sigma}^{\pm,K} \\
& \leq \frac{1}{2} \|Q^\pm\|_{\rho,3,\sigma}^{\pm,K}.
\end{align*}
Consequently, we get the following estimate for the norm $\|Q^\pm\|_{\rho,3,\sigma}^{\pm,K}$:
\begin{align}
\|Q^\pm\|_{\rho,3,\sigma}^{\pm,K} & \, \leq \, C_1 \, \big( \|Q^\pm\|_{\rho,3}^{0,\pm} \, + \, \|\zeta\|_{\rho,2,\sigma} \, \|\cF^\pm\|_{\rho,1,\sigma}^\pm \big). \label{s2-estim_Q_rho,3,sigma^K_5}
\end{align}
To conclude, we use \eqref{s2-estim_zeta_rho,2,sigma} together with \eqref{s2-estim_Q_rho,3^0} and simplify \eqref{s2-estim_Q_rho,3,sigma^K_5} to end up with
\begin{align*}
\|Q^\pm\|_{\rho,3,\sigma}^{\pm,K} & \, \leq \, C_2 \, \Big( \|\cF^\pm\|_{\rho,1,\sigma}^\pm \, + \, \|\cG\|_{\rho,\frac{3}{2}} \Big), 
\end{align*}
where $C_2>0$ is a numerical constant. Taking the supremum over $K$, we deduce the same estimate satisfied by the norm $\|Q^\pm\|_{\rho,3,\sigma}^\pm$. Theorem \ref{s2-thm_estim_Q_final} is proved.

\section{Existence of analytic solutions to the current-vortex sheets problem}\label{s2-sec_CK_thm}

In the following, both parameters $\rho_0 \in (0,1]$ and $\sigma_0 \in (0,\frac{1}{2}]$ given by Theorem \ref{s2-thm_estim_Q_final}, as well as $\sigma\in (0,\sigma_0]$ are fixed once and for all. Nevertheless, $C_0>0$ will stand for any numerical constant that is free to change from one estimate to an other.

\subsection{Preliminary estimates}\label{s2-sec_prelim_estim}

Let $R>0$ and let $(u_0^\pm,b_0^\pm,f_0)$ be some initial datum satisfying \eqref{s2-contraintes_u0_b0_f0} and such that:
\begin{equation}\label{art-hyp_cond_init}
\begin{aligned}
& u_0^\pm, \, b_0^\pm \in B_{\rho_0,3,\sigma}(\Omega^\pm)^3, \quad f_0 \in B_{\rho_0,\frac{7}{2}}(\bT^2), \\[0.5ex]
& \| u_0^\pm, \, b_0^\pm \|_{\rho_0,3,\sigma}^\pm \, < \, R, \quad \| f_0 \|_{\rho_0,\frac{7}{2}} \, < \, \eta_1,
\end{aligned}
\end{equation}
where $\eta_1$ fulfills the smallness condition \eqref{art-hyp_eta_a_petits} below. We consider $(\dot{u}^\pm,\dot{b}^\pm,\ddot{f}) \in \bE_a$ (recall definition \eqref{s2-def_bE_a}), satisfying the following upper bounds:
\begin{subequations}\label{art-nt_du_db_ddf_R/2a}
\begin{align}\label{art-nt_du_db_R/2a}
\nt{\dot{u}^\pm}_{a,\Omega^\pm} \, \leq \, \frac{R}{2a} \quad \text{ and } \quad \nt{\dot{b}^\pm}_{a,\Omega^\pm} \, \leq \, \frac{R}{2a}.
\end{align}
We remind that $(u^\pm,b^\pm,f)$, defined by \eqref{art-primitives_du_db_ddf}, satisfy the constraints \eqref{s2-contraintes_u_b_f}. Consequently, using the equation $\dot{f} = u_3^\pm|_\Gamma$ together with the trace estimate \eqref{art-estim_trace}, we can state that there exists a numerical constant $\uC > 0$ such that:
\begin{align}\label{art-nt_ddf_uCR/2a}
\nt{\ddot{f}}_{a,\bT^2} \, \leq \, \frac{\uC \, R}{2a}.
\end{align}
\end{subequations}
From now on, we give some basic estimates that we will need in Paragraphs \ref{art-sec_espace_invariant} and \ref{art-sec_contraction} below. These estimates are a straightforward consequence of Lemmas 1 and 2 in \cite{BG77}.

\medskip

\noindent
\textbf{Estimate of $u^\pm(t)$ and $b^\pm(t)$ in $B_{\rho,3,\sigma}(\Omega^\pm)$.} Using \eqref{art-primitives_du_db_ddf}, \eqref{art-hyp_cond_init} and \eqref{art-nt_du_db_R/2a}, Lemma 1 of \cite{BG77} gives:
\begin{align}\label{art-estim_u^pm_b^pm_B_rho,3,sigma}
\forall \, t\in \big[0, a(\rho_0 - \rho)\big), \quad \| u^\pm(t), \, b^\pm(t) \|_{\rho,3,\sigma}^\pm \, < \, 2R.
\end{align}
Then, in order to handle the $1-$order derivatives of $u^\pm$ and $b^\pm$ in system \eqref{s2-eq_evo_u_b}, we use Lemma 2 of \cite{BG77}:
\begin{align}\label{art-estim_du^pm_db^pm_B_rho,3,sigma}
\forall \, k\in\{1,2,3\}, \quad \forall \, t\in \big[0, a(\rho_0 - \rho)\big), \quad \| \p_k u^\pm(t), \, \p_k b^\pm(t) \|_{\rho,3,\sigma}^\pm \, < \, \frac{C_0 \, R}{\rho_0 - \rho} \, \Big( 1 - \frac{t}{a (\rho_0 - \rho)} \Big)^{-\frac{1}{2}}.
\end{align}

\medskip

\noindent
\textbf{Estimate of $f(t)$ in $B_{\rho,\frac{7}{2}}(\bT^2)$.} Recalling Definition \ref{s2-def_esp_ana_T^2}, we can easily prove the following estimate:
\begin{align}\label{art-estim_rho,7/2_rho,5/2}
\| f \|_{\rho,\frac{7}{2}} \, \leq \, C_0 \, \big( \|f\|_{\rho,\frac{5}{2}} \, + \, \| \nabla' f \|_{\rho,\frac{5}{2}} \big),
\end{align}
where $C_0 > 0$ is a numerical constant. As a consequence, we can deduce that:
\begin{align}\label{art-estim_f_B_rho,7/2_0}
\forall \, t\in \big[0, a(\rho_0 - \rho)\big), \quad \| f(t) \|_{\rho,\frac{7}{2}} \, < \, C_1 \, \big( \eta_1 \, + \, a \, R \big),
\end{align}
where $C_1 > 0$ is a numerical constant.
\begin{proof}
Using the third lign of \eqref{art-primitives_du_db_ddf} and estimating each term in $B_{\rho,\frac{5}{2}}(\bT^2)$, Lemma 1 of \cite{BG77} allows us to write the estimate:
\begin{align*}
\forall \, t \in \big[0, a(\rho_0 - \rho)\big), \quad \| f(t) \|_{\rho,\frac{5}{2}} \, \leq \, \|f_0\|_{\rho,\frac{5}{2}} \, + \, \| t \, u_{0,3}^\pm|_\Gamma \|_{\rho,\frac{5}{2}} \, + \, C_0 \, a^2 \, \nt{\ddot{f}}_{a,\bT^2},
\end{align*}
where $C_0 > 0$ is a numerical constant. Then, from \eqref{art-hyp_cond_init} and \eqref{art-nt_ddf_uCR/2a} we get:
\begin{align}\label{art-estim_f_B_rho,7/2_1}
\forall \, t \in \big[0, a(\rho_0 - \rho)\big), \quad \| f(t) \|_{\rho,\frac{5}{2}} \, \leq \, C_0 \, \big( \eta_1 \, + \, a \, R \big).
\end{align}
We now focus on the second term in the right-hand side of \eqref{art-estim_rho,7/2_rho,5/2}. Differentiating the third lign of \eqref{art-primitives_du_db_ddf} with respect to $x_j$ ($j=1,2$) and using Lemma 2 of \cite{BG77} we obtain :
\begin{align*}
\forall \, t \in \big[0, a(\rho_0 - \rho)\big), \quad \| \p_j f(t) \|_{\rho,\frac{5}{2}} 
& \, \leq \, \| \p_j f_0 \|_{\rho,\frac{5}{2}} \, + \, \| t \, \p_j u_{0,3}^\pm|_\Gamma \|_{\rho,\frac{5}{2}} \, + \, C_0 \, a^2 \, \nt{\ddot{f}}_{a,\bT^2} \\
& \, \leq \, \| f_0 \|_{\rho,\frac{7}{2}} \, + \, a \, (\rho_0 - \rho) \, \frac{C_0}{\rho_0 - \rho} \, \| u_{0,3}^\pm|_\Gamma \|_{\rho_0,\frac{5}{2}} \, + \, C_0 \, a^2 \, \nt{\ddot{f}}_{a,\bT^2}.
\end{align*}
Here, let us remark that we have used the assumption $f_0 \in B_{\rho_0,\frac{7}{2}}(\bT^2)$. Thus, assumption \eqref{art-hyp_cond_init} leads to the final estimate:
\begin{align}\label{art-estim_f_B_rho,7/2_2}
\forall \, t \in \big[0, a(\rho_0 - \rho)\big), \quad \| \p_j f(t) \|_{\rho,\frac{5}{2}} \, \leq \, C_0 \, \big( \eta_1 \, + \, a \, R \big).
\end{align}
Gathering \eqref{art-estim_f_B_rho,7/2_1}, \eqref{art-estim_f_B_rho,7/2_2} together with \eqref{art-estim_rho,7/2_rho,5/2} we end up with the desired estimate \eqref{art-estim_f_B_rho,7/2_0}.
\end{proof}

 Now, taking $\eta_1 > 0$ and $a>0$ such that
\begin{align}\label{art-hyp_eta_a_petits}
\eta_1 < \frac{\eta_0}{2C_1} \quad \text{ and } \quad a < \frac{\eta_0}{2C_1R},
\end{align}
we can state that
\begin{align}\label{art-estim_f_B_rho,7/2}
\| f(t) \|_{\rho,\frac{7}{2}} < \eta_0.
\end{align}
Therefore we shall be able to apply Corollary \ref{s2-coro_estim_1/J} and Theorem \ref{s2-thm_estim_Q_final}.

\medskip

\noindent
\textbf{Estimate of the lifting $\psi(t)$ in $B_{\rho,3,\sigma}(\Omega)$.} As a straightforward consequence of \eqref{art-estim_f_B_rho,7/2} and Proposition \ref{s2-prop_redressement_front}, we can write both following estimates:
\begin{align}\label{art-estim_psi_rho,3,sigma}
\forall \, t\in \big[0, a(\rho_0 - \rho)\big), \quad \| \psi(t) \|_{\rho,3,\sigma} \, < \, C_0 \, \eta_0,
\end{align}
and for all $k\in\{1,2,3\}$, we have:
\begin{align}\label{art-estim_dpsi_rho,3,sigma}
\forall \, t\in \big[0, a(\rho_0 - \rho)\big), \quad \| \p_k\psi(t) \|_{\rho,3,\sigma} \, < \, C_0 \, \eta_0.
\end{align}
The tricky point comes out when we need to estimate \textbf{two} derivatives of $\psi$ (which appear in the equations of system \eqref{s2-eq_evo_u_b}). Such an estimate is not proved in \cite{BG77}, so we choose to give some details below. Nevertheless, we still follow the same ideas that can be found in the proof of Lemma 2 in \cite{BG77}. Given $k,\ell\in\{1,2,3\}$, the estimate of $\p_k\p_\ell\psi$ reads:
\begin{align}\label{art-estim_ddpsi_rho,3,sigma}
\forall \, t\in \big[0, a(\rho_0 - \rho)\big), \quad \| \p_k\p_\ell\psi(t) \|_{\rho,3,\sigma} \, < \, C_0 \, \eta_0 \, \frac{1}{\rho_0 - \rho} \, \Big( 1 - \frac{t}{a (\rho_0 - \rho)} \Big)^{-\frac{1}{2}}.
\end{align}
\begin{proof}
Starting from \eqref{art-primitives_du_db_ddf} and denoting by $\Psi$ the lifting operator defined by Lemma \ref{s2-lemme_redressement_CMST}, we have:
\begin{align}\label{art-expr_ddpsi(t)}
\p_k\p_\ell\psi(t) \, = \, \p_k\p_\ell \Psi(f_0) \, + \, t \, \p_k\p_\ell\Psi\big( u_{0,3}^\pm|_\Gamma \big) \, + \, \int_0^t \int_0^{t_1} \p_k\p_\ell\Psi(\ddot{f})(t_2) \, dt_2 \, dt_1.
\end{align}
The first term in the right-hand side of \eqref{art-expr_ddpsi(t)} is estimated as follows:
\begin{align}\label{art-estim_ddpsi(t)_1}
\| \p_k\p_\ell\Psi(f_0) \|_{\rho,3,\sigma} \, \leq \, \|\p_\ell\Psi(f_0)\|_{\rho,4,\sigma} \, \leq \, \frac{C_0}{\rho_0 - \rho} \| \Psi(f_0) \|_{\rho_0,3,\sigma} \, \leq \, \frac{C_0}{\rho_0 - \rho} \, \|f_0\|_{\rho_0,\frac{7}{2}} \, < \, \frac{C_0 \, \eta_0}{\rho_0 - \rho}. 
\end{align}
Likewise, the second term in the right-hand side of \eqref{art-expr_ddpsi(t)} gives:
\begin{align*}
\| t \, \p_k\p_\ell\Psi\big( u_{0,3}^\pm|_\Gamma \big) \|_{\rho,3,\sigma} \, \leq \, a(\rho_0-\rho) \, \frac{C_0}{\tildrho - \rho} \, \|\p_\ell\Psi\big( u_{0,3}^\pm|_\Gamma \big) \|_{\tildrho,3,\sigma} \, \leq \, C_0 \, a \, \frac{\rho_0 - \rho}{(\tildrho - \rho)(\rho_0 - \tildrho)} \| \Psi\big( u_{0,3}^\pm|_\Gamma \big) \|_{\rho_0,3,\sigma},
\end{align*}
for a suitable choice of $\tildrho$ such that $0 < \rho < \tildrho < \rho_0$. We thus set $\tildrho := \frac{\rho + \rho_0}{2}$ and we use Proposition \ref{s2-prop_redressement_front} together with \eqref{art-hyp_eta_a_petits} to end up with the following estimate:
\begin{align}\label{art-estim_ddpsi(t)_2}
\| t \, \p_k\p_\ell\Psi\big( u_{0,3}^\pm|_\Gamma \big) \|_{\rho,3,\sigma} \, < \, \frac{C_0 \, \eta_0}{\rho_0 - \rho}.
\end{align}
Let us finish with the double integral in \eqref{art-expr_ddpsi(t)}. Using twice the property \eqref{s1-deriv_rho,r,sigma} together with Proposition \ref{s2-prop_redressement_front}, we get:
\begin{align}\label{art-estim_ddpsi(t)_3}
\int_0^t \int_0^{t_1} \| \p_k\p_\ell\Psi(\ddot{f}) \|_{\rho,3,\sigma} \, dt_2 \, dt_1 \, \leq \, C_0 \, \int_0^t \int_0^{t_1} \frac{\| \ddot{f}(t_2) \|_{\tildrho(t_2),\frac{5}{2}}}{\big(\rho(t_2)-\rho\big) \, \big(\tildrho(t_2)-\rho(t_2)\big)} \, dt_2 \, dt_1,
\end{align} 
for a suitable choice of $\rho(t_2)$ and $\tildrho(t_2)$ satisfying $\rho < \rho(t_2) < \tildrho(t_2) < \rho_0$. In conformance with \cite{BG77}, we set
\begin{equation}\label{s2-def_rho(s)_tildrho(s)}
\forall \, s \in \big[0, a(\rho_0 - \rho)\big), \quad \rho(s) \, := \, \frac{1}{2} \, \Big( \rho_0 \, + \, \rho \, - \, \frac{s}{a} \Big) \quad \text{ and } \quad \tildrho(s) \, := \, \frac{1}{2} \, \Big( \rho_0 \, + \, \rho(s) \, - \, \frac{s}{a} \Big).
\end{equation}
Straightforward computations show that:
\begin{subequations}\label{s2-expr_diff_rho_tildrho}
\begin{align}
& \rho(s) \, - \, \rho \, = \, \frac{\rho_0 - \rho}{2} \, \bigg( 1 - \frac{s}{a(\rho_0-\rho)} \bigg), \label{s2-expr_rho(s)-rho} \\
& \rho_0 \, - \, \rho(s) \, = \, \frac{\rho_0 - \rho}{2} \, \bigg( 1 + \frac{s}{a(\rho_0-\rho)} \bigg), \label{s2-expr_rho0-rho(s)}
\end{align}
and also
\begin{align}
& \tildrho(s) \, - \, \rho(s) \, = \, \frac{\rho_0 - \rho}{4} \, \bigg( 1 - \frac{s}{a(\rho_0-\rho)} \bigg), \label{s2-expr_tildrho(s)-rho(s)} \\
& \rho_0 \, - \, \tildrho(s) \, = \, \frac{\rho_0 - \rho}{4} \, \bigg( 1 + \frac{3s}{a(\rho_0-\rho)} \bigg), \label{s2-expr_rho0-tildrho(s)}
\end{align}
\end{subequations}
so that we have the inequalities $\rho < \rho(s) < \tildrho(s) < \rho_0$ for all $s\in \big[0, a(\rho_0 - \rho)\big)$. Let us come back to \eqref{art-estim_ddpsi(t)_3}; using definition \eqref{s2-def_nt_a_Omega} and equalities \eqref{s2-expr_diff_rho_tildrho} we end up with the estimate:
\begin{align}\label{art-estim_ddpsi(t)_4}
\int_0^t \int_0^{t_1} \| \p_k\p_\ell\Psi(\ddot{f}) \|_{\rho,3,\sigma} \, dt_2 \, dt_1 \, \leq \, C_0 \, a^2 \, \nt{\ddot{f}}_{a,\bT^2} \, \frac{1}{\rho_0 - \rho} \, \Big( 1 - \frac{t}{a (\rho_0 - \rho)} \Big)^{-\frac{1}{2}}.
\end{align}
Gathering \eqref{art-estim_ddpsi(t)_1}, \eqref{art-estim_ddpsi(t)_2} and \eqref{art-estim_ddpsi(t)_4} we obtain the desired estimate \eqref{art-estim_ddpsi_rho,3,sigma}.
\end{proof}

From now on, we deal with the time-derivatives of the unknowns $(u^\pm,b^\pm,f)$.

\medskip

\noindent
\textbf{The velocity $u^\pm$ and the magnetic field $b^\pm$.} Both following estimates are trivial but will turn out to be crucial is Paragraphs \ref{art-sec_espace_invariant} and \ref{art-sec_contraction}. According to \eqref{s2-def_nt_a_Omega} and \eqref{art-nt_du_db_R/2a} we can write:
\begin{align}\label{art-estim_dotu_dotb_rho,3,sigma}
\forall \, t\in \big[0, a(\rho_0 - \rho)\big), \quad \| \dot{u}^\pm(t), \, \dot{b}^\pm(t) \|_{\rho,3,\sigma}^\pm \, < \, \frac{R}{2a(\rho_0-\rho)} \, \Big( 1 - \frac{t}{a (\rho_0 - \rho)} \Big)^{-\frac{1}{2}}.
\end{align}

\medskip

\noindent
\textbf{The front $f$.} Likewise, we have:
\begin{align}\label{art-estim_dotdotf_rho,5/2}
\forall \, t\in \big[0, a(\rho_0 - \rho)\big), \quad \| \ddot{f}(t) \|_{\rho,\frac{5}{2}} \, < \, \frac{C_0 \, R}{a(\rho_0-\rho)} \, \Big( 1 - \frac{t}{a (\rho_0 - \rho)} \Big)^{-\frac{1}{2}}.
\end{align}
Now, let us estimate the functions $\dot{f}$ and $\nabla'\dot{f}$. Starting from the identity
\begin{align*}
\dot{f}(t) \, = \, u_{0,3}^\pm|_\Gamma \, + \, \int_0^t \ddot{f}(t_1) \, dt_1,
\end{align*}
Lemmas 1 and 2 of \cite{BG77} allow us to write both following estimates:
\begin{align}
\forall \, t\in \big[0, a(\rho_0 - \rho)\big), \quad & \| \dot{f}(t) \|_{\rho,\frac{5}{2}} \, < \, C_0 \, R, \label{art-estim_dotf_rho,5/2} \\
& \| \nabla' \dot{f}(t) \|_{\rho,\frac{5}{2}} \, < \, \frac{C_0 \, R}{\rho_0 - \rho} \, \Big( 1 - \frac{t}{a (\rho_0 - \rho)} \Big)^{-\frac{1}{2}}.\label{art-estim_grad_dotf_rho,5/2}
\end{align}

\medskip

\noindent
\textbf{The lifting $\psi$.} By definition of the lifting operator $\Psi$ (see Proposition \ref{s2-prop_redressement_front}), we can state that
\begin{align*}
\p_t\psi \, = \, \Psi(\p_t f) \, = \, \Psi\big( u_3^\pm|_\Gamma \big).
\end{align*}
As a consequence, we get:
\begin{align}\label{art-estim_dotpsi_rho,3,sigma}
\forall \, t\in \big[0, a(\rho_0 - \rho)\big), \quad  \|\p_t\psi(t)\|_{\rho,3,\sigma} \, < \, C_0 \, R.
\end{align}
Then, using \eqref{art-primitives_du_db_ddf} and the linearity of $\Psi$ we can write:
\begin{align*}
\forall \, k\in\{1,2,3\}, \quad \p_k\Psi\big( u_3^\pm|_\Gamma \big)(t) \, = \, \p_k\Psi\big( u_{0,3}^\pm|_\Gamma \big) \, + \, \int_0^t \p_k \Psi\big( \dot{u}_3^\pm|_\Gamma \big)(t_1) \, dt_1.
\end{align*}
The same arguments as before allow to obtain:
\begin{align}\label{art-estim_d_dotpsi_rho,3,sigma}
\forall \, k\in\{1,2,3\}, \quad \forall \, t\in \big[0, a(\rho_0 - \rho)\big), \quad \| \p_k\p_t\psi(t) \|_{\rho,3,\sigma} \, < \, \frac{C_0 \, R}{\rho_0 - \rho} \, \Big( 1 - \frac{t}{a (\rho_0 - \rho)} \Big)^{-\frac{1}{2}}.
\end{align}

Eventually, we deal with some estimates we shall use in Paragraph \ref{art-sec_contraction} to prove the contraction of $\Upsilon$. To do so, we now consider two vectors $(\dot{u}^\pm,\dot{b}^\pm,\ddot{f})$ and $(\dot{\uu}^\pm,\dot{\ub}^\pm,\ddot{\uf})$ in $\bE_a$ satisfying \eqref{art-nt_du_db_ddf_R/2a} with the same initial datum \eqref{art-hyp_cond_init}.

\medskip

\noindent
\textbf{The velocity and the magnetic field.} Lemma 1 of \cite{BG77} gives the estimate:
\begin{align}\label{art-estim_u-uu_rho,3,sigma}
\forall \, t\in \big[0, a(\rho_0 - \rho)\big), \quad \| u^\pm(t) \, - \, \uu^\pm(t) \|_{\rho,3,\sigma}^\pm \, < \, C_0 \, a \, \nt{\dot{u}^\pm \, - \, \dot{\uu}^\pm}_{a,\Omega^\pm}.
\end{align}
Then, using Lemma 2 of \cite{BG77}, we have for all $k\in\{1,2,3\}$:
\begin{align}\label{art-estim_du-duu_rho,3,sigma}
\forall \, t\in \big[0, a(\rho_0 - \rho)\big), \quad 
\| \p_k ( u^\pm \, - \, \uu^\pm )(t) \|_{\rho,3,\sigma}^\pm \, < \, C_0 \, a \,  \nt{\dot{u}^\pm \, - \, \dot{\uu}^\pm}_{a,\Omega^\pm} \, \frac{1}{\rho_0-\rho} \, \Big( 1 - \frac{t}{a (\rho_0 - \rho)} \Big)^{-\frac{1}{2}}.
\end{align}
Obviously, the same estimates hold for the magnetic field as well.

\medskip

\noindent
\textbf{The front.} Likewise, we can write both following estimates:
\begin{align}
\forall \, t\in \big[0, a(\rho_0 - \rho)\big), \quad 
& \| (f \, - \, \uf)(t) \|_{\rho,\frac{5}{2}} \, < \, C_0 \, a^2 \, \nt{\ddot{f} \, - \, \ddot{\uf}}_{a,\bT^2}, \label{art-estim_f-uf_rho,5/2} \\
& \| ( \nabla' f - \nabla'\uf )(t) \|_{\rho,\frac{5}{2}} \, < \, C_0 \, a^2 \, \nt{\ddot{f} \, - \, \ddot{\uf}}_{a,\bT^2}. \label{art-estim_gradf-graduf_rho,5/2}
\end{align}

\medskip

\noindent
\textbf{The lifting of the front.} Once again, using the linearity of the lifting operator $\Psi$ and the same arguments we used to prove \eqref{art-estim_ddpsi_rho,3,sigma}, we obtain for all $k,\ell\in\{1,2,3\}$:
\begin{align}\label{art-estim_ddpsi-ddupsi_rho,3,sigma}
\forall \, t\in \big[0, a(\rho_0 - \rho)\big), \quad 
\| \p_k\p_\ell(\psi \, - \, \upsi) \|_{\rho,3,\sigma} \, < \, C_0 \, a^2 \, \nt{\ddot{f} \, - \, \ddot{\uf}}_{a,\bT^2} \, \frac{1}{\rho_0 - \rho} \, \Big( 1 - \frac{t}{a (\rho_0 - \rho)} \Big)^{-\frac{1}{2}}.
\end{align}

Let us finish with estimates associated with the time-derivatives of $(u^\pm,b^\pm,f)$ and $(\uu^\pm,\ub^\pm,\uf)$. Similarly to \eqref{art-estim_dotu_dotb_rho,3,sigma} we have:
\begin{align}\label{art-estim_dotu-dotuu_rho,3,sigma}
\forall \, t\in \big[0, a(\rho_0 - \rho)\big), \quad 
\| \dot{u}^\pm(t) \, - \, \dot{\uu}^\pm(t) \|_{\rho,3,\sigma}^\pm \, \leq \, \nt{ \dot{u}^\pm \, - \, \dot{\uu}^\pm}_{a,\Omega^\pm} \, \frac{1}{\rho_0 - \rho} \, \Big( 1 - \frac{t}{a (\rho_0 - \rho)} \Big)^{-\frac{1}{2}},
\end{align}
and so on for $\dot{b}^\pm \, - \, \dot{\ub}^\pm$ and $\ddot{f} \, - \, \ddot{\uf}$. For the time-derivatives $\dot{f}$ and $\dot{\uf}$, the following estimates hold:
\begin{align}
\forall \, t\in \big[0, a(\rho_0 - \rho)\big), \quad 
& \| (\dot{f} \, - \, \dot{\uf})(t) \|_{\rho,\frac{5}{2}} \, < \, C_0 \, a \, \nt{\ddot{f} \, - \, \ddot{\uf}}_{a,\bT^2}, \label{art-estim_dotf-dotuf_rho,5/2} \\
& \| \nabla' ( \dot{f} - \dot{\uf} )(t) \|_{\rho,\frac{5}{2}} \, < \, C_0 \, a \, \nt{\ddot{f} \, - \, \ddot{\uf}}_{a,\bT^2} \, \frac{1}{\rho_0 - \rho} \, \Big( 1 - \frac{t}{a (\rho_0 - \rho)} \Big)^{-\frac{1}{2}}. \label{art-estim_graddotf-graddotuf_rho,5/2}
\end{align}
Finally, the estimates related to the time-derivative of the lifting $\psi$ read as follows:
\begin{align}
\forall \, t\in \big[0, a(\rho_0 - \rho)\big), \quad 
& \| \p_t ( \psi \, - \, \upsi )(t) \|_{\rho,3,\sigma} \, < \, C_0 \, a \, \nt{\dot{u}^\pm \, - \, \dot{\uu}^\pm}_{a,\Omega^\pm}, \label{art-estim_dotpsi-dotupsi_rho,3,sigma} \\
\forall \, k\in\{1,2,3\}, \quad 
& \| \p_k\p_t ( \psi \, - \, \upsi )(t) \|_{\rho,3,\sigma} \, < \, C_0 \, a \, \nt{\dot{u}^\pm \, - \, \dot{\uu}^\pm}_{a,\Omega^\pm} \, \frac{1}{\rho_0 - \rho} \, \Big( 1 - \frac{t}{a (\rho_0 - \rho)} \Big)^{-\frac{1}{2}}. \label{art-estim_ddotpsi-ddotupsi_rho,3,sigma}
\end{align}

Equipped with such estimates, we shall be able to highlight in next paragraph a suitable invariant subset of $\bE_a$ (defined by \eqref{s2-def_bE_a}) under $\Upsilon$.

\subsection{Invariant subset under \texorpdfstring{$\Upsilon$}{Upsilon}}\label{art-sec_espace_invariant}

Within this paragraph, we still assume that the initial datum satisfies \eqref{art-hyp_cond_init}, where $\eta_1$ fulfills \eqref{art-hyp_eta_a_petits}. Besides, we still consider $(\dot{u}^\pm,\dot{b}^\pm,\ddot{f}) \in \bE_a$ satisfying the upper bounds \eqref{art-nt_du_db_ddf_R/2a} and we assume that the smallness condition \eqref{art-hyp_eta_a_petits} holds for $a>0$. On the one hand, we have already proved in Section \ref{s2-sec_schema_reso} that $\Upsilon(\dot{u}^\pm,\dot{b}^\pm,\ddot{f})$ satisfies the constraints \eqref{s2-contraintes_u_b_f}. On the other hand, the aim is now to obtain the same upper bounds \eqref{art-nt_du_db_ddf_R/2a} for $\Upsilon(\dot{u}^\pm,\dot{b}^\pm,\ddot{f})$. The result of this paragraph is summarized below.
\begin{prop}\label{art-prop_espace_invariant}
Let $R > 0$ and let $(u_0^\pm,b_0^\pm,f_0)$ be an initial datum satisfying \eqref{art-hyp_cond_init} (in which $\eta_1 > 0$ fulfills the smallness condition \eqref{art-hyp_eta_a_petits}). Moreover, let us assume that $(\dot{u}^\pm,\dot{b}^\pm,\ddot{f}) \in \bE_a$ and satisfies the upper bounds \eqref{art-nt_du_db_ddf_R/2a}. Then there exists a small constant $a>0$ depending on $\eta_0$ and $R$, such that $\Upsilon(\dot{u}^\pm,\dot{b}^\pm,\ddot{f})\in\bE_a$ and satisfies the same upper bounds \eqref{art-nt_du_db_ddf_R/2a}.
\end{prop}
\begin{proof}
Within this proof, $C_0 > 0$ will denote any numerical constant.

Let us recall the notation $\Upsilon(\dot{u}^\pm,\dot{b}^\pm,\ddot{f}) \, = \, (\dot{u}^{\sharp\pm},\dot{b}^{\sharp\pm},\ddot{f}^\sharp)$. In the following, we shall begin by handling the velocity $u^{\sharp\pm}$ defined by \eqref{s2-eq_evo_u^sharp}, and we shall not detail the case of $b^{\sharp\pm}$ which is similar. Let us treat the tangential components $u_i^{\sharp\pm}$, for $i=1,2$, defined by \eqref{s2-eq_evo_u_i^sharp}. Using the Leibniz rule to expand the normal derivative $\p_3$ in \eqref{s2-eq_evo_u_i^sharp}, we can split $u_i^\sharp$ into several parts\footnote{Here, we omit the exponents $\,^\pm$ for the sake of clarity and we still use Einstein's summation convention for repeated indices $k\in\{1,2,3\}$.}:
\begin{subequations}\label{art-expr_u_i^sharp_split}
\begin{align}
u_i^\sharp \, = \,
& - \, \frac{1}{J} \, u_k \, \p_k u_i \, + \, \frac{1}{J} \, b_k \, \p_k b_i \label{art-expr_u_i^sharp_1} \\
& + \frac{1}{J} \, \p_3\p_t\psi \, u_i \, + \, \frac{1}{J} \, \p_t\psi \, \p_3 u_i \, - \, \frac{\p_3 J}{J^2} \, \p_t\psi \, u_i \, + \, \frac{\p_k J}{J^2} \big( u_i \, u_k \, - \,  b_i \, b_k \big)  \label{art-expr_u_i^sharp_2} \\
& \, - \, J \, \p_i Q \, + \, \p_i\psi \, \p_3 Q. \label{art-expr_u_i^sharp_3}
\end{align}
\end{subequations}

\medskip

\noindent
\textbf{\begin{small}$\blacktriangleright$\end{small} Step 1: treatment of \eqref{art-expr_u_i^sharp_1}.} Using the algebra property of $B_{\rho,3,\sigma}(\Omega^\pm)$ and applying Corollary \ref{s2-coro_estim_1/J} together with the preliminary estimates \eqref{art-estim_u^pm_b^pm_B_rho,3,sigma}, \eqref{art-estim_du^pm_db^pm_B_rho,3,sigma} we obtain:
\begin{align}\label{art-estim_u_i^sharp_1}
\forall \, t\in \big[0,a(\rho_0-\rho)\big), \quad \Big\| \frac{1}{J} \, u_k^\pm \, \p_k u_i^\pm \Big\|_{\rho,3,\sigma}^\pm \, < \, C(\eta_0,R) \, \frac{1}{\rho_0-\rho} \, \Big( 1 - \frac{t}{a (\rho_0 - \rho)} \Big)^{-\frac{1}{2}},
\end{align}
where $C(\eta_0,R) > 0$ depends only on $\eta_0$ and $R$. The same estimate holds for the second term in \eqref{art-expr_u_i^sharp_1}.

\medskip

\noindent
\textbf{\begin{small}$\blacktriangleright$\end{small} Step 2: treatment of \eqref{art-expr_u_i^sharp_2}.} Let us handle the third term in \eqref{art-expr_u_i^sharp_2}, namely:
\begin{align*}
\frac{\p_3 J}{J^2} \, \p_t\psi \, u_i \, = \, \frac{\p_3^2\psi}{J^2} \, \p_t\psi \, u_i.
\end{align*}
This term seems to be the most ``dangerous'' one because it contains \textbf{two} spatial derivatives on the lifting $\psi$. However, using Corollary \ref{s2-coro_estim_1/J} together with estimates \eqref{art-estim_ddpsi_rho,3,sigma}, \eqref{art-estim_dotpsi_rho,3,sigma} and \eqref{art-estim_u^pm_b^pm_B_rho,3,sigma} we get the same upper bound as \eqref{art-estim_u_i^sharp_1}, that is:
\begin{align}\label{art-estim_u_i^sharp_2}
\forall \, t\in \big[0,a(\rho_0-\rho)\big), \quad \Big\| \frac{\p_3 J}{J^2} \, \p_t\psi \, u_i^\pm \Big\|_{\rho,3,\sigma}^\pm \, < \, C(\eta_0,R) \, \frac{1}{\rho_0-\rho} \, \Big( 1 - \frac{t}{a (\rho_0 - \rho)} \Big)^{-\frac{1}{2}},
\end{align}
where $C(\eta_0,R) > 0$ depends only on $\eta_0$ and $R$. All the remaining terms in \eqref{art-expr_u_i^sharp_2} provide the same estimate.

\medskip

\noindent
\textbf{\begin{small}$\blacktriangleright$\end{small} Step 3: treatment of \eqref{art-expr_u_i^sharp_3}.} This step turns out to be quite tricky because of the total pressure $Q^\pm$. Let us remind that thanks to \eqref{art-estim_f_B_rho,7/2} we can apply Theorem \ref{s2-thm_estim_Q_final} in order to estimate the pressure. The latter has the same regularity than $u^\pm$ and $b^\pm$ : it belongs to $B_{\rho,3,\sigma}(\Omega^\pm)$. Consequently, the $1-$order derivatives of $Q^\pm$ in \eqref{art-expr_u_i^sharp_3} turn out to be \textit{quasi-linear} terms, so we need to handle these terms in the same way as the velocity and the magnetic field. To do so, we still follow the ideas of \cite{BG77} writing
\begin{align*}
\p_k Q^\pm \, = \, \p_k Q_0^\pm \, + \, \int_0^t \p_k \dot{Q}^\pm(t_1) \, dt_1.
\end{align*}
Applying Lemmas 1 and 2 of \cite{BG77}, we get the estimate:
\begin{align}\label{art-estim_dQ_rho,3,sigma_split}
\forall \, t\in \big[0,a(\rho_0-\rho)\big), \quad 
\| \p_k Q^\pm(t) \|_{\rho,3,\sigma}^\pm \, \leq \, \frac{C_0}{\rho_0 - \rho} \, \| Q_0^\pm \|_{\rho_0,3,\sigma}^\pm \, + \, C_0 \, a \, \nt{\dot{Q}^\pm}_{a,\Omega^\pm} \, \frac{1}{\rho_0 - \rho} \, \Big( 1 - \frac{t}{a (\rho_0 - \rho)} \Big)^{-\frac{1}{2}}.
\end{align}
After some easy but numerous estimates and using \eqref{s2-estim_Q_final} (for $t=0$) together with \eqref{s2-termes_sources_cF_cG_réécrits} and \eqref{art-hyp_cond_init}, we can show that
\begin{align}\label{art-estim_Q_0_rho_0,3,sigma}
\| Q_0^\pm \|_{\rho_0,3,\sigma}^\pm \, < \, \frac{C(\eta_0,R)}{\rho_0 - \rho},
\end{align}
where $C(\eta_0,R) > 0$ depends only on $\eta_0$ and $R$. Afterwards, the main step consists in estimating $\nt{\dot{Q}^\pm}_{a,\Omega^\pm}$ in \eqref{art-estim_dQ_rho,3,sigma_split}, so we need to estimate $\dot{Q}^\pm$ in $B_{\rho,3,\sigma}(\Omega^\pm)$.

Let us apply the time-derivative operator to system \eqref{s2-pb_ellip_pression_réécrit}. Then we obtain:
\begin{equation}\label{s2-pb_ellip_dotQ}
\left\{
\begin{array}{r l}
- \, \dv \big( a \, A^T \, \nabla \dot{Q}^\pm \big) \, = \, J \, \tildcF^\pm, & \text{ in } [0,T]\times\Omega^\pm, \\[0.5ex]
\, [\dot{Q}] \, = \, 0, & \text{ on } [0,T]\times\Gamma, \\[0.5ex]
\big( 1 + |\nabla' f|^2 \big) \, \big[\p_3 \dot{Q}] \, = \, \tildcG, & \text{ on } [0,T]\times\Gamma, \\[0.5ex]
\p_3 \dot{Q}^\pm \, = \, 0, & \text{ on } [0,T]\times\Gamma_\pm,
\end{array}
\right.
\end{equation}
where the source terms $J\,\tildcF^\pm$ and $\tildcG$ are defined by:
\begin{subequations}\label{art-def_tildcF^pm_tildcG}
\begin{align}
J\,\tildcF^\pm & \, := \, \p_t\big( J \, \cF^\pm \big) \, + \, \dv\big( \p_t(aA^T) \, \nabla Q^\pm \big), \label{art-def_tildcF^pm} \\
\tildcG & \, := \, \p_t\cG \, - \, \p_t\big( |\nabla'f|^2 \big) \, \big[ \p_3 Q \big].
\end{align}
\end{subequations}
The matrix $aA^T$ reads as follows (we refer to \eqref{s2-def_J_A_a}):
\begin{align}\label{art-expr_aA^T}
aA^T \, = \, 
\begin{pmatrix}
J & 0 & -\p_1\psi \\
0 & J & -\p_2\psi \\
-\p_1\psi & -\p_2\psi & \frac{1 + |\nabla'\psi|^2}{J}
\end{pmatrix},
\end{align}
and its coefficients will be denoted by $\alpha_{k\ell}$. Now, using Theorem \ref{s2-thm_estim_Q_final} we get the estimate:
\begin{equation}\label{art-estim_dotQ_rho,3,sigma}
\begin{aligned}
\| \dot{Q}^\pm \|_{\rho,3,\sigma}^\pm \, \leq \, C(\eta_0) \, \Big( \big\| \p_t\cF^\pm \big\|_{\rho,1,\sigma}^\pm \, + \, \Big\| \frac{\p_t J}{J} \, \cF^\pm \Big\|_{\rho,1,\sigma}^\pm & \, + \, \Big\| \frac{1}{J} \, \p_k\big( \p_t\alpha_{k\ell} \, \p_\ell Q^\pm \big) \Big\|_{\rho,1,\sigma}^\pm \\
& \, + \, \big\| \p_t\cG \big\|_{\rho,\frac{3}{2}} \, + \, \big\| \p_t( |\nabla' f|^2 ) \, [\p_3 Q] \big\|_{\rho,\frac{3}{2}} \Big),
\end{aligned}
\end{equation}
where $C(\eta_0) > 0$ depends only on $\eta_0$. In \eqref{art-estim_dotQ_rho,3,sigma} we have used Einstein's summation convention for repeated indices $k,\ell\in\{1,2,3\}$. It remains to handle each quantity in the right-hand side of \eqref{art-estim_dotQ_rho,3,sigma} to complete the estimate of $\dot{Q}^\pm$ in $B_{\rho,3,\sigma}(\Omega^\pm)$. Although the computations turn out to be quite heavy because of the numerous terms in $\cF^\pm$ and $\cG$ (see \eqref{s2-termes_sources_cF_cG_réécrits}), we shall not give the full details. Nevertheless, the method we use to estimate each term is completely analogous from one term to an other, so we can allow ourself to shorten some parts in the proof.

The first two terms in the right-hand side of \eqref{art-estim_dotQ_rho,3,sigma} can actually be estimated in $B_{\rho,2,\sigma}(\Omega^\pm)$ because $\cF^\pm$ contains $1-$order derivatives of $u^\pm$ and $b^\pm$, and $2-$order derivatives of $\psi$. To begin with, let us treat for instance the first term in the right-hand side of \eqref{art-estim_dotQ_rho,3,sigma}.

\medskip

\noindent
\textbf{\begin{small}$\blacklozenge$\end{small} Sub-step 3.1: the first term in \eqref{art-estim_dotQ_rho,3,sigma}.} The expression of $\p_t\cF^\pm$ can be easily deduced from \eqref{s2-terme_source_JcF_réécrit}, but we do not write it down for the sake of clarity. For instance, let us treat both following terms which appear when we differentiate $\cF^\pm$ with respect to $t$. Likewise, all the other terms defining $\p_t\cF^\pm$ can be estimated using the preliminaries in \mbox{Paragraph \ref{s2-sec_prelim_estim}}.

Given $k,\ell\in\{1,2,3\}$, let us set
\begin{align*}
\cT_1^\pm \, := \, \frac{1}{J^2} \, \p_k\dot{u}_\ell^\pm \, \p_\ell u_k^\pm.
\end{align*}
We directly estimate $\cT_1$ in $B_{\rho,2,\sigma}(\Omega^\pm)$ as follows:
\begin{align*}
\| \cT_1^\pm \|_{\rho,2,\sigma}^\pm \, \leq \, C_0 \, \Big\| \frac{1}{J} \Big\|_{\rho,2,\sigma}^2 \, \big\| \dot{u}_\ell^\pm \big\|_{\rho,3,\sigma}^\pm \, \big\| u_k^\pm \big\|_{\rho,3,\sigma}^\pm.
\end{align*}
Then, using Corollary \ref{s2-coro_estim_1/J} together with \eqref{art-estim_dotu_dotb_rho,3,sigma} and \eqref{art-estim_u^pm_b^pm_B_rho,3,sigma}, we get:
\begin{align}\label{art-estim_cT_1}
\forall \, t\in \big[ 0, a(\rho_0-\rho) \big), \quad \| \cT_1^\pm \|_{\rho,2,\sigma}^\pm \, \leq \, C(\eta_0,R) \, \frac{1}{a \, (\rho_0 - \rho)} \, \Big( 1 - \frac{t}{a (\rho_0 - \rho)} \Big)^{-\frac{1}{2}}.
\end{align}
In the end, since the idea will consist in choosing $a$ small enough, the factor $\frac{1}{a}$ in \eqref{art-estim_cT_1} seems not to be appropriate. However, it turns out to be harmless because $a$ already appears in \eqref{art-estim_dQ_rho,3,sigma_split}.

We finish the treatment of $\p_t \cF^\pm$ estimating the following second term $\cT_2^\pm$ defined by
\begin{align*}
\cT_2^\pm \, := \, \frac{\p_k\p_t J}{J^3} \, u_i^\pm \, \p_i u_k^\pm,
\end{align*}
where $i\in\{1,2\}$ and $k\in\{1,2,3\}$. As previously, we have:
\begin{align*}
\| \cT_2^\pm \|_{\rho,2,\sigma}^\pm \, \leq \, C_0 \, \Big\| \frac{1}{J} \Big\|_{\rho,2,\sigma}^3 \, \big\| \p_3\p_t\psi \big\|_{\rho,3,\sigma} \, \big\| u_i^\pm \big\|_{\rho,2,\sigma}^\pm \, \big\| u_k^\pm \big\|_{\rho,3,\sigma}^\pm.
\end{align*}
Thus, using Corollary \ref{s2-coro_estim_1/J} and estimates \eqref{art-estim_d_dotpsi_rho,3,sigma}, \eqref{art-estim_u^pm_b^pm_B_rho,3,sigma}, we end up with:
\begin{align}\label{art-estim_cT_2}
\forall \, t\in \big[ 0, a(\rho_0-\rho) \big), \quad \| \cT_2^\pm \|_{\rho,2,\sigma}^\pm \, \leq \, C(\eta_0,R) \, \frac{1}{\rho_0 - \rho} \, \Big( 1 - \frac{t}{a (\rho_0 - \rho)} \Big)^{-\frac{1}{2}}.
\end{align}

Gathering \eqref{art-estim_cT_1} and \eqref{art-estim_cT_2} (in which we use the upper bound $1\leq\frac{1}{a}$), and proceeding in the same way with the remaining terms in $\p_t\cF^\pm$, we can conclude that the first term in the right-hand side of \eqref{art-estim_dotQ_rho,3,sigma} is estimated as follows:
\begin{align}\label{art-estim_dotQ_rho,3,sigma_step3.1}
\forall \, t\in \big[ 0, a(\rho_0 - \rho) \big), \quad 
\big\| \p_t\cF^\pm \big\|_{\rho,1,\sigma}^\pm 
\, \leq \, \big\| \p_t\cF^\pm \big\|_{\rho,2,\sigma}^\pm 
\, \leq \, C(\eta_0,R) \, \frac{1}{a \, (\rho_0 - \rho)} \, \Big( 1 - \frac{t}{a (\rho_0 - \rho)} \Big)^{-\frac{1}{2}},
\end{align}
where $C(\eta_0,R) > 0$ depends only on $\eta_0$ and $R$. The second term in the right-hand side of \eqref{art-estim_dotQ_rho,3,sigma} is completely analogous to $\p_t\cF^\pm$ and provides the same estimate as \eqref{art-estim_dotQ_rho,3,sigma_step3.1}, so we feel free to skip the details.

\medskip

\noindent
\textbf{\begin{small}$\blacklozenge$\end{small} Sub-step 3.2: the pressure term in \eqref{art-estim_dotQ_rho,3,sigma}.} Now, let us focus on the pressure term that we split as follows:
\begin{align}\label{art-def_cT_3_cT_4}
\frac{1}{J} \, \p_k\big( \p_t\alpha_{k\ell} \, \p_\ell Q^\pm \big) 
& \, = \, \frac{1}{J} \, \p_k\p_t\alpha_{k\ell} \, \p_\ell Q^\pm \, + \, \frac{1}{J} \, \p_t\alpha_{k\ell} \, \p_k\p_\ell Q^\pm \nonumber \\
& \, =: \, \cT_3^\pm \, + \, \cT_4^\pm.
\end{align}
Using \eqref{art-estim_dpsi_rho,3,sigma} and \eqref{art-estim_d_dotpsi_rho,3,sigma} we can easily estimate the coefficients $\p_t\alpha_{k\ell}$:
\begin{align}\label{art-estim_dot_alpha_kl}
\forall \, k,\ell\in\{1,2,3\}, \quad \forall \, t\in \big[ 0, a(\rho_0-\rho) \big), \quad \| \p_t \alpha_{k\ell}(t) \|_{\rho,3,\sigma} \, \leq \, C(\eta_0,R) \, \frac{1}{\rho_0 - \rho} \, \Big( 1 - \frac{t}{a (\rho_0 - \rho)} \Big)^{-\frac{1}{2}}.
\end{align}

In order to estimate $\cT_3^\pm$, we can use the trivial inequality $\|\cdot\|_{\rho,1,\sigma}^\pm \leq \|\cdot\|_{\rho,2,\sigma}^\pm$. Indeed, we get:
\begin{align*}
\| \cT_3^\pm \|_{\rho,1,\sigma}^\pm \, \leq \, C_0 \, \Big\| \frac{1}{J} \Big\|_{\rho,2,\sigma} \, \big\| \p_t\alpha_{kl} \big\|_{\rho,3,\sigma} \, \big\| Q^\pm \|_{\rho,3,\sigma}^\pm.
\end{align*}
Thus, using Corollary \ref{s2-coro_estim_1/J} together with \eqref{art-estim_dot_alpha_kl}, we have:
\begin{align*}
\forall \, t\in \big[ 0, a(\rho_0-\rho) \big), \quad
\| \cT_3^\pm \|_{\rho,1,\sigma}^\pm \, \leq \, C(\eta_0,R) \, \frac{1}{\rho_0 - \rho} \, \Big( 1 - \frac{t}{a (\rho_0 - \rho)} \Big)^{-\frac{1}{2}} \big\| Q^\pm(t) \|_{\rho,3,\sigma}^\pm.
\end{align*}
It remains to estimate $Q^\pm(t)$ in $B_{\rho,3,\sigma}(\Omega^\pm)$. Using Theorem \ref{s2-thm_estim_Q_final}, definitions \eqref{s2-termes_sources_cF_cG_réécrits} and estimates stated in \mbox{Paragraph \ref{s2-sec_prelim_estim}}, we can easily deduce the following bound on $Q^\pm(t)$:
\begin{align}\label{art-estim_Q^pm_rho,3,sigma}
\forall \, t\in \big[ 0, a(\rho_0-\rho) \big), \quad \| Q^\pm(t) \|_{\rho,3,\sigma}^\pm \, \leq \, C(\eta_0,R),
\end{align}
where $C(\eta_0,R) > 0$ depends only on $\eta_0$ and $R$. It is important to notice that the singularity $\big( 1 - \frac{t}{a (\rho_0 - \rho)} \big)^{-\frac{1}{2}}$ does not appear in \eqref{art-estim_Q^pm_rho,3,sigma} because the source terms $J\,\cF^\pm$ and $\cG$ in \eqref{s2-pb_ellip_pression_réécrit} only contain $1-$order derivatives of $(u^\pm, b^\pm)$ and $2-$order derivatives of $\psi$. Therefore, since $\cF^\pm$ and $\cG$ are estimated in $B_{\rho,1,\sigma}(\Omega^\pm)$ and $B_{\rho,\frac{3}{2}}(\bT^2)$ respectively, we only need to estimate $(u^\pm, b^\pm)$ in $B_{\rho,3,\sigma}(\Omega^\pm)$ and $\psi$ in $B_{\rho,4,\sigma}(\Omega)$. In virtue of \eqref{art-estim_u^pm_b^pm_B_rho,3,sigma}, \eqref{art-estim_psi_rho,3,sigma} and \eqref{art-estim_dpsi_rho,3,sigma} these functions are bounded, which explains estimate \eqref{art-estim_Q^pm_rho,3,sigma}. We end up with the following estimate of $\cT_3^\pm$:
\begin{align}\label{art-estim_cT_3}
\forall \, t\in \big[ 0, a(\rho_0-\rho) \big), \quad \| \cT_3^\pm \|_{\rho,1,\sigma}^\pm \, \leq \, C(\eta_0,R) \, \frac{1}{\rho_0 - \rho} \, \Big( 1 - \frac{t}{a (\rho_0 - \rho)} \Big)^{-\frac{1}{2}}.
\end{align}

The term $\cT_4^\pm$ is a bit more tricky, because it contains two derivatives of $Q^\pm$. Let us write $\p^m$ to denote any spatial derivative of order $m$ ; then it is easy to prove the following lemma.
\begin{lemme}\label{art-lemme_produit_H^1}
Let $u\in H^2(\Omega)$ and $v\in H^3(\Omega)$, then we have $u \, \p^2 v \in H^1(\Omega)$ with the estimate:
\begin{align*}
\| u \, \p^2 v \|_{H^1(\Omega)} \, \leq \, C_0 \, \|u\|_{H^2(\Omega)} \, \| v \|_{H^3(\Omega)},
\end{align*}
where $C_0 > 0$ is a numerical constant.
\end{lemme}
\noindent A consequence of Lemma \ref{art-lemme_produit_H^1} is the following estimate, whose proof is completely analogous to \eqref{s1-prod_rho,r,sigma}:
\begin{align}\label{art-estim_produit_B_rho,1,sigma}
\forall \, u\in B_{\rho,2,\sigma}(\Omega), \quad \forall \, v\in B_{\rho,3,\sigma}(\Omega), \quad
\| u \, \p^2 v \|_{\rho,1,\sigma} \, \leq \, C_0 \, \|u\|_{\rho,2,\sigma} \, \|v\|_{\rho,3,\sigma}.
\end{align}
Now, using \eqref{art-estim_produit_B_rho,1,sigma} we obtain:
\begin{align*}
\| \cT_4^\pm \|_{\rho,1,\sigma}^\pm \, \leq \, C_0 \, \Big\| \frac{1}{J} \Big\|_{\rho,2,\sigma} \, \big\|\p_t\alpha_{k\ell}\big\|_{\rho,2,\sigma} \, \big\| Q^\pm \big\|_{\rho,3,\sigma}^\pm.
\end{align*}
Eventually, gathering Corollary \ref{s2-coro_estim_1/J}, \eqref{art-estim_dot_alpha_kl} and \eqref{art-estim_Q^pm_rho,3,sigma}, we can conclude that
\begin{align}\label{art-estim_cT_4}
\forall \, t\in \big[ 0, a(\rho_0-\rho) \big), \quad \| \cT_4^\pm \|_{\rho,1,\sigma}^\pm \, \leq \, C(\eta_0,R),
\end{align}
which is ``better'' than \eqref{art-estim_cT_3} because the singularity $\big( 1 - \frac{t}{a (\rho_0 - \rho)} \big)^{-\frac{1}{2}}$ does not appear.

Using \eqref{art-estim_cT_3} and \eqref{art-estim_cT_4}, the pressure term in \eqref{art-estim_dotQ_rho,3,sigma} is finally estimated as follows:
\begin{align}\label{art-estim_dotQ_rho,3,sigma_step3.2}
\forall \, t\in \big[ 0, a(\rho_0 - \rho) \big), \quad 
\Big\| \frac{1}{J} \, \p_k\big( \p_t\alpha_{k\ell} \, \p_\ell Q^\pm \big)  \Big\|_{\rho,1,\sigma}^\pm 
\, \leq \, C(\eta_0,R) \, \frac{1}{\rho_0 - \rho} \, \Big( 1 - \frac{t}{a (\rho_0 - \rho)} \Big)^{-\frac{1}{2}},
\end{align}
where $C(\eta_0,R) > 0$ depends only on $\eta_0$ and $R$.

\medskip

\noindent
\textbf{\begin{small}$\blacklozenge$\end{small} Sub-step 3.3: the last two terms in \eqref{art-estim_dotQ_rho,3,sigma}.} The boundary terms $\p_t\cG$ and $\p_t( |\nabla' f|^2 ) \, [\p_3 Q]$ in \eqref{art-estim_dotQ_rho,3,sigma} can be handled using the same arguments as in \textbf{Sub-step 3.1} and \textbf{Sub-step 3.2}. Indeed, using the trace estimate \eqref{art-estim_trace} (with $r=2$), it suffices to estimate the terms defining $\p_t\cG$ in $B_{\rho,2,\sigma}(\Omega^\pm)$, which is exactly what we did previously. Consequently, we feel free to skip these details and we give the final estimate:
\begin{align}\label{art-estim_dotQ_rho,3,sigma_step3.3}
\big\| \p_t\cG \big\|_{\rho,\frac{3}{2}} \, + \, \big\| \p_t( |\nabla' f|^2 ) \, [\p_3 Q] \big\|_{\rho,\frac{3}{2}}
\, \leq \, C(\eta_0,R) \, \frac{1}{a\,(\rho_0-\rho)} \, \Big( 1 - \frac{t}{a (\rho_0 - \rho)} \Big)^{-\frac{1}{2}},
\end{align}
where $C(\eta_0,R) > 0$ depends only on $\eta_0$ and $R$.

Let us come back to estimate \eqref{art-estim_dotQ_rho,3,sigma} of $\dot{Q}^\pm$ ; gathering \eqref{art-estim_dotQ_rho,3,sigma_step3.1}, \eqref{art-estim_dotQ_rho,3,sigma_step3.2} and \eqref{art-estim_dotQ_rho,3,sigma_step3.3}, we end up with the following result:
\begin{equation}\label{art-estim_dotQ_rho,3,sigma_final}
\forall \, t\in \big[ 0,a(\rho_0-\rho) \big), \quad 
\| \dot{Q}^\pm(t) \|_{\rho,3,\sigma}^\pm \, \leq \, C(\eta_0,R) \, \frac{1}{a \, (\rho_0 - \rho)} \, \Big( 1 - \frac{t}{a (\rho_0 - \rho)} \Big)^{-\frac{1}{2}}.
\end{equation}

It allows us to conclude about estimate \eqref{art-estim_dQ_rho,3,sigma_split}: using \eqref{art-estim_Q_0_rho_0,3,sigma} and \eqref{art-estim_dotQ_rho,3,sigma_final} we finally obtain:
\begin{align}\label{art-estim_dQ_rho,3,sigma}
\forall \, t\in \big[0,a(\rho_0-\rho)\big), \quad 
\| \p_k Q^\pm(t) \|_{\rho,3,\sigma}^\pm \, \leq \, C(\eta_0,R) \, \frac{1}{\rho_0 - \rho} \, \Big( 1 - \frac{t}{a (\rho_0 - \rho)} \Big)^{-\frac{1}{2}},
\end{align}
where $C(\eta_0,R) > 0$ depends only on $\eta_0$ and $R$.

We are now able to achieve \textbf{Step 3}. Thanks to \eqref{art-estim_dpsi_rho,3,sigma} and \eqref{art-estim_dQ_rho,3,sigma}, the pressure terms in \eqref{art-expr_u_i^sharp_3} satisfy the following estimate:
\begin{align}\label{art-estim_u_i^sharp_3}
\forall \, t\in \big[0,a(\rho_0-\rho)\big), \quad \big\| J \, \p_i Q^\pm \, + \, \p_i\psi \, \p_3 Q^\pm \big\|_{\rho,3,\sigma}^\pm \, < \, C(\eta_0,R) \, \frac{1}{\rho_0-\rho} \, \Big( 1 - \frac{t}{a (\rho_0 - \rho)} \Big)^{-\frac{1}{2}}.
\end{align}

\medskip

\noindent
\textbf{Conclusion.} This concludes the case of the tangential components $u_i^{\sharp\pm}$ for $i=1,2$; the case of the normal component $u_3^{\sharp\pm}$ given by \eqref{s2-eq_evo_u_3} is analogous. Eventually, recalling definition \eqref{s2-def_nt_a_Omega} we have:
\begin{equation}\label{art-estim_nt_u_i^sharp,pm_final}
\nt{\dot{u}^{\sharp\pm}}_{a,\Omega^\pm} \, \leq \, C_2(\eta_0,R),
\end{equation}
where $C_2(\eta_0,R) > 0$ is a constant depending only on $\eta_0$ and $R$. The same estimate also holds for the magnetic field $b^{\sharp\pm}$. Besides, a straightforward consequence of \eqref{art-estim_nt_u_i^sharp,pm_final} is the following estimate of $\ddot{f}^\sharp$, where $\uC > 0$ is the same constant as in \eqref{art-nt_ddf_uCR/2a}:
\begin{align*}
\nt{\ddot{f}^\sharp}_{a,\bT^2} \, \leq \, \uC \, C_2(\eta_0,R).
\end{align*}
Recalling the first smallness condition \eqref{art-hyp_eta_a_petits}, we eventually choose $a>0$ such that:
\begin{align}\label{art-hyp_a_petit}
a \, \leq \, \min\Big\{ \frac{R}{4 C_2(\eta_0,R)} \, , \, \frac{\eta_0}{2C_1R} \Big\},
\end{align}
so that $(\dot{u}^{\sharp,\pm},\dot{b}^{\sharp\pm},\ddot{f}^\sharp)$ satisfies the upper bounds \eqref{art-nt_du_db_ddf_R/2a}; Proposition \ref{art-prop_espace_invariant} is proved.
\end{proof}

\subsection{Contraction of \texorpdfstring{$\Upsilon$}{Upsilon}}\label{art-sec_contraction}

Now we own a suitable invariant complete metric space under $\Upsilon$, the purpose of this paragraph is to prove that $\Upsilon$ is a contraction map (decreasing $a>0$ one more time if necessary). The main result of this paragraph is stated in the following proposition.
\begin{prop}\label{art-prop_Upsilon_contractante}
Let $R>0$ and let $(u_0^\pm,b_0^\pm,f_0)$ be an initial datum satisfying \eqref{art-hyp_cond_init} (where $\eta$ fulfills \eqref{art-hyp_eta_a_petits}). We consider $(\dot{u}^\pm,\dot{b}^\pm,\ddot{f})$ and $(\dot{\uu}^\pm,\dot{\ub}^\pm,\ddot{\uf})$ belonging to $\bE_a$ and satisfying the upper bounds \eqref{art-nt_du_db_ddf_R/2a}. Then there exists $C_3(\eta_0,R)>0$, depending only on $\eta_0$ and $R$, such that:
\begin{equation}\label{art-estim_prop_Upsilon_contractante}
\begin{aligned}
\nt{\dot{u}^{\sharp\pm} \, - \, \dot{\uu}^{\sharp\pm}}_{a,\Omega^\pm} & \, + \, \nt{\dot{b}^{\sharp\pm} \, - \, \dot{\ub}^{\sharp\pm}}_{a,\Omega^\pm} \, + \, \nt{\ddot{f}^\sharp \, - \, \ddot{\uf}^\sharp}_{a,\bT^2} \\
& \hspace*{2.2cm} \, \leq \, C_3(\eta_0,R) \, a \, \Big( \nt{\dot{u}^\pm \, - \, \dot{\uu}^\pm}_{a,\Omega^\pm} \, + \, \nt{\dot{b}^\pm \, - \, \dot{\ub}^\pm}_{a,\Omega^\pm} \, + \, \nt{\ddot{f} \, - \, \ddot{\uf}}_{a,\bT^2} \Big).
\end{aligned}
\end{equation}
\end{prop}
Consequently, using Proposition \ref{art-prop_Upsilon_contractante}, we can easily prove that $\Upsilon$ is a contraction map. For instance, assuming $a\leq (2C_3(\eta_0,R))^{-1}$, we can conclude that
\begin{align*}
\nt{\dot{u}^{\sharp\pm} \, - \, \dot{\uu}^{\sharp\pm}}_{a,\Omega^\pm} \, + \, \nt{\dot{b}^{\sharp\pm} \, - \, \dot{\ub}^{\sharp\pm}}_{a,\Omega^\pm} & \, + \, \nt{\ddot{f}^\sharp \, - \, \ddot{\uf}^\sharp}_{a,\bT^2} \\
& \, \leq \, \frac{1}{2} \, \Big( \nt{\dot{u}^\pm \, - \, \dot{\uu}^\pm}_{a,\Omega^\pm} \, + \, \nt{\dot{b}^\pm \, - \, \dot{\ub}^\pm}_{a,\Omega^\pm} \, + \, \nt{\ddot{f} \, - \, \ddot{\uf}}_{a,\bT^2} \Big).
\end{align*}
Therefore we are able to apply the classical Banach fixed-point theorem, which gives a unique fixed point of $\Upsilon$. Obviously, the latter is a solution of system \eqref{s2-eq_evo_u_b}, \eqref{s2-contraintes_div_cond_bord_u_b}: Theorem \ref{s2-thm_solutions_analytiques_nappes} is eventually proved.

Let us now prove Proposition \ref{art-prop_Upsilon_contractante}.
\begin{proof}
As we did in the proof of Proposition \ref{art-prop_espace_invariant}, we shall only treat the tangential components $u_i^{\sharp\pm} - \uu_i^{\sharp\pm}$, for $i=1,2$, whose definition is given by \eqref{s2-eq_evo_u_i^sharp}. Contrary to what has been done in the proof of Proposition \ref{art-prop_espace_invariant}, we will have to make the difference $\dot{u}^\pm - \dot{\uu}^\pm$ appears and so on for the magnetic field and the front. The ``disadvantage'' is that it leads to heavy computations, so we will feel free to skip some details. Afterwards, we shall estimate each term using the preliminary estimates given by \eqref{art-estim_u-uu_rho,3,sigma} until \eqref{art-estim_ddotpsi-ddotupsi_rho,3,sigma}.

We begin with the same expansion \eqref{art-expr_u_i^sharp_split} (with underscores for $\uu_i^\sharp$) and we handle each contribution in three steps again.

\medskip

\noindent
\textbf{\begin{small}$\blacktriangleright$\end{small} Step 1: the convective terms in \eqref{art-expr_u_i^sharp_1}.} We want to estimate the following term in $B_{\rho,3,\sigma}(\Omega^\pm)$:
\begin{align*}
\sT_1 \, := \, \frac{1}{J} \, u_k \, \p_k u_i \, - \, \frac{1}{\uJ} \, \uu_k \, \p_k \uu_i,
\end{align*}
where we have omitted the exponents $\,^\pm$. In order to exhibit the differences $u - \uu$ and $\psi - \upsi$, we rewrite $\mathscr{T}_1$ as follows:
\begin{align}
\sT_1 & \, = \, \frac{1}{J} \, u_k \, \p_k(u_i - \uu_i) \, + \, \frac{1}{J} \, \p_k\uu_i \, (u_k - \uu_k) \, - \, \frac{1}{J\,\uJ} \, \uu_k \, \p_k\uu_i \, \p_3(\psi - \upsi) \nonumber \\
& \, =: \, \sT_{11} \, + \, \sT_{12} \, + \, \sT_{13}. \label{art-decomp_sT_1}
\end{align}
Then, we only illustrate the case of $\sT_{11}$ since both other terms $\sT_{12}$ and $\sT_{13}$ are similar. Using Corollary \ref{s2-coro_estim_1/J} together with estimates \eqref{art-estim_u^pm_b^pm_B_rho,3,sigma} and \eqref{art-estim_du-duu_rho,3,sigma}, we obtain:
\begin{align*}
\forall \, t\in\big[0,a(\rho_0-\rho)\big), \quad
\big\| \sT_{11} \big\|_{\rho,3,\sigma}^\pm \, \leq \, C(\eta_0,R) \, a \, \Big( 1 - \frac{t}{a(\rho_0-\rho)} \Big)^{-\frac{1}{2}} \, \nt{\dot{u}^\pm \, - \, \dot{\uu}^\pm}_{a,\Omega^\pm},
\end{align*}
where $C(\eta_0,R) > 0$ depends only on $\eta_0$ and $R$. Proceeding similarly for $\sT_{12}$ and $\sT_{13}$, we thus have:
\begin{equation}\label{art-estim_sT_1}
\begin{aligned}
\forall \, t\in\big[0,a(\rho_0-\rho)\big), \quad
\big\| \sT_1 \big\|_{\rho,3,\sigma}^\pm \, \leq \, C(\eta_0,R) \, \frac{a}{\rho_0-\rho} \Big( 1 - & \frac{t}{a(\rho_0-\rho)} \Big)^{-\frac{1}{2}} \\
& \times \Big( \nt{\dot{u}^\pm \, - \, \dot{\uu}^\pm}_{a,\Omega^\pm} \, + \, \nt{\ddot{f} \, - \, \ddot{\uf}}_{a,\bT^2} \Big).
\end{aligned}
\end{equation}
The term associated with the magnetic field in \eqref{art-expr_u_i^sharp_1} is analogous and provides the same estimate as \eqref{art-estim_sT_1}, in which we replace $\dot{u}^\pm \, - \, \dot{\uu}^\pm$ by $\dot{b}^\pm \, - \, \dot{\ub}^\pm$.

\medskip

\noindent
\textbf{\begin{small}$\blacktriangleright$\end{small} Step 2: the terms in \eqref{art-expr_u_i^sharp_2}.} These terms are a bit more tricky because some of them contain \textbf{two} spatial derivatives of the lifting $\psi$. As we did in the proof of Proposition \ref{art-prop_espace_invariant}, we only deal with the third term in \eqref{art-expr_u_i^sharp_2}. Let us set:
\begin{align*}
\sT_2 \, := \, \frac{\p_3 J}{J^2} \, \p_t\psi \, u_i \, - \, \frac{\p_3 \uJ}{\uJ^2} \, \p_t\upsi \, \uu_i.
\end{align*}
We rewrite $\sT_2$ using the same method as before:
\begin{align}
\sT_2 & \, = \, \frac{1}{J^2} \, \p_t\psi \, u_i \, \p_3^2(\psi - \upsi) \, + \, \frac{\p_3\uJ}{J^2} \, \p_t\psi \, (u_i - \uu_i) \, + \, \frac{\p_3\uJ}{J^2} \, \uu_i \, \p_t(\psi - \upsi) \, - \, \frac{J + \uJ}{J^2\,\uJ^2} \, \p_3\uJ \, \p_t\upsi \, \uu_i \, \p_3(\psi - \upsi) \nonumber \\
& \, =: \, \sT_{21} \, + \, \sT_{22} \, + \, \sT_{23} \, + \, \sT_{24}. \label{art-decomp_sT_2}
\end{align}
Using Corollary \ref{s2-coro_estim_1/J} and estimates \eqref{art-estim_dotpsi_rho,3,sigma}, \eqref{art-estim_u^pm_b^pm_B_rho,3,sigma}, \eqref{art-estim_ddpsi-ddupsi_rho,3,sigma}, we get the following upper bound for $\sT_{21}$:
\begin{align*}
\forall \, t\in\big[0,a(\rho_0-\rho)\big), \quad
\big\| \sT_{21} \big\|_{\rho,3,\sigma}^\pm \, \leq \, C(\eta_0,R) \, \frac{a^2}{\rho_0-\rho} \, \Big( 1 - \frac{t}{a(\rho_0-\rho)} \Big)^{-\frac{1}{2}} \, \nt{\ddot{f} \, - \, \ddot{\uf}}_{a,\bT^2}.
\end{align*}
Proceeding similarly for $\sT_{22}, \sT_{23}$ and $\sT_{24}$, we end up with:
\begin{equation}\label{art-estim_sT_2}
\begin{aligned}
\forall \, t\in\big[0,a(\rho_0-\rho)\big), \quad
\big\| \sT_2 \big\|_{\rho,3,\sigma}^\pm \, \leq \, C(\eta_0,R) \, & \frac{a}{\rho_0-\rho} \, \Big( 1 - \frac{t}{a(\rho_0-\rho)} \Big)^{-\frac{1}{2}} \\
& \times \Big( \nt{\dot{u}^\pm \, - \, \dot{\uu}^\pm}_{a,\Omega^\pm} \, + \, \nt{\dot{b}^\pm \, - \, \dot{\ub}^\pm}_{a,\Omega^\pm} \, + \, \nt{\ddot{f} \, - \, \ddot{\uf}}_{a,\bT^2} \Big),
\end{aligned}
\end{equation}
where we have used the inequality $a^2 \leq a$ because we shall only need one power of $a$ eventually.

\medskip

\noindent
\textbf{\begin{small}$\blacktriangleright$\end{small} Step 3: the pressure terms in \eqref{art-expr_u_i^sharp_3}.} For instance, let us focus on the second term in \eqref{art-expr_u_i^sharp_3}. We set:
\begin{align*}
\sT_3 \, := \, \p_i\psi \, \p_3 Q \, - \, \p_i\upsi \, \p_3 \uQ,
\end{align*}
and we split it as follows:
\begin{align}
\sT_3 & \, = \, \p_i\psi \, \p_3(Q - \uQ) \, + \, \p_3\uQ \, \p_i(\psi - \upsi) \nonumber \\
& \, =: \, \sT_{31} \, + \, \sT_{32}. \label{art-decomp_sT_3}
\end{align}

\medskip

\noindent
\textbf{\begin{small}$\blacklozenge$\end{small} Sub-step 3.1: estimate of $\sT_{31}$.} In order to use a similar estimate as \eqref{s2-estim_Q_final}, we need to exhibit a suitable Laplace problem satisfied by $P^\pm \, := \, Q^\pm - \uQ^\pm$. Starting from system \eqref{s2-pb_ellip_pression} (with underscores for $\uQ^\pm$), we obtain a similar elliptic problem satisfied by $P^\pm$. The latter reads:
\begin{equation}\label{s2-pb_ellip_P^pm}
\left\{
\begin{array}{r l}
- \, A_{ji} \, \p_j\big( A_{ki} \, \p_k P^\pm \big) \, = \, \bF^\pm, & \text{ in } [0,T]\times\Omega^\pm, \\[0.5ex]
\, [P] \, = \, 0, & \text{ on } [0,T]\times\Gamma, \\[0.5ex]
\big( 1 + |\nabla' f|^2 \big) \, \big[\p_3 P] \, = \, \bG, & \text{ on } [0,T]\times\Gamma, \\[0.5ex]
\p_3 P^\pm \, = \, 0, & \text{ on } [0,T]\times\Gamma_\pm.
\end{array}
\right.
\end{equation}
Now, let us give the expressions of the source terms $\bF^\pm$ and $\bG$ in \eqref{s2-pb_ellip_P^pm}. To do so, we make a link between both elliptic problems satisfied by $Q^\pm$ and $\uQ^\pm$ writing
\begin{align}\label{s2-def_R}
A = \uA + R,
\end{align} 
where $R$ is a ``remainder'' term given by
\begin{align*}
R \, := \, A \, - \, \uA \, = \, 
\begin{pmatrix}
0 & 0 & 0 \\
0 & 0 & 0 \\
R_{31} & R_{32} & R_{33}
\end{pmatrix},
\end{align*}
and the coefficients $R_{3j}$ are defined as follows:
\begin{align*}
\forall \, j = 1,2,3, \quad 
R_{3j} \, := \, - \, \frac{1}{\uJ} \, \p_j \big( \psi - \upsi \big) \, + \, \frac{\p_j\psi}{J \, \uJ} \, \p_3 \big( \psi - \upsi \big).
\end{align*}
Plugging \eqref{s2-def_R} into the first equation of \eqref{s2-pb_ellip_P^pm}, we can express the source terms $\bF^\pm$:
\begin{subequations}\label{s2-def_termes_sources_bF_bG}
\begin{equation}\label{s2-def_terme_source_bF}
\bF^\pm \, := \, \cF^\pm \, - \, \ucF^\pm \, + \, A_{ji} \, \p_j\big( R_{3i} \, \p_3 Q^\pm \big) \, + \, R_{3i} \, \p_3 \big( \uA_{ki} \, \p_k Q^\pm \big).
\end{equation}
Afterwards, setting $g := f - \uf$ and writing $\nabla' f \, = \, \nabla' \uf \, + \, \nabla' g$ in the third equation of \eqref{s2-pb_ellip_P^pm}, the source term $\bG$ reads:
\begin{equation}\label{s2-def_terme_source_bG}
\bG \, := \, \cG \, - \, \ucG \, - \, \big(2 \, \nabla'f\cdot\nabla' g \, + \, |\nabla' g|^2 \big) \, \big[\, \p_3 \uQ \, \big].
\end{equation}
\end{subequations}

We recall that we need to estimate $\p_3 P^\pm$ in $B_{\rho,3,\sigma}(\Omega^\pm)$ (see \eqref{art-decomp_sT_3}). Since $(u^\pm,b^\pm,f)$ and $(\uu^\pm,\ub^\pm,\uf)$ have the same initial data, Lemma 2 of \cite{BG77} provides the following estimate:
\begin{align}\label{art-estim_dP^pm_rho,3,sigma}
\forall \, t\in\big[ 0, a(\rho_0-\rho)\big), \quad
\big\| \p_3 P^\pm(t) \big\| \, \leq \, C_0 \, a \, \nt{\dot{P}^\pm}_{a,\Omega^\pm} \, \frac{1}{\rho_0 - \rho} \, \Big( 1 - \frac{t}{a(\rho_0-\rho)} \Big)^{-\frac{1}{2}}.
\end{align}
Therefore, the key point is to estimate $\dot{P}^\pm$ in $B_{\rho,3,\sigma}(\Omega^\pm)$. The method is identical to the one we use in the proof of Proposition \ref{art-prop_espace_invariant} when we had to estimate $\| \dot{Q}^\pm \|_{\rho,3,\sigma}^\pm$ (see system \eqref{s2-pb_ellip_dotQ} and estimate \eqref{art-estim_dotQ_rho,3,sigma}). Thus, we get the same estimate as \eqref{art-estim_dotQ_rho,3,sigma} but with $P$ (resp. $\bF^\pm$, $\bG$) instead of $Q$ (resp. $\cF^\pm$, $\cG$). In order to conclude about this estimate, it remains to use definitions \eqref{s2-def_termes_sources_bF_bG} and to estimate each term using preliminaries in Paragraph \ref{s2-sec_prelim_estim}. The slight difference here is that we have to exhibit the differences $u^\pm - \uu^\pm$ and so on, as we did in \mbox{\textbf{Step 1}} and \mbox{\textbf{Step 2}} above. These computations generate (very) numerous terms to estimate; nevertheless, there is no additional difficulties compared to our previous work so we shall skip these details and give the final estimate on $\dot{P}^\pm$:
 \begin{align}\label{art-estim_dP^pm_nt}
\nt{\dot{P}^\pm}_{a,\Omega^\pm} \, \leq \, C(\eta_0,R) \, \Big( \nt{\dot{u}^\pm \, - \, \dot{\uu}^\pm}_{a,\Omega^\pm} \, + \, \nt{\dot{b}^\pm \, - \, \dot{\ub}^\pm}_{a,\Omega^\pm} \, + \, \nt{\ddot{f} \, - \, \ddot{\uf}}_{a,\bT^2} \Big),
\end{align}
where $C(\eta_0,R) > 0$ depends only on $\eta_0$ and $R$. Plugging \eqref{art-estim_dP^pm_nt} into \eqref{art-estim_dP^pm_rho,3,sigma}, we can conclude that for all $t\in\big[ 0, a(\rho_0-\rho)\big)$,
\begin{equation}\label{art-estim_dP^pm_rho,3,sigma_final}
\begin{aligned}
\big\| \p_3 P^\pm(t) \big\|_{\rho,3,\sigma}^\pm \, \leq \, C(\eta_0,R) \,  \frac{a}{\rho_0 - \rho} \, & \Big( 1 - \frac{t}{a(\rho_0-\rho)} \Big)^{-\frac{1}{2}} \\
& \times \Big( \nt{\dot{u}^\pm \, - \, \dot{\uu}^\pm}_{a,\Omega^\pm} \, + \, \nt{\dot{b}^\pm \, - \, \dot{\ub}^\pm}_{a,\Omega^\pm} \, + \, \nt{\ddot{f} \, - \, \ddot{\uf}}_{a,\bT^2} \Big).
\end{aligned}
\end{equation}
Let us achieve the case of $\sT_{31}$; thanks to \eqref{art-estim_dpsi_rho,3,sigma} and \eqref{art-estim_dP^pm_rho,3,sigma_final} we finally have:
\begin{equation}\label{art-estim_sT_31}
\begin{aligned}
\forall \, t\in\big[ 0, a(\rho_0-\rho)\big), \quad 
\big\| \sT_{31} \big\|_{\rho,3,\sigma}^\pm \, \leq \, C(\eta_0,R) \, & \frac{a}{\rho_0 - \rho} \, \Big( 1 - \frac{t}{a(\rho_0-\rho)} \Big)^{-\frac{1}{2}} \\
& \times \Big( \nt{\dot{u}^\pm \, - \, \dot{\uu}^\pm}_{a,\Omega^\pm} \, + \, \nt{\dot{b}^\pm \, - \, \dot{\ub}^\pm}_{a,\Omega^\pm} \, + \, \nt{\ddot{f} \, - \, \ddot{\uf}}_{a,\bT^2} \Big).
\end{aligned}
\end{equation}

\medskip

\noindent
\textbf{\begin{small}$\blacklozenge$\end{small} Sub-step 3.2: estimate of $\sT_{32}$.} The case of $\sT_{32}$ is straightforward. Indeed, using \eqref{art-estim_dQ_rho,3,sigma} together with \eqref{art-estim_gradf-graduf_rho,5/2} we end up with the same estimate as \eqref{art-estim_sT_31} for $\sT_{32}$.

\medskip

\noindent
\textbf{\begin{small}$\blacktriangleright$\end{small} Conclusion.} Gathering \eqref{art-estim_sT_1}, \eqref{art-estim_sT_2}, \eqref{art-estim_sT_31} and recalling definition \eqref{s2-def_nt_a_Omega}, we obtain the desired estimate \eqref{art-estim_prop_Upsilon_contractante}: Proposition \ref{art-prop_Upsilon_contractante} is eventually proved.
\end{proof}

\section{Conclusion and perspectives}

The analytic solution we obtain is defined at least until the time $a\rho_0 >0$. However, the quantity $a>0$ depends on $\eta_0$ and $R$, that is on the size of the initial data in the analytic scale. Consequently, the lifespan can become smaller as the initial data get bigger in the analytic spaces $B_{\rho_0,3,\sigma}(\Omega^\pm)$ and $B_{\rho_0,\frac{7}{2}}(\bT^2)$.

The \textit{a priori} estimate established in \cite{CMST} would be a starting point to prove, by a compactness argument, the existence and uniqueness of solutions to the current-vortex sheets problem, taking an initial data in some Sobolev spaces (typically $H^3(\Omega^\pm)$ for the velocity $u_0^\pm$ and the magnetic field $b_0^\pm$, and $H^\frac{7}{2}(\bT^2)$ for the front $f_0$).

To do so, we would approximate the Sobolev initial data $(u_0^\pm,b_0^\pm,f_0)$ by a sequence of analytic initial data  $\big(u_0^{\pm,n},b_0^{\pm,n},f_0^n\big)_{n\geq 0}$, using the density of the spaces $B_{\rho,r,\sigma}(\Omega^\pm)$ in $H^r(\Omega^\pm)$ (we refer to Remark \ref{art-rmq_density}). The unique solution $\big(u^{\pm, n},b^{\pm, n},Q^{\pm, n},f^n \big)$ associated with this initial datum  thus has a radius of analyticity $\rho^n = \rho^n(t)$, that possibly tends to 0 as $n$ goes to $+\infty$. To address such an issue, it would suffice to exhibit a lower bound on $\rho^n(t)$, that depends only on a Sobolev norm of the solution. Therefore, we could propagate the analyticity of the solutions to a time interval depending \textbf{only} on the Sobolev norm of the solutions. The latter does not blow up as $n$ goes to $+\infty$, and would allow to get a positive lower bound on the lifespan of the solutions.

For results about propagation of analyticity of solutions to the incompressible Euler equations, we can refer to \cite{BB,AM} and more recently to \cite{LO,KV-3,KV-2,KV-1}.

\medskip

\bibliographystyle{alpha}
\bibliography{biblio}
\end{document}